\documentclass[12pt,leqno]{amsart}
        \title{The Farrell-Jones Conjecture for mapping class groups}

       \author{Arthur Bartels}
      \address{WWU M\"unster\\
               Mathematisches Institut\\
               Einsteinstr.~62,
               D-48149 M\"unster, Germany}
        \email{a.bartels@wwu.de}
      \urladdr{http://www.math.uni-muenster.de/u/bartelsa} 
       \author{Mladen Bestvina}
      \address{Department of Mathematics\\
               University of Utah\\
               155 South 1400 East, JWB 233\\
               Salt Lake City, Utah 84112-0090 }
        \email{bestvina@math.utha.edu}
        \urladdr{http://www.math.utah.edu/$\sim$bestvina/} 
         \date{October, 2018}

  \usepackage{comment}
  \usepackage{hyperref}
  \usepackage{calc}
  \usepackage{enumerate,amssymb}
  \usepackage[shortlabels]{enumitem} \setitemize{leftmargin=3.5ex}
  \usepackage[arrow,curve,matrix,tips,2cell]{xy}
    \SelectTips{eu}{10} \UseTips
    \UseAllTwocells
  \usepackage{tikz}
  \usepackage{pdfcolmk}
  \usepackage{wasysym}
  \usepackage{float}

\usepackage{etoolbox}

  \DeclareMathAlphabet{\matheurm}{U}{eur}{m}{n}

  \newcommand{\IH}{\mathbb{H}}

  \newcommand{\IN}{\mathbb{N}}

  \newcommand{\IR}{\mathbb{R}}

  \newcommand{\IZ}{\mathbb{Z}}
  \newcommand{\Z}{{\mathbb Z}}
  \newcommand{\R}{{\mathbb R}}

  \newcommand{\EF}{\matheurm{F}}

  \newcommand{\ER}{\matheurm{R}}

  \newcommand{\cala}{\mathcal{A}}
  
  \newcommand{\calc}{\mathcal{C}}

  \newcommand{\calf}{\mathcal{F}}
  \newcommand{\calg}{\mathcal{G}}

  \newcommand{\calm}{\mathcal{M}}

  \newcommand{\calp}{\mathcal{P}}

  \newcommand{\cals}{\mathcal{S}}
  \newcommand{\calt}{\mathcal{T}}
  \newcommand{\calu}{\mathcal{U}}
  \newcommand{\calv}{\mathcal{V}}

  \newcommand{\caly}{\mathcal{Y}}

  \newcommand{\bfK}{{\mathbf K}}
  \newcommand{\bfL}{{\mathbf L}}

  \newcommand{\bfS}{{\mathbf S}}

  \newcommand{\bfY}{{\mathbf Y}}

  \theoremstyle{plain}
  \newtheorem{theorem}{Theorem}[section]
  \newtheorem{thm}[theorem]{Theorem}
  \newtheorem{lemma}[theorem]{Lemma}
  \newtheorem{corollary}[theorem]{Corollary}
  \newtheorem{proposition}[theorem]{Proposition}
  \newtheorem{prop}[theorem]{Proposition}

  \newtheorem{ABCtheorem}{Theorem}

  \theoremstyle{definition}
  \newtheorem{definition}[theorem]{Definition}

  \newtheorem{axiom}[theorem]{Axiom}
  \newtheorem*{definition*}{Definition}

  \theoremstyle{remark}
  \newtheorem{remark}[theorem]{Remark}  
    
  \newtheorem{example}[theorem]{Example}

  \makeatletter\let\c@equation=\c@theorem\makeatother

   {\end{list}}

  \DeclareMathOperator{\diam}{diam}
  
  \DeclareMathOperator{\Fin}{Fin}
  
  \DeclareMathOperator{\id}{id}
  
  \DeclareMathOperator{\Image}{Im}

  \DeclareMathOperator{\map}{map}
  \DeclareMathOperator{\MCG}{Mod(\Sigma)}

  \DeclareMathOperator{\Prob}{Prob}
  
  \DeclareMathOperator{\pt}{pt}

  \DeclareMathOperator{\SL}{SL}

  \DeclareMathOperator{\GL}{GL}
  
  \DeclareMathOperator{\Wh}{Wh}
  \DeclareMathOperator{\VCyc}{VCyc}

  \newcommand{\abelian}{{\mathit{ab}}}
  
  \newcommand{\ANR}{{\mathit{ANR}}}
  
  \newcommand{\CAT}{\mathit{CAT}}
  \newcommand{\CF}{\mathit{CF}}
  
  \newcommand{\dd}{{\partial}}
  \newcommand{\e}{{\varepsilon}}
  \newcommand{\EAC}{{\matheurm{AC}}}
  
  \newcommand{\Eac}{{\matheurm{ac}}}
  \renewcommand{\ER}{{\mathit{ER}}}
  \newcommand{\EVNil}{{\matheurm{VNil}}}
  \newcommand{\EVSol}{{\matheurm{VSol}}}

  \newcommand{\FS}{{\mathit{FS}}}

  \newcommand{\longthin}{{\mathit{long}}}
  \newcommand{\MF}{{\mathit{\calm \hspace{-.1ex} \calf}}}

  \newcommand{\PMF}{{\mathit{\calp \hspace{-.6ex} \calm \hspace{-.1ex} \calf}}}

  \newcommand{\smint}[1]{\mbox{$\int_{#1}$}}
  
  \newcommand{\thin}{{\mathit{thin}}}
  \newcommand{\thick}{{\mathit{thick}}}

  \newcommand{\x}{{\times}}

\begin{document}
\begin{abstract}We prove the Farrell-Jones Conjecture for mapping class groups. 
  The proof uses the Masur-Minsky theory of the large
  scale geometry of mapping class groups and the geometry of the thick
  part of Teichm\"uller space. The
  proof is presented in an axiomatic setup, extending the projection
  axioms in \cite{bbf}. More specifically, we prove that the action of
  $\MCG$ on the Thurston compactification of Teichm\"uller space is
  finitely $\mathcal F$-amenable for the family $\mathcal F$
  consisting of virtual point stabilizers.
\end{abstract}

\maketitle
\tableofcontents

\section*{Introduction}

\noindent
The goal of this paper is to prove the following theorem.
 
\begin{ABCtheorem} \label{thm:FJC-for-MCG}
  The mapping class group $\MCG$ of any oriented surface $\Sigma$ of finite type satisfies the Farrell-Jones Conjecture.
\end{ABCtheorem}
  
We will review the Farrell-Jones Conjecture and its motivation later in the introduction.
The main step in our proof of Theorem~\ref{thm:FJC-for-MCG} is the verification of a regularity condition, called finite $\calf$-amenability in~\cite{Bartels-Coarse-flow}, for the action of $\MCG$ on the space $\PMF$ of projective measured foliations on $\Sigma$, see Theorem~\ref{thm:finitely-amenable-action-MCG} below.
Theorem~\ref{thm:FJC-for-MCG} is then a consequence of the axiomatic results of L\"uck, Reich and the first author for the Farrell-Jones Conjecture from~\cite{Bartels-Lueck-Borel, Bartels-Lueck-Reich-K-FJ-hyp} and an induction on the complexity of the surface.
A similar induction has been used for $\GL_n(\IZ)$~\cite{Bartels-Lueck-Reich-Rueping-GLnZ}. 

\subsubsection*{Topologically amenable actions}
Let $G$ be a group.
The space $\Prob(G)$ of probabilty measures on $G$ is a subspace of $\ell^1(G)$.
An action of $G$ on a compact space $\Delta$ is said to be \emph{topologically amenable} if there exists a sequence of $\text{weak}^*$ continuous maps $f_n \colon \Delta \to \Prob(G)$ satisfying the following condition: for any $g \in G$ 
\begin{equation*}
	\sup_{x \in \Delta} \| f_n(gx) - g f_n(x) \|_1 \to 0  \; \text{as} \; n \to \infty.
\end{equation*}
It will be convenient to refer to sequences of maps satisfying this condition as \emph{almost equivariant maps}.
Groups that admit a topologically amenable action on a compact space are often said to be \emph{boundary amenable}.
This condition is known to imply the Novikov conjecture~\cite{Higson-bivariant+Novikov}.
Since the Stone-\u{C}ech compactification $\beta G$ maps to any compact $G$-space, it follows that a group is boundary amenable if and only if its action on $\beta G$ is topologically amenable.
Hamenst\"adt proved that the mapping class group of a surface $\Sigma$ is boundary amenable by proving that its action on the space of complete geodesic laminations on $\Sigma$ is topologically amenable~\cite{Hamenstaedt-boundary-amenable}.
Kida showed that the action of the mapping class group on its Stone-\u{C}ech compactification is topologically amenable~\cite[App.~C]{Kida-MCG-measure-equiv-Memoirs}.

\subsubsection*{Finite asymptotic dimension}
For any $N$ the subspace $\Prob(G)^{(N)}$ of probability measures supported on sets of cardinality at most $N+1$ is naturally a simplicial complex of dimension $N$.
The isotropy groups for this action all belong to the family $\Fin$ of finite subgroups.
For an amenable action of $G$ on a compact space $\Delta$ one can ask whether the almost equivariant maps $f_n$ from above can be chosen to have image in $\Prob(G)^{(N)}$ for some $N$ independent of $n$.
For $\Delta = \beta G$ such maps $f_n$ exist if and only if the asymptotic dimension of $G$ is $\leq N$~\cite[Thm.~6.5]{Guentner-Willett-Yu-dyn-asy-dim}.
Together with Bromberg and Fujiwara the second author proved that the mapping class group has finite asymptotic dimension~\cite{bbf}. 
This result implies a stronger form of the Novikov conjecture~\cite{Kasprowski-On-FDC}, often called the integral Novikov conjecture, i.e., the integral injectivity of the assembly maps of the assembly maps in algebraic $K$-theory and $L$-theory relative to the family of finite subgroups. 
These assembly maps are briefly reviewed in Subsection~\ref{subsec:FJC}.

\subsubsection*{Finite $\calf$-amenability} 
The axiomatic condition from~\cite{Bartels-Lueck-Reich-K-FJ-hyp} that we will use requires an action on a space much nicer than the Stone-\u{C}ech compactification.
For the mapping class group this will be the Thurston compactification of Teichm\"uller space, i.e., a disk.
Technically, the requirement is that the space is a Euclidean retract. 
For actions on such spaces there are typically infinite isotropy groups and this obstructs the existence of almost equivariant maps $f_n$ into a finite dimensional simplicial complex with a proper simplicial $G$-action. 
The condition from~\cite{Bartels-Lueck-Reich-K-FJ-hyp} is therefore formulated relative to a family $\calf$ of subgroups; the action on the target simplicial complex is then allowed to have isotropy in this family.
Maps to a finite dimensional simplicial space translate to finite dimensional covers of the source, as we can pull back standard coverings of the simplicial complex.
This translation is used in the formulation of the regularity
condition that we recall now in detail.

Let $\calf$ be a \emph{family of subgroups} of a group $G$, i.e., $\calf$ is set of subgroups of $G$ that is closed under conjugation and taking subgroups. 
A subset $U$ of a $G$-space is said to be an \emph{$\calf$-subset} if there is $F \in \calf$ such that $gU = U$ for all $g \in F$ and $gU \cap U = \emptyset$ if $g \not\in F$. 
A cover $\calu$ of a $G$-space is said to be an \emph{$\calf$-cover} if $\calu$ is $G$-invariant and consists of $\calf$-subsets. 
If all members of $\calu$ are in addition open, then we say that $\calu$ is an open $\calf$-cover.
For $N \in \IN$ an action of $G$ on a space $\Delta$ is said to be \emph{$N$-$\calf$-amenable} if for all finite subsets $S$ of $G$ there exists an open $\calf$-cover $\calu$  of $G \x \Delta$ (equipped with the diagonal $G$-action $g \cdot (h,x) = (gh,gx)$) such that
\begin{itemize}
     \item the order of $\calu$ is at most $N$;
     \item the cover $\calu$ is \emph{$S$-long (in the group coordinate)},
        i.e., for every $(g,x) \in G \x \Delta$ there is $U \in \calu$ with
        $gS \x \{ x \} \subseteq U$.
\end{itemize}
An action that is $N$-$\calf$-amenable for some $N$ is said to be
\emph{finitely $\calf$-amenable}.

Given such a cover one obtains almost equivariant maps $f_n$ from  $\Delta$ to simplicial complexes of dimension $\leq N$ as follows.
For each $\calu$ a partition of unity subordinate to $\calu$ provides a $G$-equivariant map $f_\calu$ from $G \x \Delta$ to the Nerve of $\calu$. 
By the first condition this nerve is of dimension at most $N$.
The second condition translates into the almost equivariance of the restrictions of the $f_\calu$ to $\{ e \} \x \Delta$.

It is straightforward to check that the action of the mapping class group on Teichm\"uller space is finitely $\Fin$-amenable.
The key point for us is to understand the action on the boundary of Teichm\"uller space, i.e., on the space $\PMF$ of projective measured foliations.

\begin{ABCtheorem} \label{thm:finitely-amenable-action-MCG} 
    Let $\Sigma$ be a closed oriented surface of genus $g$ with $p$ punctures where 
    $6g + 2p - 6 > 0$.     
    Let $\calf$ be the family of subgroups that virtually fix an essential simple closed curve
    (up to isotopy) or are virtually cyclic. 
    Then the action of its mapping class group $\MCG$ on the
    space $\PMF$ of projective measured foliations on $\Sigma$
    is finitely $\calf$-amenable.   
\end{ABCtheorem}  

The family $\calf$ appearing in Theorem~\ref{thm:finitely-amenable-action-MCG} can
alternatively be described as the family of subgroups that virtually fix a point in $\PMF$: Firstly, every curve determines a point in $\PMF$. 
Secondly, every cyclic subgroup fixes a point in the Thurston compactification $\calt$ of Teichm\"uller space by the Brouwer fixed point theorem; since the action on Teichm\"uller space is proper, every infinite cyclic subgroup fixes a point in $\PMF$.      
On the other hand, for any finitely $\calf$-amenable action on any space all isotropy groups of the action necessarily belong to $\calf$.
Thus the family $\calf$ above is, up to finite index subgroups, the smallest family for which the action of the mapping class group on $\PMF$ can be finitely $\calf$-amenable. 

\subsubsection*{Surface groups} 
To motivate our proof of Theorem~\ref{thm:finitely-amenable-action-MCG} we recall the model argument, due to Farrell-Jones~\cite{Farrell-Jones-K-th-dynamics-I}.
We later also discuss the situation of $\SL_2(\Z)$, where the action
is not cocompact.
Let $\Sigma$ be a closed hyperbolic surface. Thus $G=\pi_1(\Sigma)$
acts on the universal cover $\widetilde\Sigma=\mathbb H^2$ and on the
circle at infinity $\Delta$. We sketch a proof that the action of $G$
on $\Delta$ is finitely $\mathcal F$-amenable for the family $\mathcal
F$ consisting of cyclic subgroups. There are two steps in the proof,
{\it long thin covers} and a {\it geodesic flow argument}.

Let $M=T_1\Sigma$ be the unit tangent bundle of $\Sigma$. This is a
closed 3-manifold equipped with a 1-dimensional foliation by the orbits
of the geodesic flow. Thus $v,w\in T_1\Sigma$ are in the same leaf if
and only if there is a geodesic line in $\Sigma$ tangent to both $v$
and $w$. For any $R$ there are only finitely many closed leaves of
length $\leq R$. 

{\it Step 1: Long thin covers.} For every $\epsilon>0$ and $R>0$
construct an open cover $\mathcal U$ of $M$ such that: 
\begin{itemize}
\item every leaf segment of length $\leq 2R$ is contained in some
  $U\in\mathcal U$,
\item every $U\in \mathcal U$ is contained in the
  $\epsilon$-neighborhood of some leaf segment of length $\leq 9R$,
\item the multiplicity of the cover is bounded above independently of
  $R$ and $\epsilon$.
\end{itemize}
The elements of the cover are going to be small neighborhoods of the
closed leaves of length $\leq R$, and otherwise they will be flow
boxes for the foliation of the form (leaf segment)$\times$(small cross
section). Care must be taken to arrange the third bullet.

Lifting this cover to $T_1\mathbb H^2$ produces an open cover
$\widetilde{\mathcal U}$. This is
going to be an $\mathcal
F$-cover if $\epsilon$ is sufficiently small so that
$\pi_1(U)\to\pi_1(M)$ has cyclic image for every $U$ (nontrivial for
neighborhoods of closed leaves, and otherwise trivial). The number $R$
will depend on the given finite set $S\subset G$ and then $\epsilon$
depends on $R$.

{\it Step 2: Geodesic flow argument.} 
First define the {\it flow
  space}
$$\FS=\{(x,p,\xi)\in \overline {\mathbb H^2}\times \mathbb
H^2\times\Delta\mid p\in [x,\xi)\}$$
where $[x,\xi)$ is the geodesic joining $x$ to $\xi$. 
There is an embedding $T_1\mathbb H^2\to \FS$ onto a closed subset
defined by $v\mapsto (\xi_-,p,\xi_+)$, where $p\in\mathbb H^2$ is the
point where $v$ is based, and $\xi_\pm\in\Delta$ are the points at
$\pm\infty$ of the geodesic tangent to $v$.
The construction of the $\mathcal F$-cover $\widetilde{\mathcal
  U}$ in the first step can also be applied to $\FS$.
(Alternatively, it suffices to extend the cover of $T_1 \IH^2$ such that it covers a neighborhood of $T_1\mathbb H^2$ in $\FS$
preserving the multiplicity, see Lemma
\ref{lem:extension-of-covers}). 
Consider for any $\tau \geq 0$ the map
$$\iota_\tau \colon G \x \Delta \to \FS$$ 
defined  as follows. 
First identify $G$ with an orbit in $\mathbb H^2$, thus
$G\subset \mathbb H^2$. Then let 
$$\iota_\tau(g,\xi)=(g,x_\tau,\xi)$$
where $x_\tau$ is the unique point on the geodesic ray $[g,\xi)$ at distance $\tau$ from $g$.
Thus $\iota_\tau$ is the map ``flow for time $\tau$''. 
Now we argue that for $\tau$
sufficiently large the cover $\iota_\tau^{-1}(\widetilde{\mathcal U})$
satisfies the requirements. This is accomplished by a geometric limit
argument: if the statement is false then for every $\tau$ we have
$(g_\tau,\xi_\tau)$ so that $\iota_\tau(B_R(g_\tau)\times
\{\xi_\tau\})$ is not contained in any $\widetilde U$. By equivariance
we may assume that the middle coordinate of $\iota_\tau(g_\tau)$
belongs to a fixed compact set $K$ (this uses cocompactness of the
action) and passing to the limit as $\tau\to\infty$ produces
$(\xi_-,p,\xi_+)\in T_1\mathbb H^2$ not contained in any $\widetilde
U$, contradiction. The key point here is that an $R$-ball gets
squeezed by the geodesic flow to a long thin set. 

\subsubsection*{The group $\SL_2(\Z)$}
Let $G$ be a torsion-free subgroup of finite index of $\SL_2(\Z)$.
Here we sketch a proof (close to what we do for mapping class groups)
that the standard action of $G$ on the circle at infinity $\Delta$ is
finitely $\mathcal F$-amenable for the family $\mathcal F$ of cyclic
subgroups. The proof for surface groups breaks down since the action
on $\mathbb H^2$ is not cocompact and it is not possible to arrange
that the point $\iota_\tau(g)$ belongs to a fixed compact set. Fix an
equivariant, pairwise disjoint collection of horoballs, so that the
action on the complement is cocompact and identify $G$ with an orbit
outside the horoballs. For a fixed $\Theta>0$ say that the pair
$(g,\xi)\in G\times\Delta$ is $\Theta$-thick if the geodesic ray
$[g,\xi)$ intersects every horoball in a segment of length
  $\leq\Theta$. Then the above argument generalizes to show that for
  any $\Theta>0$ and any $S\subset G$ there is a required cover of the
  $\Theta$-thick part of $G\times\Delta$. But we are still left to
  cover the thin part.

The naive idea is as follows. 
For every $(g,\xi)$ which is not
$\Theta$-thick let $B(g,\xi)$ be the first horoball that the ray
$[g,\xi)$ intersects in a segment of length $>\Theta$. Then for each
  horoball $B$ define the set
$$U(B)=\{(g,\xi)\mid B(g,\xi)=B\}$$
Clearly, the collection $\{U(B)\}$ is equivariant, covers the
$\Theta$-thin part of $G\times\Delta$, consists of
$\mathcal F$-sets, and any two are disjoint, i.e. the multiplicity is
1. The problem is that the sets $U(B)$ are not necessarily open, and
the cover may not be $S$-long. In both cases, the issue is the {\it
  threshold problem}, i.e., small perturbations of $(g,\xi)$ may change
$B(g,\xi)$ dramatically. 

There are two things we do to solve the threshold problem. The first
is to modify how we measure the size of intersection between a ray and
a horoball. This is formalized in the notion of an {\it angle} and its
main feature is that if a ray has long intersections with three
horoballs, perturbing $(g,\xi)$ does not affect the angle at the
middle horoball (see Section \ref{subsec:proj-cx}). To define the
angle we use the language of projection complexes \cite{bbf}, but in
the present case one could use the Farey graph, whose vertices
correspond to the horoballs.
This leaves the
threshold problem only in the case of the first large intersection,
and we solve this by working at two (or technically six) different
scales, see Section \ref{subsec:numbers}.

There is an alternative argument for $\SL_2(\IZ)$ where the failure of cocompactness is addressed on the flow space~\cite{Bartels-Lueck-Reich-Rueping-GLnZ}. 
This alternative argument seems not to be applicable to the mapping class group, for example because we have only control over the behavior of Teichm\"uller rays that stay in a thick part.

\subsubsection*{Sketch of proof of Theorem~\ref{thm:finitely-amenable-action-MCG}}
Our argument follows the model arguments sketched above,
but with several important differences. First, our geodesic flow
argument, modelled on \cite{Bartels-Coarse-flow}, is {\it coarse},
i.e. it squeezes balls to sets with bounded, rather than
$\epsilon$-small, cross-section. The flow space $\FS$ is replaced by
the coarse flow space $\CF$.

Second, instead of measuring lengths
of intersections with horoballs, we use the Masur-Minsky notion of
subsurface projection.

The hyperbolic plane is replaced with the Teichm\"uller space $\calt$ of
complete hyperbolic structures of finite area on $\Sigma\smallsetminus
P$. It is
equipped with the Teichm\"uller metric, which is invariant under the
action of $\MCG$, and it is compactified by $\PMF$. Fix a basepoint
$X_0\in\calt$ and identify $\MCG$ with the orbit of $X_0$. Given a
pair $(g,\xi)\in \MCG\times\PMF$ there is a unique Teichm\"uller ray
$c_{g,\xi}$ that starts at $g(X_0)$ and is ``pointing towards $\xi$''
(technically, the vertical foliation of the quadratic differential is
$\xi$). The construction of the required cover of $\MCG\times\PMF$ is
divided into two parts. For $\epsilon>0$ the pair $(g,\xi)$ is
$\epsilon$-thick if no geodesic along $c_{g,\xi}$ has length
$<\epsilon$, and otherwise the pair is $\epsilon$-thin. 

Given a finite set $S\subset\MCG$ we first find $\epsilon>0$ and an
$S$-long cover of the $\epsilon$-thin part. Then we cover the
$\epsilon$-thick part.

When covering the thin part the main tool is the Masur-Minsky notion
of subsurface projections. As in \cite{bbf}, the collection of all
subsurfaces is divided into finitely many subcollections $\bfY^i$,
$i=1,2,\cdots,k$ so that any two subsurfaces in the same subcollection
overlap. There is a finite index subgroup $G<\MCG$ that preserves each
subcollection. Our cover of the thin part will consist of $k$
collections of sets, one for each $\bfY^i$.  Roughly speaking, a
theorem of Rafi~\cite{rafi-short} guarantees that if $(g,\xi)$ is in
the thin part then for some subsurface $Y$ we have a large projection
distance in $Y$ between $g(X_0)$ and $\xi$, and the elements of the
cover will be parametrized by such subsurfaces.  Here we use for each
$\bfY^i$ the projection complex~\cite{bbf} (see also \cite{bbfs} for a
streamlined construction) to obtain a good notion of \emph{first}
subsurface with large projection for a given pair $(g,\xi)$, similar
to the first horoball in the model case of $\SL_2(\IZ)$ above.  To
this end we find it useful to extend the axiomatic setup for the
projection complex to include projections of foliations $\xi\in\PMF$.

The strategy for covering the thick part is a coarse version of the
model arguments recalled earlier.  Here we use hyperbolicity
properties of the thick part of Teichm\"uller space. We summarized
the properties we use in the form of flow axioms.  These axioms are
formulated in a quasi-isometry invariant way.  We note that our fellow
traveler axiom~\ref{axiom:flow:fellow-travel} is weaker than what is
actually known now about Teichm\"uller geodesics in the thick part,
see for example~\cite[\S 8]{Eskin-Mirzakhani-Rafi-Counting-in-Orbits}.

\subsubsection*{Outline by sections}
We start by listing all the axioms in Section \ref{s:axioms}. 
In Section \ref{sec:proj-covers} we construct the cover of the thin part
using the projection axioms and in Section \ref{sec:cover-thick} we
construct the cover of the thick part using the flow axioms. 
Section \ref{s:fac} contains a general discussion of the Farrell-Jones
Conjecture and finite $\mathcal F$-amenability. In Section \ref{s:mcg}
we review the basics of mapping class groups: Teichm\"uller space,
measured foliations and geodesic laminations, Teichm\"uller metric,
quadratic differentials. In Section \ref{s:proj} we review the Masur-Minsky
subsurface projections and verify the projection axioms. 
In Section \ref{s:2} we verify the flow axioms. The main ingredients
are Minsky's Contraction Theorem \cite{minsky-contracting}, the Masur
Criterion \cite{masur_criterion}, and Klarreich's description of the
boundary of the curve complex \cite{klarreich}.
In Section \ref{sec:teich-geod-in-thin} we discuss Teichm\"uller geodesics that
enter the thin part and in Section \ref{sec:FJC-for-MCGs} we provide
the roadmap showing which axioms are verified where.

\subsubsection*{The Farrell-Jones Conjecture}
   Surgery theory as developed by Brow\-der-Novikov-Sullivan-Wall relates the classification of manifolds (of dimension at least $5$) to algebraic invariants, i.e., to the algebraic $K$-groups and $L$-groups of $\IZ[G]$, the integral group ring over the fundamental group of the manifold.
   Farrell and Jones~\cite{Farrell-Jones-Isom-Conj} formulated a general conjecture about the structure of these $K$- and $L$-groups.
   Informally, the conjecture asserts that the computation of $K$- and $L$-groups of $\IZ[G]$ can be reduced (modulo group homology) to the computation of $K$- and $L$-groups of group rings $\IZ[V]$, where $V$ varies over the family $\VCyc$ of virtually cyclic subgroups of $G$.
   Often it is beneficial to consider as an intermediate step a family of subgroups $\calf$ 
   containing $\VCyc$ and to consider 
   the Farrell-Jones Conjecture relative to $\calf$.
   The formulation of the Farrell-Jones Conjecture that we will be using is recalled in Subsection~\ref{subsec:FJC}.
   
   Building on the work of Farrell and Jones their Conjecture has been
   verified for many classes of groups. Among these are hyperbolic
   groups, $\CAT(0)$-groups, solvable groups and lattices in Lie
   groups~\cite{Bartels-Lueck-Borel, Bartels-Lueck-Reich-K-FJ-hyp,
     Kammeyer-Lueck-Rueping-lattices, Wegner-2012CAT0,
     Wegner-FJ-solv}. 
   A more complete list can be found in~\cite[Thm.~2]{Kammeyer-Lueck-Rueping-lattices}. 
   
   The Farrell-Jones Conjecture implies the Novikov Conjecture, but is stronger.
   Roughly speaking, the Novikov Conjecture asserts that a certain map is (rationally) injective whereas the Farrell-Jones Conjecture asserts it is bijective.
   More concretely, as reviewed below, the Farrell-Jones Conjecture implies that aspherical manifolds are determined up to homeomorphism by their fundamental group (in dimension $\geq 5$), while the Novikov Conjecture only asserts that their higher signatures are determined by the fundamental group.

\subsubsection*{The surgery exact sequence} 
Applications to the classification of manifolds was the main motivation for the Conjecture and the monumental works surrounding it of Farrell and Jones.
We give a short summary of the key instance of such an applications. 
In the following manifolds are always topological manifolds, i.e., we
do not require a smooth or $\mathit{PL}$-structure; the topological
implications of the Farrell-Jones Conjecture are cleanest. 
Let $M$ be a closed oriented manifold of dimension $n \geq 5$. 
The simple topological structure set $\cals(M)$ of $M$ consists of
equivalence classes $[f \colon X \to M]$, where $X$ is a closed
topological manifold and $f$ is a simple homotopy equivalence.
We have $[f \colon X \to M] = [f' \colon X' \to M]$ if and only if
there is a homeomorphism $h \colon X \to X'$ with $f' \circ h$
homotopic to $f$.  
Thus understanding the structure set amounts to classifying all manifolds in the simple homotopy type of $M$.
To understand the structure set surgery theory provides the surgery exact sequence, which we outline next.
Let $\bfL$ be the $L$-theory spectrum of the ring $\IZ$ and let $\bfL\langle 1 \rangle$ be its connective cover.
In particular, $\pi_n(\bfL \langle 1 \rangle) = \pi_n(\bfL) =
L_n(\IZ)$ for $n \geq 1$ and $\pi_n(\bfL \langle 1 \rangle) = 0$ for
$n \leq 0$.
We write $L^s_*(\IZ[G])$ for the simple $L$-groups of the integral group ring over the fundamental group of $M$. 
The \emph{surgery exact sequence for topological manifolds} can be formulated as follows
\begin{equation*} \label{eq:surgery-exact-seq}
	H_{n+1}(M;\bfL) \xrightarrow{\alpha_M} L^s_{n+1}(\IZ[G]) \to \cals(M) \to H_{n}(M;\bfL \langle 1 \rangle) \xrightarrow{\alpha'_M}  L^s_n(\IZ[G]). 
\end{equation*}
Here $H_*(-;\bfL \langle 1 \rangle)$ and $H_*(-;\bfL)$ are  the homology theories associated to the corresponding spectra, $\alpha_M$ is Quinn's assembly map for $M$ and $\alpha'_M$ is the composition of $\alpha_M$ with the map induced from $\bfL\langle 1 \rangle \to \bfL$.
Exactness of this sequence includes the statement that the structure set $\cals(M)$ carries a natural group structure (which is not easily described). 
There is an extension of this sequence to a long exact sequence.
The construction and exactness of the surgery exact sequence for topological manifolds depends among other things on the work of Kirby-Siebenmann on topological manifolds and Ranicki's identification of the geometric assembly map with the algebraic assembly map.
For a more detailed summary of these results see the introduction of~\cite{Ranicki-blue-book}.

If $G$ is trivial, then $L^s_*(\IZ[G]) = L_*(\IZ) = H_*(\pt;\bfL)$ and $\alpha_M$ is the map induced by the projection $M \to \pt$ for $H_*(-;\bfL)$.
Consequently, $\alpha_M$ and $\alpha'_M$ are both surjective.
If we specialize further to $M = S^n$, then $H_n(S^n;\bfL \langle 1 \rangle) \cong L_n(\IZ)$ and $\cals_n(S^n) = \{ \id_{S^n} \}$, i.e., we recover the high-dimensional Poincar\'e conjecture from the surgery exact sequence.

The Farrell-Jones Conjecture gives information about the $L$-groups in the surgery exact sequence.
This information has the cleanest form if $G$ is torsion free.
So assume now that $G$ satisfies the Farrell-Jones Conjecture and is torsion free.
Then, after some algebraic manipulations, there is an isomorphism $\alpha_{BG} \colon H_*(BG;\bfL) \to L^s_n(\IZ[G])$, see~\cite[Prop.~23]{Lueck-Reich-BC-FJ-survey}. 
Moreover, if $\kappa \colon M \to BG$ is the classifying map (inducing $\pi_1(M) = G$), then
$\alpha_M = \alpha_{BG} \circ \kappa_*$.
This yields an injection of the structure set $\cals(M)$ into the relative homology group
$H_{n+1}(BG,M;\bfL)$ (where we think of $M$ as a subspace of $BG$ via $\kappa$).
The cokernel of this inclusion is a subgroup of $H_n(M;L_0(\IZ)) = L_0(\IZ) = \IZ$\footnote{In fact, under our assumptions, $H_{n+1}(BG,M;\bfL)$ can be identified with the $\ANR$-homology manifold structure set and the map to $\IZ$ is the Quinn obstruction~\cite[\S25]{Ranicki-blue-book}.}.
Moreover, as reviewed below, the $K$-theory part of the Farrell-Jones conjecture implies that the Whitehead group of $G$ is trivial and therefore every homotopy equivalence $X \to M$ is simple.
Consequently, the Farrell-Jones Conjecture yields an identification of all manifold structures on $M$ with relative homology classes.
Finally, let us assume in addition that $M$ is aspherical, i.e., that $\kappa$ is a homotopy equivalence.
Then of course $H_{n+1}(BG,M;\bfL) = 0$ and therefore $\cals(M) = \{ [\id_M] \}$.
In other words, every homotopy equivalence $X \to M$ is homotopic to a homeomorphism, i.e., the Borel Conjecture holds for $M$.

For the $K$-theory of the integral group ring $\IZ[G]$ of a torsion free group the Farrell-Jones Conjecture predicts that the Loday assembly map $H_n(BG;\bfK_\IZ) \to K_n(\IZ[G])$ is bijective.
Here $H_n(-;\bfK_\IZ)$ is the homology theory associated to the $K$-theory spectrum $\bfK_\IZ$ of the integers $\IZ$.
Since $K_n(\IZ) = 0$ for $n < 0$, $K_0(\IZ) = \IZ$, $K_1(\IZ) = \IZ^\times \cong \IZ/2\IZ$,
the Farrell-Jones conjecture predicts for torsion free groups $K_n(ÍZ[G]) = 0$ for $n < 0$, $K_0(\IZ[G]) = \IZ$ and $K_1(\IZ[G]) = \IZ/2\IZ \times G_\abelian$; in particular, it predicts the vanishing of the reduced projective class group $\tilde K_0(\IZ[G])$ and of the Whitehead group $\Wh(G)$~\cite[Cor.~67]{Lueck-Reich-BC-FJ-survey}. 
The Farrell-Jones Conjecture has applications to a number of further
conjectures about group rings, for example Kaplansky's idempotent
conjecture, Bass' Conjecture about the Hattori-Stallings rank, and
Serre's Conjecture that groups of type FP are of type
FF~\cite{Bartels-Lueck-Reich-FJ+appl,Lueck-ICM2010,Lueck-Reich-BC-FJ-survey}.

\subsection*{Acknowledgements}
Our collaboration started during a workshop at the Hausdorff Institute for Mathematics in Bonn in April 2015.
We thank Wolfgang L\"uck for inviting us to this workshop.

We thank Ken Bromberg, Jon Chaika, Howard Masur, and Kasra Rafi for
many interesting conversations about Teichm\"uller theory. We
especially thank Kasra Rafi for his help with Theorem
\ref{large-proj}.

We thank the referee for a long list with helpful
comments.

The first author is supported by the SFB 878 in M\"unster. The second
author is supported by the NSF under the grant number DMS-1308178.

  \section{Axioms}\label{s:axioms}

  Let $G$ be a group, and let $\Delta$ be a $G$-space.
  We assume that $\Delta$ is compact, metrizable and finite dimensional.
  In this section we discuss somewhat elaborate axioms that will imply
  that the action of $G$ on $\Delta$ is finitely $\calf$-amenable.
  The axioms will be defined in the presence of two pieces of further data.
  
  \begin{definition} \label{def:proj-data}
    \emph{Projection data} for the action of $G$ on $\Delta$ consists of
    \begin{itemize}
      \item a finite collection of $G$-sets $\caly = \{ \bfY^1,\bfY^2,\cdots,\bfY^k \}$;
      \item for each $\bfY\in \caly$ and each $Y \in \bfY$ an open
        subspace $\Delta(Y) \subseteq \Delta$ and a
        map 
        $$d^\pi_Y \colon \big( \bfY \smallsetminus \{ Y \} \big) \x 
               \big(\Delta(Y) \amalg \bfY \smallsetminus \{Y\} \big) \to [0,\infty].$$
    \end{itemize}
    We also require $G$-equivariance, i.e., for $g \in G$, $Y \in \bfY$, 
    we require $\Delta(gY) = g\Delta(Y)$ and for
    $X \in \bfY \smallsetminus \{ Y \},\xi \in \Delta(Y) \amalg \bfY \smallsetminus \{ Y \}$ we require
    $d^\pi_{gY}(gX,g\xi) = d^\pi_{Y}(X,\xi)$.
    We also require $d_Y^\pi(X,Z) < \infty$ for $X,Y,Z \in \bfY$.
    We will refer to the $d^\pi_Y$ as \emph{projection distances}.
  \end{definition}

  \begin{definition} \label{def:flow-data}
    \emph{Flow data} for the action of $G$ on $\Delta$ consists of
    \begin{itemize}
     \item a proper metric space $T$ with a proper isometric $G$-action;
     \item a compact metrizable space $\overline{T} = T \sqcup \Delta$ 
       such that $T \subseteq \overline{T}$ is open and the
       $G$-actions on $T$ and $\Delta$ combine to a continuous $G$-action on $\overline{T}$;
     \item for every compact subset $K \subseteq T$ a 
      collection $\calg_{K}$ of $(\mu,A_K)$-quasi-geodesic rays
       $c \colon [0,\infty) \to G \cdot K \subseteq T$.
    \end{itemize}
    The space $\overline{T}$ is assumed to be finite dimensional, separable and metrizable.    As indicated with the notation, the additive constant $A_K$ for quasi-geodesic rays 
    is uniform for all $c \in \calg_K$, but is allowed to depend on $K$, 
    while the multiplicative constant 
    $\mu$ is required to be independent of $K$ as well. In our
    application we will have $(\mu,A_K)=(1,0)$, i.e. each
    $c\in\calg_K$ will be a geodesic ray.
    We also require that $\calg_K$ is $G$-invariant.
    Finally, we require that each ray $c \in \calg_K$ has a limit $c(t) \to c(\infty) \in \Delta$.
  \end{definition}

\begin{remark}
Examples will follow the axioms, but we point out the
following. First, the projection distance $d^\pi_Y(X,Z)$ will always
be finite when $X,Y,Z$ are distinct elements of $\bfY$. On the other
hand, $d^\pi_Y(X,\xi)$ may be infinite or undefined when $\xi\in
\Delta$. Second, the collection $\calg_K$ should be thought of as
really depending only on $G\cdot K$ and consists of (quasi)geodesic
rays in $T$ that are contained in $G\cdot K$, which is thought of as
the ``thick part''.
\end{remark}

  We now fix base points $X_\bfY \in \bfY$ and a base point $x_0 \in T$.
  The choice of the base points 
  only affects the values of $\Theta$ and $K$ later on, but not
  the validity of the axioms.

  \begin{axiom}[Projections] \label{axiom:projections}
    For each $\bfY \in \caly$ the projection distances $(d^\pi_Y)_{Y \in \bfY}$ satisfy the following 
    projection axioms with respect to some constant $\theta \geq 0$.
    \begin{enumerate}[label=(P\arabic*)]
    \item \emph{Symmetry.} \label{axiom:proj:symmetry}
       For $X,Z \in \bfY \smallsetminus \{Y\}$, 
       $$d^\pi_Y(X,Z)=d^\pi_Y(Z,X).$$
    \item \label{axiom:proj:triangle} \emph{Triangle inequality.}
       For all $X,Z \in \bfY \smallsetminus \{ Y \}, \xi \in \Delta(Y) \amalg \bfY \smallsetminus \{ Y \}$ 
       $$d^\pi_Y(X,Z) + d^\pi_Y(Z,\xi) \geq d^\pi_Y(X,\xi).$$
           \item \label{axiom:proj:triples} \emph{Inequality on triples.}
       For all $\xi \in \Delta(Y) \cap \Delta(Y') \amalg \bfY \smallsetminus \{ Y,Y' \}$, $Y \neq Y'$ we have 
       $$\min\{d^\pi_Y(Y',\xi),d^\pi_{Y'}(Y,\xi)\} < \theta.$$
    \item \label{axiom:proj:finiteness} \emph{Finiteness.} 
       For all $X,Z\in\bfY$ the set $$\big\{Y\in\bfY \setminus \{X,Z\}\mid d^\pi_Y(X,Z)>\theta\big\}$$ is finite.
    \item \label{axiom:proj:coarse-semi-cont} \emph{Coarse semi-continuity.} 
       For $\xi \in \Delta(Y)$, $X \in \bfY \setminus \{ Y \}$, $\theta < \Theta < \infty$ 
     with $d_Y(X,\xi) \geq \Theta$ 
     there exists a neighborhood $U$ of $\xi$ in $\Delta$ such that $U \subseteq \Delta(Y)$ 
     and for all $\xi' \in U$, $$d_Y^\pi(X,\xi') > \Theta - \theta.$$

    \end{enumerate}
  \end{axiom}

  \begin{axiom}[Thick or thin] \label{axiom:large-projection}
    For every $\Theta > 0$ there is $K \subseteq T$ compact such that for any 
    $(g,\xi) \in G \x \Delta$ 
    \begin{enumerate}
      \item \label{axiom:large-proj:thick}
         either, there is $c \in \calg_K$ with $c(0) = gx_0$ and $c(\infty)=\xi$
      \item \label{axiom:large-proj:thin} 
         or, there are $\bfY \in \caly$, $Y \in \bfY$ with $\xi \in \Delta(Y)$ and
         $d^\pi_Y(gX_\bfY,\xi) > \Theta$.
    \end{enumerate} 
  \end{axiom}

  Pairs $(g,\xi)$ to which~\ref{axiom:large-proj:thick} of
  Axiom~\ref{axiom:large-projection} applies will be said to belong to
  the \emph{$K$-thick part of $G \x \Delta$}.  Pairs $(g,\xi)$ to
  which~\ref{axiom:large-proj:thin} of
  Axiom~\ref{axiom:large-projection} applies will be said to
  \emph{admit a $\Theta$-large projection}.

  \begin{axiom}[Flow axioms] \label{axiom:flow}
   Let $K \subseteq T$ be compact.
  \begin{enumerate}[label=(F\arabic*)]
    \item \label{axiom:flow:small-at-infty}
       \emph{Small at $\infty$.} 
       Let $c_n \in \calg_K$, $x_n \in T$ such that 
       \begin{equation*}
         d_T(\Image(c_n),x_0) \; \text{and} \; d_T(c_n(0),x_n)
       \end{equation*}
       are bounded. Then $x_n \to \xi \in \Delta$, if and only if
       $c_n(0) \to \xi$.

    \item \label{axiom:flow:fellow-travel} 
       \emph{Fellow traveling.}  
       For any $\rho > 0$ there is $R > 0$ with the following property.
       For all $x \in T$, all $\xi_+ \in \Delta$, all $t \in [0,\infty)$ 
       there exists an open neighborhood $U_+$ of $\xi_+$ in $\Delta$ with the following property.
       Let $c,c' \in \calg_K$ be two quasi-geodesic rays that both start in the $\rho$-neighborhood
       of $x$ and satisfy $c(\infty),c'(\infty) \in U_+$.
       We require that $d_T(c(t),c'(t)) < R$. 
    \item \label{axiom:flow:infinite}
       \emph{Infinite quasi-geodesics.}  
       For all $\rho > 0$ there is $R > 0$ with the following property.
       For $\xi_-,\xi_+ \in \Delta$ we define $T_{K,\rho}(\xi_-,\xi_+) \subseteq T$ 
       to consist of all $x$
       for which there are $c_n \in \calg_K$ with $c_n(0) \to \xi_-$, $c_n(\infty) \to \xi_+$ and
       $d_T(\Image(c_n),x) \leq \rho$. 
       If $T_{K,\rho}(\xi_-,\xi_+) \neq \emptyset$, then we require that there exists a quasi-geodesic
       $c \colon \IR \to T$ such that $T_{K,\rho}(\xi_-,\xi_+)$ is contained in
       the $R$-neighborhood of $c$.
       Here the additive constant for $c$ depends only on $K$, while the multiplicative 
       constant for $c$ is required to not depend on any choices. 
    \end{enumerate}
  \end{axiom}

  \begin{remark}
    The fellow traveling axiom~\ref{axiom:flow:fellow-travel} implies that for 
    any $\alpha > 0$ there is $R > 0$ such that for all $c,c' \in \calg_K$
    with $d_T(c(0),c'(0)) \leq \alpha$ and $c(\infty)=c'(\infty)$ we have 
    $$\forall t \geq 0 \quad d_T(c(t),c'(t)) \leq R.$$ 
  \end{remark}

\begin{remark}
If $T$ is a proper $\delta$-hyperbolic space and $\Delta$ is its Gromov boundary,
then (F1)-(F3) hold, with $\calg_K$ consisting of all geodesics
contained in the ``thick part'' $G\cdot K$. Moreover, (F1)-(F2) hold
when $T$ is a proper $CAT(0)$ space and $\Delta$ its visual
boundary. However, (F3) may fail, e.g. when $T=\R^2$ and $\xi_{\pm}$
are two antipodal points on the boundary circle. It is possible to
weaken (F3) and only demand that $T_{K,\rho}(\xi_-,\xi_+)$ has a
doubling property, see Proposition \ref{prop:CFK}. The collection
$\mathcal F$ in Theorem~\ref{thm:axioms->-fin-amenable} below would have to be suitably enlarged to include stabilizers
of such sets.
\end{remark}

\begin{example} \label{ex:SL_2}
To fix ideas we point out that the axioms hold for the group
$G=SL_2(\Z)$ acting on the hyperbolic plane $T=\mathbb H^2$ with its Gromov
boundary $\Delta=S^1$. Fix a pairwise disjoint and equivariant
collection $\bfY$ of open horoballs. Projection data $\mathcal Y$ will
consists of the single collection $\bfY$, and we will set
$\Delta(Y)=\Delta$ for all $Y\in\bfY$. If $Y,Z$ are distinct
horoballs, denote by $\pi_Y(Z)$ the nearest point projection of $Z$ to
the closure of $Y$. This is an open interval in the horocycle boundary
of $Y$ and it has uniformly bounded diameter. If $X,Y,Z$ are distinct
horoballs we set $$d^\pi_Y(X,Z)=\diam \big(\pi_Y(X)\cup\pi_Y(Z)\big)$$
where the diameter is taken with respect to the metric in $T$ (or if
one wishes in the path metric of the horocycle; they are coarsely
equivalent and the distinction is irrelevant). Similarly, if
$\xi\in\Delta$ is a point on the circle that does not belong to the
closure of the horoball $Y$, there is a well defined nearest point
projection $\pi_Y(\xi)$ in the boundary of $Y$, and we again define
$$d^\pi_Y(X,\xi)=\diam \big(\pi_Y(X)\cup\pi_Y(\xi)\big)$$
Finally, if $\xi$ is in the closure of $Y$, we define
$d^\pi_Y(X,\xi)=\infty$ for all horoballs $X\neq Y$. Here $\pi_Y(\xi)$
is not defined, but we may think of it as the point at infinity of
the horocycle.

To complete the description of the flow data, we define $\calg_K$ as
the set of geodesic rays contained in $G\cdot K$, for any compact
$K\subset T$. Verification of the axioms  is
left to the reader. 
\end{example}

\begin{example}
Let $G$ be a hyperbolic group relative to a finitely generated subgroup $H$ (or more generally relative to a collection of such subgroups).
Then $G$ acts on a proper $\delta$-hyperbolic space (see \cite{bowditch,farb,gromov,groves-manning})
with Gromov boundary $\Delta$ and the action is
cocompact in the complement of a pairwise disjoint equivariant
collection of horoballs. 
Projection and flow data can be constructed
in the same way as above and all axioms hold.

The Farrell-Jones conjecture for relatively hyperbolic groups was proved in~\cite{Bartels-Coarse-flow} and generalizing this argument to mapping class groups was the motivation for the present work.
\end{example}

For the description of flow and projection data in the case of mapping
class groups, see Section \ref{sec:FJC-for-MCGs}.

  \begin{theorem}
    \label{thm:axioms->-fin-amenable}
    Assume that there is projection data and flow data for the action of $G$ on $\Delta$
    satisfying the axioms from~\ref{axiom:projections},~\ref{axiom:large-projection} and~\ref{axiom:flow}.
    Assume that $\Delta$ is compact, metrizable and finite dimensional. 
    If, in addition, $\calf$ contains the
    family $\VCyc$ of virtually cyclic subgroups of $G$ and all isotropy groups for the
    action of $G$ on all the $\bfY \in \caly$, then
    the action of $G$ on $\Delta$ is finitely $\calf$-amenable.
  \end{theorem}

  \begin{proof}
    This will be an easy consequence of
    Theorems~\ref{prop:projection-cover} and~\ref{prop:cover-flow}
    that we prove later.

    Let $S \subseteq G$ finite.
    Let us say that a collection of open subsets of $G \x \Delta$ is \emph{$S$-long} at $(g,\xi)$ if
    for one of its members $U$ we have $gS \x \{ \xi \} \subseteq U$.

    Theorem~\ref{prop:projection-cover} provides for each $\bfY \in \caly$ a $G$-invariant 
    collection $\calu_\bfY$ of
    open $\calf$-subsets and a number $\Theta$ such that $\calu_{\thin} := \bigcup_{\bfY \in \caly} \calu_\bfY$
    is $S$-long at all $(g,\xi)$ that admit a $\Theta$-large projection.
    The order of each $\calu_\bfY$ is at most $1$ and therefore the order of
    $\calu_\thin$ is at most $N_\thin := 2 k - 1$.

    Next we use Axiom~\ref{axiom:large-projection}.
    Thus there is $K \subseteq G$ compact such that all $(g,\xi)$ that do not admit a $\Theta$-large projection
    belong to the $K$-thick part of $G \x \Delta$.

    Theorem~\ref{prop:cover-flow} provides a $G$-invariant collection $\calu_\thick$ 
    of open $\VCyc$-subsets of $G \x \Delta$ that is $S$-long at all $(g,\xi)$ from the $K$-thick part.
    Moreover, the order of $\calu_\thick$ is bounded by a number $N_\thick$ independent from $K$ and $S$.
    
    Altogether $\calu_\thin \cup \calu_\thick$ is the cover we need and the action 
    of $G$ on $\Delta$ is $(N_\thin + N_\thick  + 1)$-$\calf$-amenable. 
  \end{proof}

  \section{Partial covers from the projection complex}
    \label{sec:proj-covers}

  \subsection{Projection covers} \label{subsec:proj-covers}
    
    Throughout this section we consider a $G$-space $\Delta$.
    We assume that we are given a $G$-set $\bfY$, subsets $\Delta(Y) \subseteq \Delta$
    for $Y \in \bfY$ and projection distances 
    $$d^\pi_Y \colon \big( \bfY \smallsetminus \{ Y \} \big) \x 
               \big(\Delta(Y) \amalg \bfY \smallsetminus \{Y\} \big) \to [0,\infty].$$  
    For $\xi \in \Delta$, $X \in \bfY$ we set
    $d_\bfY^\pi(X,\xi) := \sup \{ d^\pi_Y(X,\xi) \mid {\xi \in \Delta(Y)} \}$.
    If there is no such $Y$, then we set $d_\bfY^\pi(X,\xi) = -\infty$.
    We also write $\calf_\bfY$ for the family of subgroups of $G$ that fix an element of $\bfY$.

  \begin{thm}
    \label{prop:projection-cover}
    Assume that the projection distances $(d^\pi_Y)_{Y \in \bfY}$ satisfy the 
    axioms~\ref{axiom:proj:symmetry} to~\ref{axiom:proj:coarse-semi-cont}  
    listed in~\ref{axiom:projections}.
    Pick a base point $X_\bfY \in \bfY$.
    Let $S \subseteq G$ be finite.
    Then there is $\Theta \geq 0$ and     
    a $G$-invariant collection $\calu$ of open $\calf_\bfY$-subsets 
    of $G \x \Delta$ such that
    \begin{enumerate}
    \item \label{prop:projection-cover:order}
       the order of $\calu$ is at most $1$;
    \item \label{prop:projection-cover:long} 
       for any $(g,\xi) \in (G \x \Delta)$ with $d_\bfY^\pi(gX_\bfY,\xi) \geq \Theta$ 
       there is $U \in \calu$ with $gS \x \{ \xi \} \subseteq U$.
    \end{enumerate}
  \end{thm}

  The proof of Theorem~\ref{prop:projection-cover} occupies the remainder of this section.

  \subsection{The projection complex and angles.} \label{subsec:proj-cx}

  The restrictions of the $d^\pi_Y$ to $\bfY \smallsetminus \{Y\}$ satisfy the projection axioms
  (PC1) to (PC4) from~\cite[Sec.~3.1]{bbf}.  
  These axioms allow, for a constant $K >> \theta$, 
  the construction of the projection complex $\calp_K(\bfY)$~\cite[Sec.~3.3]{bbf}.
  This complex is a graph whose set of vertices is $\bfY$.
  We write $d_\calp$ for the path metric with edges of length $1$ on $\calp_K(\bfY)$.
  The action of $G$ on $\bfY$ extends to a simplicial action on $\calp_K(\bfY)$.
  Crucial for us will be the following two properties from~\cite{bbf} of the projection complex.

  \begin{proposition}
    \label{prop:proj-cx}
    There is a constant $\theta'_\calp > 0$ (depending on $K$) such that the following holds.
    \begin{enumerate}
      \item \label{prop:proj-cx:loc-estimate} \emph{Local estimate.}
        If $Y$ is an internal vertex of a geodesic in $\calp_K(\bfY)$ from $X$ to $Z$, and if
        $X'$ and $Z'$ are two vertices on this geodesic such that $X'$ is between $X$ and $Y$, 
        and $Z'$ is between $Y$ and $Z$ then
        $|d^\pi_Y(X,Z) - d^\pi_Y(X',Z')| < \theta'_\calp$.  
      \item \label{prop:proj-cx:attraction} \emph{Attraction property.}
        If $d^\pi_Y(X,Z) > \theta'_\calp$, then any geodesic in $\calp_K(\bfY)$ between $X$ and $Z$ will
        pass through $Y$. 
    \end{enumerate} 
  \end{proposition}

  \begin{proof}
    The construction of $\calp_K(\bfY)$ depends on a bounded perturbation $d_Y$ of the restrictions 
    of the $d^\pi_Y$ to $\bfY \smallsetminus \{Y\}$,
    i.e., $|d_Y^\pi(X,Z) - d_Y(X,Z)|$ is uniformly bounded in $X,Y$ and $Z$, 
    see~\cite[Thm.~3.3(B)]{bbf}.
    Let $c$ be a geodesic in $\calp_K(\bfY)$ from $X$ to $Y$ and let
    $X'$ be any internal vertex of $c$. 
    Then $d_{Y}(X,X')$  is uniformly bounded by~\cite[Cor.~3.15]{bbf}.
    Thus $d^\pi_{Y}(X,X')$ is also uniformly bounded and the local estimate follows from the
    triangle inequality.
 
    For $d_Y$ the attraction property is the content of the first statement of~\cite[Lem.~3.18]{bbf}.
    Since $d_Y$ is a uniformly bounded perturbation of $d^\pi_Y$, the attraction property follows 
    for $d^\pi_Y$ as well.
  \end{proof}
  
  For an internal vertex $Y$ of a geodesic $c$ in the projection complex we define the
  \emph{angle of $c$ at $Y$} as $\varangle_Y c := d^\pi_Y(X,Z)$ where $X$ and $Z$ are the two vertices
  on $c$ adjacent to $Y$.  
  If $Y$ is disjoint from $c$, then we set $\varangle_Y c = 0$; 
  if $Y$ is the start or end point of $c$, then $\varangle_Y c$ remains undefined.
  For $X,Z \neq Y$ we set $d^{\max}_Y(X,Z) := \max \{ \varangle_Y c \}$,
  where $c$ varies over the set of all geodesics from $X$ to $Z$. 

  We record the following consequences of Proposition~\ref{prop:proj-cx}
  for these the definitions. 
  Item~\ref{lem:angles:equal-after-attraction} is the 
  main advantage $d^{\max}$ has over $d^\pi$.  

  \begin{lemma}
    \label{lem:angles}
    There is $\theta_\calp > 0$ such that for all vertices $X,Z \neq Y$
    \begin{enumerate}
    \item \label{lem:angles:pi-vs-minmax} 
       $|d^\pi_Y(X,Z) - d^{\max}_Y(X,Z)| < \theta_\calp$;
    \item \label{lem:angles:min-vs-max}
      for any geodesic $c$ from $X$ to $Z$ we have $d^{\max}_Y(X,Z) - \varangle_Y c < \theta_\calp$;
    \item \label{lem:angles:attraction} if $d^{\max}_Y(X,Z) \geq \theta_\calp$ or $d^\pi_Y(X,Z) \geq \theta_\calp$
            then any geodesic from $X$ to $Z$ passes through $Y$;
     \item \label{lem:angles:equal-after-attraction} 
        let $c$ be a geodesic from $X$ to $Z$ that passes through $Y_0$ and $Y$ in this order,
        if $$\max \{ d_{Y_0}^\pi(X,Y), d_{Y_0}^\pi(X,Z), \varangle_{Y_0} c \} \geq \theta_\calp,$$ 
        then $d^{\max}_{Y}(X,Z) = d^{\max}_{Y}(Y_0,Z)$.
    \end{enumerate}
  \end{lemma}

  Note that there is a uniform bound on the difference between any of the three numbers appearing in 
  the hypothesis of item~\ref{lem:angles:equal-after-attraction} in the above lemma.

  \begin{proof}[Proof of Lemma~\ref{lem:angles}]
    For $\theta_\calp > 2 \theta'_\calp$, properties~\ref{lem:angles:pi-vs-minmax}
    and~\ref{lem:angles:min-vs-max} are a consequence of the local estimate provided there
    exists a geodesic from $X$ to $Z$ that passes through $Y$.
    If there is no such geodesic,  properties~\ref{lem:angles:pi-vs-minmax}
    and~\ref{lem:angles:min-vs-max} follow from the attraction property. 
    For $\theta_\calp > 3 \theta'_\calp$, property~\ref{lem:angles:attraction} follows
    again from the attraction property.
     Property~\ref{lem:angles:attraction} implies that under the
    assumption of~\ref{lem:angles:equal-after-attraction} 
    any geodesic from $X$ to $Z$ can be built by concatenation of a geodesic from $X$ to $Y_0$ with a geodesic 
    from $Y_0$ to $Z$. 
    Thus~\ref{lem:angles:equal-after-attraction} holds. 
  \end{proof}
  
  \subsection{The numbers $\Theta_0,\dots,\Theta_5$.}  \label{subsec:numbers}

  \begin{lemma} \label{lem:Theta-bfS}
    For $\bfS \subset \bfY$ finite, there is $\theta'_\bfS \geq 0$ such that
    for all $X,X' \in \bfS$, $Y,Z \in \bfY$, 
      $$|d^\pi_{Y}(X,Z) - d^\pi_{Y}(X',Z)| < \theta'_\bfS.$$
  \end{lemma}

  \begin{proof}
    By finiteness~\ref{axiom:proj:finiteness}, $d^\pi_Y(X,X')$ is bounded for $Y \in \bfY$, $X,X' \in \bfS$.
  \end{proof}

  \begin{lemma}
    \label{lem:Theta-bfS-angles}
    For $\bfS \subset \bfY$ finite, there is $\theta_\bfS \geq 0$ such that
    for $X \in \bfS$, $Y,Z \in \bfY$, $X,Z \neq Y$ we have  
    \begin{enumerate}
    \item \label{lem:bfS-angles:X-vs-X'-ds} 
       the distance between any two numbers from 
       $$\{ d^\pi_Y(X',Z), d^{\max}_Y(X',Z) \mid X' \in \bfS \}$$ 
       is $< \theta_\bfS$;
    \item \label{lem:bfS-angles:attraction} 
      if $d^{\max}_Y(X,Z) \geq \theta_\bfS$ or $d^\pi_Y(X,Z) \geq \theta_\bfS$
      then for any $X' \in \bfS$ any geodesic from $X'$ to $Z$ passes through $Y$;
    \item \label{lem:bfS-angles:equal-after-attraction} 
      suppose there is a geodesic $c$ from $X$ to $Z$ that passes through $Y_0$ and $Y$ in this order,
      if one of the numbers
      $$d_{Y_0}^\pi(X,Y), d_{Y_0}^{\max}(X,Y),d_{Y_0}^\pi(X,Z),d_{Y_0}^{\max}(X,Z)$$
      is $\geq \theta_\bfS$, then for any $X' \in \bfS$
      $d^{\max}_{Y}(X',Z) = d^{\max}_{Y}(Y_0,Z)$. 
    \end{enumerate}
  \end{lemma}

  \begin{proof}
     This follows by combining Lemma~\ref{lem:Theta-bfS} with Lemma~\ref{lem:angles}. 
  \end{proof}

  Fix now $S \subseteq G$ finite.
  Let $\theta_S := \theta_{S \cdot X_\bfY}$ as in Lemma~\ref{lem:Theta-bfS-angles}.
  Next choose numbers $0 << \Theta_0 << \Theta_1 << \Theta_2 << \Theta_3 << \Theta_4 << \Theta_5$. 
  Later we will need estimates of the form $\Theta_i > \Theta_j + C$
  for $i>j$ 
and for $C$ a constant depending
  on $\theta$ and $\theta_S$ and it will be clear that we can choose the $\Theta_i$ at this point  
  to satisfy all required estimates.
  (On the other hand, $\Theta_i := 10 \cdot (i+1) \cdot (\theta + \theta_S)$ will certainly work.) 

  \subsection{The finite projections $Z(g,\xi)$.} \label{subsec:proj-Z}

  For all $(g,\xi) \in G \x \Delta$ with $d^\pi_\bfY(gX_\bfY,\xi) > \Theta_4$
  we pick $Z(g,\xi) \in \bfY$ 
  such that $d^\pi_{Z(g,\xi)}(gX_\bfY, \xi) > \Theta_4$.
  In addition, if possible, choose $Z(g,\xi)$ so that 
  $d^\pi_{Z(g,\xi)}(gX_\bfY, \xi) > \Theta_5$.  
  We can arrange this map to be $G$-equivariant, i.e., such that $Z(hg,h\xi) = hZ(g,\xi)$ for $h \in G$.
  
  \begin{remark} 
    The use of the axiom of choice to produce the $Z(g,\xi) \in \bfY$ 
    may seem a little heavy handed.
    Assuming that $\bfY$ is countable we can do this, with a little more care, 
    using only countable choice:
     
    By equivariance it suffices to choose the $Z(1,\xi)$ with $1 \in G$ the unit.
    Fix a countable basis $U_i$ of open sets for
    $\Delta$ and consider all pairs $(U_i,Y)$ such that
    $d^\pi_Y(X_\bfY,\xi)>\Theta_4$ for all $\xi\in U_i$. 
    Choose an ordering of this countable set of pairs. 
    Then for $\xi$ let $Z(1,\xi)=Y$
    where $(U_i,Y)$ is the first pair with $\xi \in U_i$.
    By the coarse semi-continuity axiom~\ref{axiom:proj:coarse-semi-cont} 
    this produces $Z(1,\xi)$ for all $(1,\xi)$ with $d^\pi_\bfY(1,\xi)>\Theta_4+\theta$.
    It is not difficult to adjust the constants in the rest of our argument to account
    for this slightly weaker statement.
  \end{remark}

  \begin{lemma}
    \label{lem:Y}
    Let $(g,\xi) \in G \x \Delta$ with $d^\pi_\bfY(gX_\bfY,\xi) > \Theta_5$. 
    Then there is an open neighborhood $U$ of $\xi$ in $\Delta$ and $Y \in \bfY$ such that
    for any $s \in S$, $\xi' \in U$ either $Z(gs,\xi') = Y$ or 
    $d^{\max}_Y(gsX_\bfY,Z(gs,\xi)) > \Theta_3$.

  \end{lemma}

  \begin{proof}  
    It suffices to consider $g=e$.
    By coarse semi-continuity~\ref{axiom:proj:coarse-semi-cont} 
    we find a neighborhood $U$ of $\xi$ in $\Delta$ such that
    $d^\pi_{Z(e,\xi)}(X_\bfY,\zeta) > \Theta_5 - \theta$ for all $\zeta \in U$.
    For $s \in S$ then $d^\pi_{Z(e,\xi)}(sX_\bfY,\zeta) > \Theta_5 - \theta_S - \theta > \Theta_4$
    by Lemma~\ref{lem:Theta-bfS-angles}~\ref{lem:bfS-angles:X-vs-X'-ds}.
    In particular, for all $s \in S$, $\zeta \in U$, the vertex $Z(s,\zeta)$ is defined.
    We claim that for $(s,\zeta) \in S \x U$ with $Z(s,\zeta) \neq Z(e,\xi)$ 
    \begin{equation} \label{eq:one-is-large}
      \max \{ d^\pi_{Z(s,\zeta)}(X_\bfY,Z(e,\xi)), d^\pi_{Z(e,\xi)}(X_\bfY,Z(s,\zeta)) \} > 
                 \Theta_3 + \theta_S. 
    \end{equation}
    Indeed, assume $d^\pi_{Z(e,\xi)}(X_\bfY,Z(s,\zeta)) \leq \Theta_3 + \theta_S$.
    Then 
    \begin{align*}  
      d^\pi_{Z(e,\xi)}(Z(s,\zeta),\zeta) 
          \geq d^\pi_{Z(e,\xi)}(X_\bfY,\zeta)   -  d^\pi_{Z(e,\xi)}(X_\bfY,Z(s,\zeta)) & \\
          > \Theta_5 - \Theta_3 & - \theta_S  > \theta.
    \end{align*}

    The inequality on triples implies $d^\pi_{Z(s,\zeta)}(Z(e,\xi),\zeta) < \theta$.
    Thus 
    \begin{align*}
       d^\pi_{Z(s,\zeta)}(X_\bfY,Z(e,\xi)) > \qquad &  \\
         d^\pi_{Z(s,\zeta)}(sX_\bfY,\zeta) -  d^\pi_{Z(s,\zeta)}&(X_\bfY,sX_\bfY) -d^\pi_{Z(s,\zeta)}(Z(e,\xi),\zeta) \\
         & \qquad >  \Theta_4 - \theta_S - \theta > \Theta_3 + \theta_S, 
    \end{align*}
    proving~\eqref{eq:one-is-large}.
    Now we combine~\eqref{eq:one-is-large} with 
    Lemma~\ref{lem:Theta-bfS-angles}~\ref{lem:bfS-angles:attraction},
            \ref{lem:bfS-angles:equal-after-attraction}.
    Thus for $(s,\zeta) \in S \x U$ with $Z(s,\zeta) \neq Z(e,\xi)$ we have 
    \begin{itemize}
    \item either, for every $s' \in S$, every geodesic from $s'X_\bfY$ to $Z(e,\xi)$ passes through
       $Z(s,\zeta)$ and $d^{\max}_{Z(s,\zeta)}(s'X_\bfY,Z(e,\xi)) \geq \Theta_3$,
    \item or, for every $s' \in S$, every geodesic from $s'X_\bfY$ to $Z(s,\zeta)$ passes through
       $Z(e,\xi)$ and $d^{\max}_{Z(e,\xi)}(s'X_\bfY,Z(s,\zeta)) \geq \Theta_3$. 
    \end{itemize}
    If there is no $(s,\zeta)$ to which the first case applies, then we set $Y := Z(e,\xi)$.
    Otherwise, we pick $Y$ among the $Z(s,\zeta)$ to which the first case applies of minimal distance
    to $X_\bfY$.
    Since this $Y$ is also among those the $Z(s,\zeta)$ of maximal distance from $Z(e,\xi)$, it follows
    that it is also of minimal distance from $s'X_\bfY$ for all $s' \in S$. 
    Now, Lemma~\ref{lem:Theta-bfS}~\ref{lem:bfS-angles:equal-after-attraction} implies that 
    whenever $Z(s,\zeta) \neq Y$, then $d^{\max}_Y(sX_\bfY,Z(s,\zeta)) \geq \Theta_3$.  
  \end{proof}
  
\begin{remark} 
A key tool from~\cite{bbf} are linear orders constructed from the projection distances.
For $X,Z \in \bfY$ there is a linear order on the set of all $Y \in \bfY$ for
which $d^\pi_Y(X,Z)$ is defined and large. 
It is natural to add $X$ as a minimal element and $Z$ as a maximal element to this linear order;
we will then call it the linear order from $X$ to $Z$.
In this order $Y < Y'$ if and only if $d^\pi_Y(X,Y')$ is large.
Using these orders the construction of $Y$ in Lemma~\ref{lem:Y} can be summarized as follows:
For each pair $(s,\zeta)$ either $Z(s,\zeta)$ belongs to the linear order from $X_\bfY$ to $Z(e,\xi)$
or $Z(e,\xi)$ belongs to the order from $X_\bfY$ to $Z(s,\zeta)$.
The possible positions of $Z(s,\zeta)$ in these orders are sketched in Figure~\ref{fig:Z(s,zeta)}.
Then $Y$ can be defined as the minimal element in the order from $X_\bfY$ to $Z(e,\xi)$ among
all $Z(s,\zeta)$ to which the first case applies.
This is well defined as these linear orders are finite by the finiteness axiom~\ref{axiom:proj:finiteness}.
\begin{figure}[h]
\begin{tikzpicture}
  \draw (0,0) -- (10,0) (-.5,1) -- (1,0) (0,-1.5) -- (2,0)  (7,0) -- (9.5,-1.5) (8,0) -- (10.5,1.25);
  \fill [black,opacity=.5] (0,0) circle (2pt); \draw (-.5,0) node {$X_\bfY$};
  \fill [black,opacity=.5] (10,0) circle (2pt);  \draw (10.5,0) node {$\xi$};
  \fill [black,opacity=.5] (-.5,1) circle (2pt); \draw (-1,.8) node {$s'X_\bfY$};
  \fill [black,opacity=.5] (0,-1.5) circle (2pt); \draw (-.5,-1.5) node {$sX_\bfY$};
  \fill [black,opacity=.5] (9.5,-1.5) circle (2pt); \draw (10,-1.5) node {$\zeta$};
  \fill [black,opacity=.5] (10.5,1.25) circle (2pt); \draw (10.8,1.25) node {$\zeta'$};   
  \fill [black,opacity=.5] (3,0) circle (2pt); \draw (3,.5) node {$Z(s,\zeta)$};
  \fill [black,opacity=.5] (5,0) circle (2pt); \draw (5,.5) node {$Z(e,\xi)$};
  \fill [black,opacity=.5] (9,.5) circle (2pt); \draw (8.7,1) node {$Z(s',\zeta')$};
\end{tikzpicture}
\caption{Possible positions of $Z(s,\zeta)$} \label{fig:Z(s,zeta)}
\end{figure}
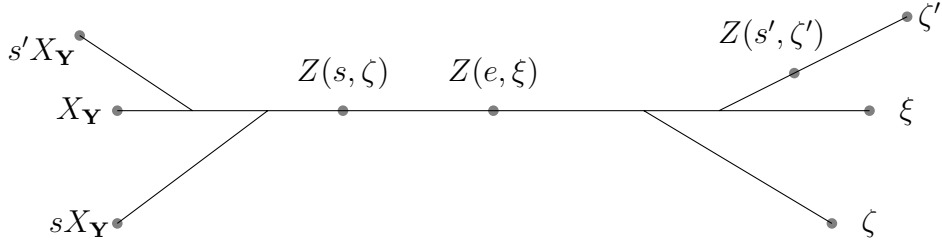  
\end{remark} 
  
  \subsection{The open sets $U(Y,i)$.} \label{subsec:U(Y,i)}

  For $(g,\xi) \in G \x \Delta$ with $d^\pi_\bfY(gX_\bfY,\xi) > \Theta_4$ we now use
  the projection complex $\calp_K(\bfY)$ to make the following definitions.
  \begin{itemize}
  \item For $i=0,1,2$  we define vertices $Y_i(g,\xi)$ of $\calp_K$ as 
    the unique vertex with the following two properties.
    Firstly, $d^{\max}_Y(gX_\bfY, Y_i(g,\xi)) < \Theta_i$ for all $Y \in \bfY \smallsetminus \{ gX_\bfY, Y_i(g,\xi) \}$.
    Secondly, $Y_i(g,\xi) = Z(g,\xi)$ or there exists a geodesic $c$ from
    $gX_\bfY$ to $Z(g,\xi)$ with $d^{\max}_{Y_i(g,\xi)} (gX_\bfY, Z(g,\xi)) \geq \Theta_i$.
    (The uniqueness of $Y_i(g,\xi)$ is a consequence of Lemma~\ref{lem:angles} since
    $\Theta_i > \theta_\calp$.)
  \item For $Y \in \bfY$, $i=1,2$ we define $U_+(Y,i) \subseteq G \x \Delta$ to consist of all
    $(g,\xi)$ with $d^\pi_\bfY(gX_\bfY,\xi) > \Theta_4$ and $Y = Y_i(g,\xi)$.
    We define $U(Y,i)$ as the interior of $U_+(Y,i)$ in $G \x \Delta$.
    For $i = 1,2$ we set $\calu(i) := \{ U(Y,i) \mid Y \in \bfY \}$.
  \end{itemize}

  \begin{lemma}
    \label{lem:U(Y,i)-disjoint} 
    For $Y \neq Y' \in \bfY$ we have $U(Y,i) \cap U(Y',i) = \emptyset$.
    For $g \in G$, $Y \in \bfY$ we have $g(U(Y,i)) = U(gY,i)$.
  \end{lemma}
 
  \begin{proof}
    This is a direct consequence of the definition of $U(Y,i)$.
  \end{proof}
    
  \begin{lemma}
    \label{lem:calu(g,i)-long}
    Let $(g,\xi) \in G \x \Delta$ with $d^\pi(gX_\bfY,\xi) > \Theta_5$.
    Then there are $Y \in \bfY$ and $i \in \{ 1,2 \}$ with 
    $gS \x \{ \xi \} \subseteq U(Y,i)$.
  \end{lemma}

  \begin{proof}  
    We can assume $g=e$. 
    Set $Y_i := Y_i(e,\xi)$ for $i=0,1,2$.

    According to Lemma~\ref{lem:Y} there is $Y_3 \in \bfY$ and an open neighborhood $U$ of $\xi$ in $\Delta$ 
    such that for any $s \in S$, $\zeta \in U$, either $Y_3 = Z(s,\zeta)$ or 
    $d^{\max}_{Y_3}(sX_\bfY,Z(s,\zeta)) > \Theta_3$.
    This implies, since $\Theta_3 > \theta_\calp$, by 
    Lemma~\ref{lem:angles}~\ref{lem:angles:equal-after-attraction},
    \begin{equation}
      \label{eq:Y_3-vs-Z}
      d^{\max}_Y(sX_\bfY,Z(s,\zeta)) = d^{\max}_Y(sX_\bfY,Y_3)
    \end{equation}
    for any $Y$ on a geodesic from $sX_\bfY$ to $Y_3$.  
   
    We claim that for any internal vertex $Y$ of a geodesic from $Y_0$ to $Y_3$ and 
    any $s \in S$, $\zeta \in U$ we have 
    \begin{equation} \label{eq:same-between-Y0-and-Y3}
      d^{\max}_{Y}(sX_\bfY,Z(s,\zeta)) = d^{\max}_{Y}(X_\bfY,Z(e,\xi)).
    \end{equation}
    To prove this claim we observe first
    \begin{equation*}
      d^{\max}_{Y_0}(X_\bfY,Y_3)  > 
         d^{\max}_{Y_0}(X_\bfY,Z(e,\xi)) - \theta_S \geq \Theta_0 - \theta_S > \theta_S. 
    \end{equation*}
    Therefore Lemma~\ref{lem:Theta-bfS-angles}~\ref{lem:bfS-angles:equal-after-attraction} implies
    $d^{\max}_Y(sX_\bfY,Y_3) = d^{\max}_Y(X_\bfY,Y_3)$ for all $s \in S$.
    Now~\eqref{eq:same-between-Y0-and-Y3} follows from~\eqref{eq:Y_3-vs-Z}.
 
    Next we claim that, provided $Y_0 \neq Y_3$, we have   
    \begin{equation}
      \label{eq:at-Y0}
      |d^{\max}_{Y_0}(sX_\bfY,Z(s,\zeta)) - d^{\max}_{Y_0}(X_\bfY,Z(e,\xi))| < \theta_S
    \end{equation}
    for all $s \in S$, $\zeta \in U$.
    To prove this claim we note that by~\eqref{eq:Y_3-vs-Z} to we have,  
    $d^{\max}_{Y_0}(X_\bfY,Y_3) = d^{\max}_{Y_0}(X_\bfY,Z(e,\xi)) \geq \Theta_0 > \theta_\bfS$.
    Thus, by Lemma~\ref{lem:Theta-bfS-angles}~\ref{lem:bfS-angles:attraction},
    any geodesic from $sX_\bfY$ to $Y_3$ will pass through $Y_0$.
    Using again~\eqref{eq:Y_3-vs-Z} we have $d^{\max}_{Y_0}(sX_\bfY,Z(s,\zeta)) = d^{\max}_{Y_0}(sX_\bfY,Y_3)$  
    for $s \in S$, $\zeta \in U$. 
    Now Lemma~\ref{lem:Theta-bfS-angles}~\ref{lem:bfS-angles:X-vs-X'-ds} implies~\eqref{eq:at-Y0}.

    Let for $s \in S$, $Y$ be an internal vertex of a geodesic from $sX_\bfY$ to $Y_0$.
    We claim that then for any $\zeta \in U$
    \begin{equation}
      \label{eq:before-Y0}
      d^{\max}_{Y}(sX_\bfY,Z(s,\zeta)) < \Theta_1.
    \end{equation}
    Suppose, by contradiction, $d^{\max}_{Y}(sX_\bfY,Z(s,\zeta)) \geq \Theta_1$.
    Then, by~\eqref{eq:Y_3-vs-Z}, $d^{\max}_{Y} (sX_\bfY,Y_3) \geq \Theta_1$.
    Lemma~\ref{lem:Theta-bfS-angles}~\ref{lem:bfS-angles:X-vs-X'-ds}
    implies $d^{\max}_{Y}(X_\bfY,Y_3) > \Theta_1 - \theta_{S} > \Theta_0$.
    Using~\eqref{eq:Y_3-vs-Z} again, we have $d^{\max}_{Y}(X_\bfY,Z(e,\xi)) > \Theta_0$.
    By definition of $Y_0$ and $Y_3$, this implies that $Y$ is closer to $Y_3$ then $Y_0$.
    But this contradicts that $Y$ is closer to $sX_\bfY$ than $Y_0$.
    This establishes~\eqref{eq:before-Y0}. 
   
    Now, if $Y_0 = Y_1 = Y_2 \neq Y_3$, then, by~\eqref{eq:at-Y0}, 
    $$d^{\max}_{Y_0}(sX_\bfY,Z(s,\zeta)) \geq d^{\max}_{Y_0}(X_\bfY,Z(e,\xi))
    - \theta_S \geq \Theta_2 - \theta_S > \Theta_1.$$
    Using~\eqref{eq:before-Y0} this implies $Y_0 = Y_1 = Y_1(s,\zeta)$
    for all $s \in S$, $\zeta \in U$.  Thus, in this case, $S \x U
    \subseteq U_+(Y_1,1)$ and therefore $S \x \{\xi\} \subseteq
    U(Y_1,1)$.  

    If $Y_0 \neq Y_2 \neq Y_3$ then~\eqref{eq:same-between-Y0-and-Y3},
    \eqref{eq:at-Y0} and~\eqref{eq:before-Y0}, imply $Y_2(s,\zeta) =
    Y_2$ for all $s \in S$, $\zeta \in U$.  Thus, in this case
    $S \x U \subseteq U_+(Y_2,2)$ and therefore $S \x \{\xi\}
    \subseteq U(Y_2,2)$. 

    Finally, if $Y_2 = Y_3$, then we use in addition that
    $d^{\max}_{Y_3}(sX_\bfY,Z(s,\zeta)) > \Theta_3$ for all $s \in S$, $\zeta \in U$
    with $Z(s,\zeta) \neq Y_3$.
    Combining this with~\eqref{eq:same-between-Y0-and-Y3},
    \eqref{eq:at-Y0} and~\eqref{eq:before-Y0}, we find again $Y_2(s,\zeta) =
    Y_2$ for all $s \in S$, $\zeta \in U$.  Thus, also in this case
    $S \x U \subseteq U_+(Y_2,2)$ and therefore $S \x \{\xi\}
    \subseteq U(Y_2,2)$.  
  \end{proof}
  
\begin{remark} 
	Informally the key observation in the proof of Lemma~\ref{lem:calu(g,i)-long} is:
	angles at $Y_0$ and $Y_3$ depend on $(s,\zeta)$ only up to a bounded error and 
	all other angles behave as indicated in Figure~\ref{fig:Y_i}. 
    \begin{figure}[H]
    \begin{tikzpicture}
       \draw (0,0) -- (10,0) (0,-1.5) -- (3,0)  (7,0) -- (9.5,-1.5);
       \fill [black,opacity=.5] (0,0) circle (2pt); \draw (-.5,0) node {$X_\bfY$};
       \fill [black,opacity=.5] (10,0) circle (2pt);  \draw (10.5,0) node {$\xi$};
       \fill [black,opacity=.5] (0,-1.5) circle (2pt); \draw (-.5,-1.5) node {$sX_\bfY$};
       \fill [black,opacity=.5] (9.5,-1.5) circle (2pt); \draw (10,-1.5) node {$\zeta$};  
       \fill [black,opacity=.5] (3,0) circle (2pt); \draw (3,.5) node {$Y_0$};
       \fill [black,opacity=.5] (3.6,0) circle (2pt); \draw (3.6,.5) node {$Y_1$};
       \fill [black,opacity=.5] (5,0) circle (2pt); \draw (5,.5) node {$Y_2$};
       \fill [black,opacity=.5] (7,0) circle (2pt); \draw (7,.5) node {$Y_3$};
       \draw [dotted] (3,.3) -- (3,-3)  (7,.3) -- (7,-3);
       \draw (1,-2.3) node {angles are small};
       \draw (5,-1) node {angles do not};
       \draw (5,-1.6) node {depend on $(s,\zeta)$}; 

    \end{tikzpicture}
     \caption{The position of the $Y_i$} \label{fig:Y_i}
    \end{figure}
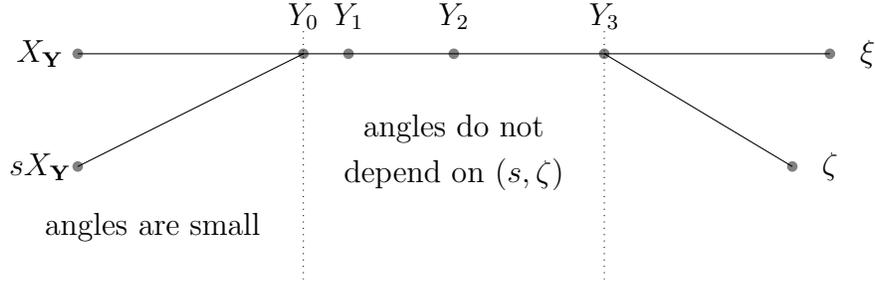  
\end{remark}  
  
  \begin{proof}[Conclusion of Proof of Theorem~\ref{prop:projection-cover}]
    We can use $\calu := \calu(1) \cup \calu(2)$ and $\Theta := \Theta_5$.
    Lemma~\ref{lem:U(Y,i)-disjoint} implies that the order of $\calu$ is at most $1$
    and that its members are $\calf_\bfY$-sets.
    Lemma~\ref{lem:calu(g,i)-long} states that $\calu$ has the
    property required in
    \ref{prop:projection-cover:long} in the $G$-direction.
  \end{proof}

  \section{Partial covers from a flow space}
     \label{sec:cover-thick}

   Throughout this section $\Delta$ will be a finite dimensional, metrizable, compact
   space with a $G$-action.
   Moreover, $\overline T = T \cup \Delta$,  $\calg_K$ will
   be flow data as in Definition~\ref{def:flow-data}.

   We fix a base point $x_0 \in T$ and define for $K \subseteq T$ compact 
   $$(G \x \Delta)_{K} \subseteq G \x \Delta$$ 
   to consist of all $(g,\xi)$ for which
   there exists  a ray $c \in \calg_K$ with $c(0)=gx_0$ and $c(\infty)=\xi$, i.e.,
   $(G \x \Delta)_{K}$ is the $K$-thick part of $G \x \Delta$.

   \begin{thm} \label{prop:cover-flow}
     Assume that the flow axioms~\ref{axiom:flow:small-at-infty},~\ref{axiom:flow:fellow-travel} 
     and~\ref{axiom:flow:infinite} listed in~\ref{axiom:flow} are satisfied.
     Then there exists a number $N_{\thick}$ with the following property.
     For $S \subseteq G$ finite and $K \subseteq T$ compact,
     there exists a $G$-invariant collection $\calu_{\thick}$ of
     open $\calf$-subsets of $G \x \Delta$ such that the following two
     conditions are satisfied:
     \begin{enumerate}
     \item \label{prop:cover-flow:dim}
        the order of $\calu_{\thick}$ is at most $N_{\thick}$;
     \item \label{prop:cover-flow:long}
        for any $(g,\xi) \in (G \x \Delta)_{K}$ there is $U \in \calu_{\thick}$
        with $gS \x \{ \xi \} \subseteq U$.
    \end{enumerate}
   \end{thm}

   The proof of this result is in two steps.
   We first build a coarse flow space that admits long thin covers, i.e., covers that have a large Lebesgue number in the direction of the flow.
   Then we coarsely map $(G \x \Delta)_K$ to this flow space and show that for large time the coarse flow sends $gS \x \{ \xi \}$ into such long and thin sets.
   The cover $\calu_{\thick}$ is then obtained by pulling back the long thin cover from the coarse flow space.  
   To guide the reader through this proof we comment on the dependence of the appearing constants.
   In Theorem~\ref{prop:cover-flow} we are given $K \subseteq T$ compact and $S \subseteq G$ finite.
   Together, $K$ and $S$ determine a number $\rho := d_T(K \cup Sx_0, x_0)$.
   Through the fellow traveler axiom~\ref{axiom:flow:fellow-travel} $\rho$ determines a number $\beta$ in Lemma~\ref{lem:beta-rho}.
   Finally, the time $\tau$ for which the coarse flow is applied is provided in Lemma~\ref{lem:thick-from-long}.

   \subsection{Coarse flow spaces}

   We set $V := Gx_0 \subseteq T$.
   Note that since the action of $G$ on $T$ is proper, $\overline{V} := V \cup \Delta$ 
   is a closed and therefore compact subspace of $\overline{T}$. 
   We will define the coarse flow space as a subspace of $\overline{V} \x V \x \Delta$.
   Informally, it consists of all triples $(\xi_-,v,\xi_+)$ for which $v$ coarsely belongs to a quasi geodesic from $\xi_-$ to $\xi_+$.
   
   \begin{definition}  \label{def:CFK}
     For $K \subseteq T$ compact and $\rho > 0$
     we define $\CF_{0}(K,\rho)$ be the subspace of $V \x V \x \Delta$ consisting of all
     triples $(v_-,v,\xi_+) \in V \x V \x \Delta$ for which there is $c \in \calg_K$ with
     $d_T(c(0), v_-) \leq \rho$, $d_T(\Image(c),v) \leq \rho$ and $c(\infty) = \xi_+$. 

     We define the \emph{coarse $(K,\rho)$-flow space $\CF(K,\rho)$}  as the closure
     of $\CF_{0}(K,\rho)$ inside of $\overline{V} \x V \x \Delta$.
     For $(x_-,\xi_+) \in \overline{V} \x \Delta$ we define the \emph{coarse flow line} 
     between $x_-$ and $\xi_+$ as 
     \begin{equation*}
        V_{K,\rho}({x_-,\xi_+})  :=  \{ v \in V \mid (x_-,v,\xi_+) \in \CF({K,\rho}) \} \subseteq V.
     \end{equation*} 
   \end{definition}
   
  \begin{lemma} 
    \label{lem:V_K-rho}
    For $K \subseteq T$ compact and $\rho > 0$ there is $R > 0$ with the following property.
    \begin{enumerate}
    \item \label{lem:V_K-rho:ray} 
      Let $(v_-,\xi_+) \in V \x \Delta$ with $V_{K,\rho}(v_-,\xi_+) \neq \emptyset$.
      Then there exists a quasi-geodesic ray $c \colon [0,\infty) \to T$ such that
      the coarse flow line $V_{K,\rho}(v_-,\xi_+)$ is contained in the $R$-neighborhood
      of the image of $c$.
    \item \label{lem:V_K-rho:bi-infinite}
      Let $(\xi_-,\xi_+) \in \Delta \x \Delta$ with $V_{K,\rho}(\xi_-,\xi_+) \neq \emptyset$. 
      Then there exists a quasi-geodesic $c \colon \IR \to T$ such that
      the coarse flow line $V_{K,\rho}(\xi_-,\xi_+)$ is contained in the $R$-neighborhood
      of the image of $c$.
    \end{enumerate}
    Here the additive constant for the quasi-geodesic (ray) $c$
    depends only on $K$ and $\rho$, while the multiplicative constant is independent from
    $K$ and $\rho$.
  \end{lemma}

  \begin{proof}
    \ref{lem:V_K-rho:ray}
    We use $R$ from the fellow traveling axiom~\ref{axiom:flow:fellow-travel}.
    Still from~\ref{axiom:flow:fellow-travel} we obtain for all $t \in [0,\infty)$
    a neighborhood $W_t$ of $\xi_+$ in $\Delta$ such that for all $c,c' \in \calg_K$
    with $d_T(c(0),v_-), d_T(c'(0),v_-) \leq \rho$ and $c(\infty), c'(\infty) \in W_t$ 
    we have $d_T(c(t),c'(t)) \leq R$.
    As $c$ is a quasi-geodesics with uniform constants, $d_T(c(n),c(t))$ is uniformly bounded for $t \in [n,n+1)$, and similar for $c'$. 
    After increasing the constant $R$, if necessary, we can assume $W_t$ is constant on intervals
    $[n,n+1)$, $n \in \IN$.
    Now we can also assume, that for $t \geq t'$ we have $W_t \subseteq W_{t'}$.
    If $V_{K,\rho}(v_-,\xi_+) \neq \emptyset$, then for $t \geq 0$ there is $c_t \in \calg_K$
    with $d_T(c_t(0),v_-) \leq \rho$ and $c_t(\infty) \in W_t$.
    We now define the quasi-geodesic ray $c$ by $c(t) := c_t(t)$.
    The multiplicative constant for $c$ agrees with the multiplicative constant of the 
    rays from $\calg_K$, while the additive constant may increase by at most $R$.
    It is not difficult to check that $V_{K,\rho}(v_-,\xi_+)$ is contained in the
    $R + \rho$-neighborhood of the image of $c$. 
    \\[1ex]
    \ref{lem:V_K-rho:bi-infinite} 
    The quasi-geodesic $c$ and $R$ are provided by the
    infinite quasi-geodesic axiom~\ref{axiom:flow:infinite}.
    Using in addition the small at $\infty$ axiom~\ref{axiom:flow:small-at-infty}
    it follows that $V_{K,\rho}(\xi_-,\xi_+)$ is contained in the $R$-neighborhood
    of the image of $c$.  
  \end{proof}

  \begin{lemma}
    \label{lem:qi-embedding}
    Let $K \subseteq T$ compact and  $\rho > 0$. 
    Then for any $(x_-,\xi_+) \in \overline{V} \x \Delta$  there 
    exists a quasi isometric embedding 
    $V_{K,\rho}(x,\xi) \to \IZ$.

    Moreover, the additive constant for this embedding depends only on $K$ and $\rho$, while the
    multiplicative constant is also independent from $K$ and $\rho$.
  \end{lemma}

  \begin{proof}
    This follows from Lemma~\ref{lem:V_K-rho} since the $R$-neighborhood of a quasi-geodesic
    ray is quasi-isometric to $\IN \subset \IZ$ and the $R$-neighborhood of a quasi-geodesic
    is quasi-isometric to $\IZ$.
  \end{proof}

  \subsection{Long thin covers.}

  A subset $W \subseteq V$ is said to be $R$-separated if
  $d_T(w,w') \geq R$ for all $w \neq w' \in W$.
  A subset $V_0 \subset V$ is said to be $(D,R_0)$-doubling if the 
  following holds for all $R \geq R_0$: if $W \subseteq V_0$ 
  is $R$-separated and contained in a ball of radius $2R$,
  then the cardinality of $W$ is at most $D$.   
  
  \begin{proposition}
    \label{prop:CFK}
    Let $K \subseteq T$ be compact and $\rho > 0$.
    \begin{enumerate}
    \item \label{prop:CFK:dim} $\dim \CF(K,\rho) \leq 2 \dim \Delta < \infty$. 
    \item \label{prop:CFK:doubling} 
       For all $(x_-,\xi_+) \in \overline{V} \x \Delta$
       the set $V_{K,\rho}({x_-,\xi_+})$ is $(D,R_0)$-doubling.
       Here the constant $D$ is independent of $K$, $\rho$, $x_-$ and $\xi_+$, while
       the constant $R_0$ depends on $K$ and $\rho$, but not on $x_-$ and $\xi_-$.
    \item \label{prop:CFK:isotropy}
       For each $(x_-,v,\xi_+) \in \CF({K,\rho})$, the isotropy group
       $G_{(x_-,\xi_+)} := \{ g \in G \mid g(x_-,\xi_+) = (x_-,\xi_+) \}$ 
       of $(x_-,\xi_+) \in \overline{V} \x \Delta$ 
       is virtually cyclic.
    \end{enumerate}
  \end{proposition}

  \begin{proof}
    \ref{prop:CFK:dim} 
    As $\overline{V}$ and $\Delta$ are separable and  metrizable any subspace of 
    $\overline{V} \x V \x \Delta$ is of
    dimension at most $\dim \overline{V} + \dim \Delta = 2 \dim \Delta$.
    \\[1ex]
    \ref{prop:CFK:doubling}
    The metric space $\IZ$ is $(D,R_0)$-doubling with $D=3$, $R_0 = 0$.
    It follows that every space  
    that quasi isometrically embeds into
    $\IZ$ is $(D',R_0')$-doubling with $D'$ depending on $D$ and the multiplicative
    constant of the quasi isometry and 
    $R_0'$ depending on $D,R_0$ and the constants for the quasi isometry.
    Therefore~\ref{prop:CFK:doubling} is a consequence of Lemma~\ref{lem:qi-embedding}.  
    \\[1ex]
    \ref{prop:CFK:isotropy}
    The isotropy group $G_{x_-,\xi_+}$ acts properly and isometrically on the coarse flow line
    $V_{L,\rho}(x_-,\xi_+)$.
    Since $V_{L,\rho}(x,\xi)$ embeds quasi-iso\-metri\-cally into $\IZ$
    by Lemma~\ref{lem:qi-embedding} it follows that $G_{x,\xi}$ is virtually cyclic.
  \end{proof}

  We can now prove that our coarse flow spaces admit long thin covers.

  \begin{proposition}
    \label{prop:long-thin-cover-CFK}
    There is a number $N_\longthin$ such that for any $K \subseteq T$ compact   
    and any $\rho, \beta > 0$ 
    there exists a $G$-invariant 
    cover $\calu_{\longthin}$ of $\CF(K,\rho)$ by open $\VCyc$-sets such that the following two
    conditions are satisfied:
    \begin{enumerate}
    \item the order of $\calu_\longthin$ is at most $N_\longthin$;
    \item for any $(x_-,v,\xi_+) \in \CF(K,\rho)$ there is $U \in \calu_\longthin$ with
          $$\{x_-\} \x B_\beta(v) \x \{ \xi_+ \} \cap \CF(K,\rho) \subseteq U.$$
    \end{enumerate}
    Here $B_\beta(v)$ is the $\beta$-neighborhood of $v$ in $V$.
  \end{proposition}

  \begin{proof}
    Using Proposition~\ref{prop:CFK} this follows from~\cite[Thm.~1.1]{Bartels-Coarse-flow}. 
  \end{proof}

  If $U \in \calu_\longthin$ as above, then, typically, as a subset of $\overline{V} \x V \x \Delta$ the set $U$ will be very small (i.e. thin) in the $\overline{V}$- and the $\Delta$-coordinate, while the coordinate in $V$ varies over a subset that is long in the coarse flow lines (and thus long and coarsely thin).    

  \subsection{The coarse flow.}

  Informally, we have a coarse flow on the coarse flow space that moves towards $\xi_+$ along coarse flow lines.
  The partially defined coarse maps $\iota_\tau$ in the next definition should be thought of as the composition of the coarse flow for time $\tau$ with the map $\iota_0$ that sends
      $(g,\xi)$ to the initial point $gx_0$ in the coarse flow line $V_{K,\rho}(gx_0,\xi)$.
      
  \begin{definition}
    \label{def:coarse-flow}
    Let $K \subseteq T$ compact and $\rho \geq 0$. 
    For $\tau \geq 0$ and $(g,\xi) \in G \x \Delta$ 
    we define $$\iota_\tau(g,\xi) \subseteq V_{K,\rho}(gx_0,\xi)$$ to
    consist of all $v \in V$ for which there is $c \in \calg_K$ with
    $d_T(c(0),gx_0) \leq \rho$, $d_T(c(\tau),v) \leq \rho$ and $c(\infty) = \xi$.
  \end{definition}

  For $K \subseteq T$ compact and $S \subseteq G$ finite we enlarge $(G \x \Delta)_K$ to  
  \begin{equation*}
    (G \x \Delta)_K^S := \{ (gs,\xi) \mid s \in S, (g,\xi) \in (G \x \Delta)_K\}
  \end{equation*}
  in order to have a space that contains $gS \x \{ \xi \}$ for all $(g,\xi) \in (G \x \Delta)_K$. 
  We now use $\iota_\tau$ to pull back open sets from the coarse flow space $\CF(K,\rho)$ to $G \x \Delta$ and to open subsets of $(G \x \Delta)_K^S$.

  \begin{definition}
    \label{def:iota-U}
    Let $K \subseteq T$ compact and $\rho \geq \diam K \cup \{x_0\}$.
    For $U \subseteq \CF(K,\rho)$ and $\tau > 0$ we define 
    $\iota^{-\tau} U \subseteq G \x \Delta$
    to consist of all $(g,\xi)$ for which
    \begin{equation*}
      \{ gx_0 \} \x \iota_\tau(g,\xi) \x \{ \xi \} \subseteq U
    \end{equation*} 
    For $S \subseteq G$ finite we define 
    $$\iota_S^{-\tau} U \subseteq (G \x \Delta)_K^S$$ 
    as the interior of $\iota^{-\tau} U \cap (G \x \Delta)_K^S$ in $(G \x \Delta)_K^S$. 
    If $\calu$ is a collection of open subsets of $\CF(K,\rho)$, then we set
    $\iota^{-\tau}_S \calu := \{ \iota^{-\tau}_S U \mid U \in \calu \}$.
  \end{definition}

  In the proof of Theorem~\ref{prop:cover-flow} we will
  later use a collection of the form
  $\iota^{-\tau}_S \calu_\longthin$, where 
  $\calu_\longthin$ comes from Proposition~\ref{prop:long-thin-cover-CFK}.

  The intersection with $(G \x \Delta)_K^S$ in the definition of $\iota_S^{-\tau} U$
  is used to guarantee $\iota_\tau(g,\xi) \neq \emptyset$ in the proof of the following lemma. 

  \begin{lemma}
    \label{lem:iota-tau}
    Let $K \subseteq T$ compact, $S \subseteq G$ finite.
    Let $\rho \geq \diam K \cup Sx_0 \cup \{x_0\}$.
    Let $U,U' \subseteq \CF(K,\rho)$, $\tau > 0$ and $\gamma \in G$.
    Then
    \begin{enumerate}
    \item \label{lem:iota-tau:disjoint}
      if $U \cap U' = \emptyset$, then $\iota_S^{-\tau}U \cap \iota_S^{-\tau} U' = \emptyset$;
    \item \label{lem:iota-tau:G}
      $\iota_S^{-\tau}(\gamma U) = \gamma (\iota_S^{-\tau}U)$.
    \end{enumerate}
  \end{lemma}

  \begin{proof} 
     We start with \ref{lem:iota-tau:disjoint}.
     Let $(g,\xi) \in \iota_S^{-\tau} U \cap \iota_S^{-\tau} U'$.
     Then $(g,\xi) \in (G \x \Delta)_K^S$ and 
     $\{gx_0 \} \x \iota_\tau(g,\xi) \x \{\xi\} \subseteq U \cap U'$.
     It remains to show that $\iota_\tau(g,\xi) \neq \emptyset$.
     Since $(g,\xi) \in (G \x \Delta)_K^S$ there are $s \in S$ and $c \in \calg_K$ with
     $c(0) = gs^{-1}x_0$, $c(\infty) = \xi$.
     There is $h \in G$ such that $c(\tau) \in hK$, since $c \in \calg_K$.
     Now $\rho \geq \diam K \cup \{x_0\}$ implies $d_T(c(\tau),hx_0) \leq \rho$, and $\rho \geq diam Sx_0 \cup \{x_0\}$ implies $d_T(c(0),gx_0) = d_T(gs^{-1}x_0,gx_0) = d_T(x_0,sX_0) \leq \rho$. 
     Thus $hx_0 \in \iota_\tau(g,\xi)$.
     
     Assertion~\ref{lem:iota-tau:G} is a direct consequences of the definitions.

  \end{proof}

  \subsection{Construction of $\calu_\thick$}

  By construction, the coarse flow lines are thickenings of the quasi geodesics from $\calg_K$.
  In the next lemma we replace one of the two quasi geodesics appearing in the fellow traveler axiom~\ref{axiom:flow:fellow-travel} with a coarse flow line. 
  
  \begin{lemma}
    \label{lem:beta-rho}
    Let $K \subseteq T$ be  compact  and $\rho > 0$.     
    There is $\beta > 0$ such that the following holds.
    Let $c \in \calg_K$ and $\tau \geq 0$ and such that $c(\tau) \in K$.
    Then there exists a neighborhood $W$ of $\xi := c(\infty)$ such that for all
    $\xi' \in W$, $g' \in G$, $v' \in \iota_\tau(g',\xi')$ with $d_T(c(0),g'x_0) \leq \rho$
    we have $$d_T(x_0,v') \leq \beta.$$     
  \end{lemma}
   
  \begin{proof} 
    We apply the fellow traveler axiom~\ref{axiom:flow:fellow-travel} for $2\rho$ and obtain a number $R \geq 0$.
    Still from~\ref{axiom:flow:fellow-travel} we obtain for $x := c(0)$, $\xi := c(\infty)$, and $t := \tau$ a neighborhood $W$ of $\xi$. 
    Let now $\xi' \in W$, $g' \in G$, $v' \in \iota_\tau(g'x_0,\xi')$ with $d_T(c(0),g'x_0) \leq \rho$.
    Then there is $c' \in \calg_K$ with $d_T(c'(0),g'x_0) \leq \rho$, $d_T(c'(\tau),v') \leq \rho$
    and $c'(\infty) = \xi'$.
    Now $d_T(c(0),c'(0)) \leq d_T(c(0),g'x_0) + d_T(g'x_0,c'(0)) \leq 2\rho$.
    Therefore, the assertion of the fellow traveling property yields
    $d_T(c(\tau),c'(\tau)) \leq R$.
    This implies
    \begin{align*}
      d_T(x_0,v') \leq d_T(x_0,c(\tau)) +  d_T(c(\tau)&,c'(\tau))  + d_T(c'(\tau),v') \\ \leq & 
      \diam K \cup \{x_0\} + R + \rho =: \beta.
    \end{align*}  
  \end{proof}

  To prove Theorem~\ref{prop:cover-flow} we will use a $\beta$-long thin cover $\calu_\longthin$ from Proposition~\ref{prop:long-thin-cover-CFK} where $\beta$ comes from Lemma~\ref{lem:beta-rho}.
  In the next lemma we use the small at $\infty$ axiom~\ref{axiom:flow:small-at-infty} to show that for sufficiently large $\tau$ the pull back of $\calu_\longthin$ with $\iota_\tau$ to $(G \x \Delta)_K^S$ is $S$-long for all $(g,\xi) \in (G \x \Delta)_K$, i.e., it satisfies the assertion~\ref{prop:cover-flow:long} in Theorem~\ref{prop:cover-flow}. 

  \begin{lemma}
    \label{lem:thick-from-long}
    Let $K \subseteq T$ be  compact  and $S \subseteq G$ be  finite.
    Let $\rho \geq \diam Sx_0 \cup \{x_0\} \cup K$. 
    Let $\beta$ be as in Lemma~\ref{lem:beta-rho}.

    Let $\calu_\longthin$ be the cover of $\CF(K,\rho)$ appearing in
    Proposition~\ref{prop:long-thin-cover-CFK}.
    Then there is $\tau > 0$ such that 
    for all $(g,\xi) \in (G \x \Delta)_{K}$ there is $U \in \calu_\longthin$ with 
    \begin{equation*}
      gS \x \{ \xi \} \subseteq \iota_S^{-\tau} U.
    \end{equation*}
  \end{lemma}

  \begin{proof}
    We argue by contradiction and assume that the assertion fails.
    Then, for $\tau \to \infty$, there are $(g_\tau,\xi_\tau) \in (G \x \Delta)_{K}$
    such that 
    \begin{equation}
      \label{eq:no-U}
      g_\tau S \x \{ \xi_\tau \} \not \subseteq \iota_S^{-\tau} U 
        \quad \text{for all} \quad U \in \calu_\longthin.
    \end{equation}
    Since $(g_\tau,\xi_\tau) \in (G \x \Delta)_K$ there is $c_\tau \in \calg_K$ with
    $c_\tau(0) = g_\tau x_0$ and $c_\tau(\infty) = \xi_\tau$.
    Since $\calu_\longthin$ is $G$-invariant, we may assume that $c_\tau(\tau) \in K$ for all $\tau$.
    Using $\rho \geq \diam K \cup \{ x_0 \}$, we obtain $x_0 \in V_{K,\rho}(g_\tau x_0, \xi_\tau)$ for all $\tau$.

    Since $\overline{V}$ and $\Delta$ are compact we can, after a subsequence, assume that
    \begin{equation*}
      \lim_{\tau \to \infty} (g_\tau x_0,x_0,\xi_\tau)  = (\xi_-,x_0,\xi_+) \in \CF({K,\rho})
    \end{equation*}
    exists.
    By Proposition~\ref{prop:long-thin-cover-CFK} there is $U \in \calu_\longthin$
    such that $$\{\xi_-\} \x B_\beta(x_0) \x \{ \xi_+ \} \cap \CF(K,\rho) \subseteq U.$$
    As $U$ is open and $B_\beta(x_0)$ is finite there are open neighborhoods 
    $U_- \subseteq \overline{V}$ of $\xi_-$ and $U_+ \subseteq \Delta$ of $\xi_+$ such that
    \begin{equation} \label{eq:to-be-in-U}
      U_- \x B_\beta(x_0) \x U_+ \cap \CF(K,\rho) \subseteq U.
    \end{equation}
    The small at $\infty$ axiom~\ref{axiom:flow:small-at-infty} implies that for 
    all $s \in S$ eventually $g_\tau sx_0 \in U_-$.
    We find now $\tau$ such that $g_\tau sx_0 \in U_-$ for all $s \in S$ and $\xi_\tau \in U_+$.
    We claim that     
    \begin{equation} \label{eq:in-iota-S-tau-U}
      g_\tau S \x \{ \xi_\tau \} \subseteq \iota_S^{-\tau} U.
    \end{equation}
    This will contradict~\eqref{eq:no-U}. 
    To prove~\eqref{eq:in-iota-S-tau-U} we apply Lemma~\ref{lem:beta-rho} 
    to $c := c_\tau$
    and obtain a neighborhood $W$ of $\xi_\tau$ in $\Delta$.
    After shrinking $W$ we may assume $W \subseteq U_+$.
    We claim that
    \begin{equation} \label{eq:W-in-iota-U}
      g_\tau S \x W \cap (G \x \Delta)_K^S \subseteq \iota^{-\tau} U
    \end{equation}
    which will imply~\eqref{eq:in-iota-S-tau-U}.
    Let $s \in S$ and $\xi' \in W$ with $(g_\tau s,\xi') \in (G \x \Delta)_K^S$. 
    We need to show that 
    $$\{ g_\tau sx_0 \} \x \iota_\tau(g_\tau s,\xi') \x \{ \xi' \} \subseteq U.$$
    Let $v' \in \iota_\tau(g_\tau s,\xi')$.
    Since $\xi' \in W$ and since $$d_T(c_\tau(0),g_\tau sx_0) = d_T(g_\tau x_0,g_\tau sx_0) \leq \rho$$ 
    we can use the assertion of Lemma~\ref{lem:beta-rho} (for $g' := g_\tau s$) to obtain 
    $d_T(x_0,v') \leq \beta$.
    Since $g_\tau sx_0 \in U_-$ and since $\xi' \in W \subseteq U_+$ we obtain
    $(g_\tau sx_0,v',\xi') \in U$ from~\eqref{eq:to-be-in-U}.
    Thus~\eqref{eq:W-in-iota-U} holds.
  \end{proof}

  \begin{proof}[Proof of Theorem~\ref{prop:cover-flow}]
    Set $N_\thick := N_\longthin$, where $N_\longthin$ is from 
    Proposition~\ref{prop:long-thin-cover-CFK}.
    Let $S \subseteq G$ finite and $K \subseteq T$ compact be given.
    Let $\rho := diam  K \cup Sx_0 \cup  \{x_0\}$.
    Let $\beta$ as in Lemma~\ref{lem:beta-rho}.
    Let $\calu_\longthin$ be the cover of $\CF(K,\rho)$ appearing in
    Proposition~\ref{prop:long-thin-cover-CFK}.
    As $\calu$ consists of $\VCyc$-subsets and is $G$-invariant, the same holds for
    $\iota_S^{-\tau} \calu_\longthin$ for any $\tau$ by Lemma~\ref{lem:iota-tau}.
    Lemma~\ref{lem:iota-tau} also implies that
    for any $\tau$, the order of $\iota^{-\tau}_S \calu_\longthin$, does
    not exceed the order of $\calu_\longthin$.
    By definition $\iota_S^{-\tau}$ consists of open subsets of $(G \x \Delta)_K^S$.
    As $\Delta$ is metrizable there exists a $G$-invariant metric on $G \x \Delta$.
    This allows us to extend each $V \in \iota^{-\tau}_S \calu$ to an open subset
    $V' \subseteq G \x \Delta$ such that $\calu_\thick := \{ V' \mid V \in \iota^{-\tau}_S \calu \}$
    also consists of $\VCyc$-subsets, also is $G$-invariant and also is of order at most $N_\thick$,
    see Lemma~\ref{lem:extension-of-covers} below for more details.
    Finally, there is, by Lemma~\ref{lem:thick-from-long}, $\tau > 0$ such that for 
    each $(g,\xi) \in (G \x \Delta)_K$ there is $V = \iota^{-\tau}_S U$ with
    \begin{equation*}
      gS \x \{ \xi \} \subseteq V \subseteq V' \in \calu_\thick.
    \end{equation*}
  \end{proof}

  \begin{lemma} 
     \label{lem:extension-of-covers}
     Let $X$ be a $G$-space and $Y$ be a $G$-invariant subspace.
     Assume that the topology on $X$ can be generated by a
     $G$-invariant metric $d$.
     Let $\calu$ be a $G$-invariant collection of open $\calf$-subsets of $Y$.
     Then there exits a $G$-invariant collection $\calu^+$ of open $\calf$-subsets of $X$
     such that
     \begin{enumerate}
     \item the order of $\calu^+$ equals the order of $\calu$;
     \item for each $U \in \calu$ there is $U^+ \in \calu^+$ with $U = U^+ \cap Y$. 
     \end{enumerate}
   \end{lemma}

   \begin{proof}
     For $U \in \calu$ set $U^+ := \{ x \in X \mid d(x,U) < d(x,Y
     \smallsetminus U)\}$.  It is not difficult to check that $\calu^{+} :=
     \{ U^+ \mid U \in \calu \}$ has the required properties, see for
     example~\cite[App.~B]{Bartels-Coarse-flow}.
   \end{proof}

  \section{Finitely $\calf$-amenable actions.} \label{s:fac}

   \subsection{Finite extensions and $N$-$\calf$-amenability}

The main results in this section are Propositions
\ref{prop:finite-index+fin-F-amenabilty} and
\ref{prop:finite-ind-calf-amenable-product}. While in principle it is
possible to prove these directly from the definition, we find it
convenient to reformulate $N$-$\calf$-amenability in terms of maps to
simplicial complexes and prove the statements from this point of
view. The covers one would naturally write down from the definition
would not be $\calf$-covers, corresponding to actions on simplicial
complexes with elements that leave a simplex invariant without fixing
it pointwise. This is fixed by barycentrically subdividing, while the
operation on covers is less transparent.

   Let $E$ be a simplicial complex with vertex set $V(E)$.
   Every point of $E$ can be written as $y = \sum_{v \in V_C} y_v \cdot v$ with
   $y_v \in [0,1]$ and $\sum_{v \in V(E)} y_v = 1$. 
   The $\ell^1$-metric on $E$ is defined by $d_E^1(y,y') := \sum_{v \in V(E)} |y_v - y'_v|$. 
   We briefly discuss products of simplicial complexes.
   Let $E$ be a simplicial complex.
   For $n \in \IN$ we define a simplicial structure on the $N$-fold cartesian product 
   $E^{\x n}$ 
   of $E$ as follows:
   First we replace $E$ by its barycentric subdivision.  
   This is a locally ordered simplicial complex;
   for each simplex the set of its vertices has a linear order and this order is
   compatible with taking faces of simplices.       
   The product of locally ordered simplicial complexes is canonically a simplicial complex.
   For a simplicial action of a group $G$ on $E$ the product action on $E^{\x n}$ 
   is also simplicial.
   On $E^{\x n}$ it will be convenient to use the product metric
   $d_{E^{\x n}}$ defined by 
   $d_{E^{\x n}}( (y_1,\dots,y_n), (y_1',\dots,y_n')) = \max_{1 \leq i \leq n} d_{E}^1(y_i,y'_i)$.
   This is not the $\ell^1$-metric $d^1_{E^{\x n}}$, but the change is
   uniformly controlled, provided that $E$ is finite dimensional.

   \begin{lemma}
     \label{lem:unif-cont}
     Fix $n,N \in \IN$. 
     Then for any $\e > 0$ there is $\e_0 > 0$ such that for any 
     simplicial complex $E$ of dimension at most $N$ we have
     for all $y = (y_1,\dots,y_n), y' = (y_1',\dots,y_n') \in E^{\x n}$
     \begin{eqnarray*}
        d_{E^{\x n}}( y,y') < \e_0 \;  & \implies & d^1_{E^{\x n}}(y,y') < \e, \\
        d^1_{E^{\x n}}( y,y') < \e_0 \; & \implies & d_{E^{\x n}}(y,y') < \e.
     \end{eqnarray*}
   \end{lemma}

   \begin{proof}
     If $E = \Delta^{N'}$, then this is a consequence of compactness of $(\Delta^{N'})^{\x n}$.
     But this case implies the general case for the following reason.
     Let $y,y' \in E^{\x n}$ be given.
     Then there are simplices $\sigma_i,\sigma_i'$, $i=1,\dots,n$ of $E$
     with $y_i \in \sigma_i$, $y'_i \in \sigma_i'$.
     Let $F \subset E$ be the subcomplex spanned by the $\sigma_i$.
     This subcomplex has at most $2n(N+1)$-many vertices and embeds therefore 
     into an $N'$-simplex $\sigma \cong \Delta^{N'}$, where $N' = 2n(N+1)-1$.
     Since $d_{E^{\x n}}(y,y') = d_{F^{\x n}}(y,y') = d_{\sigma^{\x n}}(y,y')$ and
     $d^1_{E^{\x n}}(y,y') = d^1_{F^{\x n}}(y,y') = d^1_{\sigma^{\x n}}(y,y')$,
     the general case follows.
   \end{proof}

   Let $E$ be a simplicial complex equipped with a simplicial $G$-action and $\Delta$ be a $G$-space.
   For $S \subset G$ finite and $\e > 0$ a continuous map $f \colon \Delta \to E$ is said to
   be \emph{$(S,\e)$-equivariant} if $$\sup_{x \in \Delta,s \in S} d_E^1( s f(x),f(sx)) < \e.$$ 
   Let $\calf$ be a family of subgroups of $G$.
   By an $(G,\calf)$-simplicial complex we mean a simplicial complex $E$ with a simplicial
   $G$-action such that all isotropy groups belong to $\calf$.
   
   \begin{lemma}
     \label{lem:almost-equiv-vs-cover}
     Let $\calf$ be a family of subgroups of $G$ and 
     $\Delta$ be a compact metrizable space with a $G$-action.
     Then the following are equivalent.
     \begin{enumerate}
     \item \label{lem:almost-equiv-vs-cover:amenable}
       The action of $G$ on $\Delta$ is $N$-$\calf$-amenable;
     \item \label{lem:almost-equiv-vs-cover:almost-equi}
       For any $S \subseteq G$ finite and $\e > 0$ there exists
       a $(G,\calf)$-simplicial complex of dimension at most $N$ and 
       an $(S,\e)$-equivariant map $\Delta \to E$.
     \end{enumerate}
   \end{lemma}

   \begin{proof}
     This is proven in~\cite[Prop.~4.2]{Guentner-Willett-Yu-dyn-asy-dim}
     with the following two minor changes:
     Firstly, in~\cite{Guentner-Willett-Yu-dyn-asy-dim}
     the covers of $G \x \Delta$ are in addition assumed to be cofinite for the action of $G$.
     However, since $\Delta$ is compact, it is always possible to pass to a cofinite subcover. 
     Secondly, in~\cite{Guentner-Willett-Yu-dyn-asy-dim} the family $\calf$ 
     is assumed to be closed under taking supergroups of finite index.
     This has the advantage that the isotropy groups for a simplicial action belong to $\calf$ if and
     only if the isotropy groups of all vertices belong to $\calf$.
     For general $\calf$ this is only true for cellular actions.
     However, the induced action on the nerve of an $\calf$-cover is cellular.
     Therefore~\cite[Prop.~4.2]{Guentner-Willett-Yu-dyn-asy-dim} remains true without the
     assumption that $\calf$ is closed under taking supergroups of finite index. 
   \end{proof}

   We will use Lemma~\ref{lem:almost-equiv-vs-cover} to discuss the
   behavior of $N$-$\calf$-amenability under finite extensions.  This
   is closely related
   to~\cite[Sec.~5]{Bartels-Lueck-Reich-Rueping-GLnZ}.  As a
   preparation we discuss coinduction.  Let $G_0 \subset G$ be a
   subgroup of finite index.  For a $G_0$-space $E_0$ we set $E :=
   \map_{G_0}(G,E_0)$; this is the coinduction of $E_0$ from $G_0$ to
   $G$.  We obtain an action of $G$ on $E$ as follows: the action of
   $g \in G$ on $(y \colon G \to E_0) \in E$ is given by the formula
   $(gy)(a) := y(ag)$ for all $a \in G$.  If $G_y$ is the isotropy
   subgroup of $G$ for $(y \colon G \to E_0) \in E$, then $G_y \cap
   G_0$ is contained in the isotropy subgroup $(G_0)_{y(e)}$ of $G_0$
   for $y(e) \in E_0$.  If $E_0$ is a simplicial complex, then $E$ is
   a simplicial complex with the following construction.  First we
   replace $E_0$ by its barycentric subdivision.  Now it is a locally
   ordered simplicial complex and the action of $G_0$ preserves this
   order.  Write $V(E_0)$ for the vertices of $E_0$ and define the
   vertices of $E$ by $V(E) := \map_{G_0}(G,V(E_0))$.  Now
   $v_0,\dots,v_n \in V(E)$ span a simplex for $E$ if for each $g \in
   G$ the vertices $v_0(g),\dots,v_n(g)$ span a simplex (possibly of
   dimension $<n$) of $E_0$ and in the local order of this simplex we
   have $v_0(g) \leq v_1(g) \leq \dots \leq v_n(g)$.  It will be
   convenient to use the $G$-invariant metric $d_{E}$ on $E$ defined
   by $d_E(y,y') := \max_{a \in G} d^1_{E_0}(y(a),y'(a))$.  This is
   \emph{not} the $\ell^1$-metric on $E$.  However,
   since, forgetting the $G$-action, $E=E_0^{\times m}$ for $m=[G:G_0]$, 
   Lemma~\ref{lem:unif-cont} implies that if $E_0$ is finite
   dimensional, then the identity on $E$ is in both directions between
   $(E,d^{1}_E)$ and $(E,d_E)$ uniformly continuous.

   If $f_0 \colon \Delta_0 \to E_0$ is $G_0$-equivariant, then we obtain a $G$-map
   $$\map_{G_0}(G,\Delta_0) \to \map_{G_0}(G,E_0), \quad \xi \mapsto f_0 \circ \xi.$$ 
   But if $f_0$ is not $G_0$-equivariant (maybe only $(S,\e)$-equivariant), then this 
   only defines a map
   $\map_{G_0}(G,\Delta_0) \to \map(G,E_0).$ 
   If $G = G_0t_1 \sqcup \dots \sqcup G_0t_n$, then there is a projection
   $\pi \colon \map(G,E_0) \to \map_{G_0}(G,E_0)$ determined by
   $\pi(y)(t_i) = y(t_i)$ for $i=1,\dots,n$. 
   We write now $$f \colon \map_{G_0}(G,\Delta_0) \to \map_{G_0}(G,E_0)$$ 
   for the map $\pi \circ (f_0)_*$; it is determined by
   $(f(\xi))(t_i) = f_0(\xi(t_i))$ for $i=1,\dots,n$.  

   \begin{lemma}
     \label{lem:prod-is-S,e-equivariant}
     For any finite $S \subset G$ and $\e > 0$ there are $S_0 \subset G_0$ 
     finite and $\e_0 > 0$ such that the following holds.
     Suppose that $f_0 \colon \Delta_0 \to E_0$ is $(S_0,\e_0)$-equivariant
     where $\dim E_0 \leq N$,
     then $f \colon \map_{G_0}(G,\Delta_0) \to \map_{G_0}(G,E_0) =: E$ 
     with $(f(\xi))(t_i) = f_0(\xi(t_i))$ for $i=1,\dots,n$ is
     $(S,\e)$-equivariant.    
   \end{lemma}

   \begin{proof} 
     Because of Lemma~\ref{lem:unif-cont} we can use the metric $d_E$ instead of $d^1_E$.
     Note that $d_E(y,y') = \max_{i} d^1_{E_0}(y(t_i),y'(t_i))$, since the action of $G_0$ on
     $E_0$ is isometric for $d^1_{E_0}$.
  
     Let $S \subset G$ be finite.
     Set $S_0 := \{ t_i s t_j^{-1} \mid s \in S, 1 \leq i,j \leq n \} \cap G_0$.
     Let $f_0 \colon \Delta_0 \to E_0$ be $(S_0,\e)$-equivariant.
     For $s \in S$, $\xi \in \map_{G_0}(G,\Delta_0)$ and $t_i$ we pick
     $t_j$ with $t_is \in G_0t_j$, thus $t_i s t_j^{-1} \in G_0$.
     Then
     \begin{align*}
       (f(s \xi)) (t_i) &= f_0((s\xi)(t_i)) = f_0( \xi(t_i s)) \\
          & \qquad = f_0( \xi (t_i s t_j^{-1} t_j)) = 
                f_0( t_i s t^{-1}_j \xi(t_j)) \\
       (s(f(\xi))) (t_i) &= (f(\xi))(t_is) = (f(\xi))(t_i s t_j^{-1} t_j) \\ 
          & \qquad = t_i  s t_j^{-1} (f(\xi)(t_j)) 
           = t_i s t_j^{-1} f_0(\xi(t_j))
     \end{align*}
     and therefore $d_{E_0}^{1}( (f(s \xi)) (t_i), (s(f(\xi))) (t_i) ) < \e$.
     It follows that $$d_E( f(s\xi), s(f(\xi))) < \e$$ and
     that $f$ is $(S,\e)$-equivariant. 
   \end{proof}

   \begin{proposition}
     \label{prop:finite-index+fin-F-amenabilty}
     Let $G$ act on the compact metrizable space $\Delta$. 
     Let $G_0$ be a subgroup of finite index $n$ in $G$ and let
     $\calf_0$ be a family of subgroups of $G_0$.
     Suppose that the restriction of the action on $\Delta$ to the subgroup $G_0$ 
     is $N$-$\calf_0$-amenable.
     Then the action of $G$ on $\Delta$ is $n \cdot N$-$\calf$-amenable, where
     $\calf$ is the family of subgroups $F$ of $G$ for which
     $F \cap G_0$ belongs to $\calf_0$.
   \end{proposition} 

   \begin{proof}
     It suffices to show that condition~\ref{lem:almost-equiv-vs-cover:almost-equi}
     from Lemma~\ref{lem:almost-equiv-vs-cover}
     passes from $G_0$ to $G$.
     Let $S \subseteq G$ be finite and $\e > 0$ be given.
     Write $\Delta_0$ for $\Delta$ with the action restricted to $G_0$.
     Let $f_0 \colon \Delta_0 \to E_0$ be $(S_0,\e_0)$-equivariant where
     $E_0$ is an $(G_0,\calf_0)$-simplicial complex of dimension at most $N$
     and $S_0 \subseteq G_0$ finite and $\e_0 > 0$ come from Lemma~\ref{lem:prod-is-S,e-equivariant}.
     Let $E := \map_{G_0}(G,E_0)$; this is an $(G,\calf)$-simplicial complex of
     dimension $\leq n \cdot N$. 
     By Lemma~\ref{lem:prod-is-S,e-equivariant} there exists an $(S,\e)$-equivariant 
     map $f \colon \map_{G_0} (G,\Delta_0) \to E$.
     Composing $f$ with the $G$-equivariant map $\Delta \to \map_{G_0}(G,\Delta_0)$,
     $\xi \mapsto (g \mapsto g\xi)$ we obtain a $(S,\e)$-equivariant map
     $\Delta \to E$.  
   \end{proof}
  
   \begin{proposition}
     \label{prop:finite-ind-calf-amenable-product}
     Let $G_0$ be a subgroup of $G$ of finite index $n$.
     Let $\calf_0$ be a family of subgroups of $G_0$.
     Let $\calf$ be the family of subgroups $F$ of $G$ for which
     $F \cap G_0$ belongs to $\calf$.
     Assume that there exists an $N$-$\calf_0$-amenable action 
     of $G_0$ on a compact metrizable space $\Delta_0$. 
     Then the induced action of $G$ on $\map_{G_0}(G,\Delta_0) \cong \Delta_0^n$ is 
     $n \cdot N$-$\calf$-amenable. 
   \end{proposition}

   \begin{proof} 
     If $E_0$ is an $(G_0,\calf_0)$-simplicial complex
     of dimension $N_0$, then 
     $E := \map_{G_0}(G,E_0)$ is an $(G,\calf)$-simplicial complex
     of dimension $n \cdot N_0$.
     The result follows from Lemma~\ref{lem:prod-is-S,e-equivariant} and 
     Lemma~\ref{lem:almost-equiv-vs-cover}.
   \end{proof}

  \subsection{The Farrell-Jones Conjecture}
     \label{subsec:FJC}

  Let $G$ be a group.
  Let $\cala$ be an additive category with a strict $G$-action 
  and a strict direct sum.
  Following~\cite[Sec.~4.1]{Bartels-Lueck-Borel} we will call
  such a category an additive $G$-category.
  For such a category $\cala$ there is then an additive category $\int_G \cala$.
  If $\cala$ is equivalent to the category of finitely generated free $R$-modules 
  and the $G$-action is trivial, then $\int_G \cala$ is equivalent to the category of
  finitely generated free $R[G]$-modules.
  Given a family $\calf$ of subgroups of $G$ there is the $K$-theoretic $\calf$-assembly map
  \begin{equation*}
    \alpha^K_{G,\calf} \colon H^G_*(E_\calf G;\bfK_\cala) \to K_*(\smint{G} \cala).
  \end{equation*}
  The $K$-theoretic Farrell-Jones Conjecture (with coefficients) asserts that 
  this map is an isomorphism if we use for $\calf$
  the family $\VCyc$ of virtually cyclic subgroups~\cite[Conj.~3.2]{Bartels-Reich-coeff-for-FJ}. 
  Farrell and Jones' original formulation~\cite{Farrell-Jones-Isom-Conj} for the group
  ring $\IZ[G]$ is a special case of this more general formulation.
  
  If $\cala$ is in addition equipped with a strict involution~\cite[Sec.~4.1]{Bartels-Lueck-Borel},
  then $\int_G \cala$ inherits an involution and there is for any family $\calf$ of subgroups of
  $G$ the $L$-theoretic $\calf$-assembly map
  \begin{equation*}
    \alpha^L_{G,\calf} \colon H^G_*(E_\calf G;\bfL_\cala^{-\infty}) \to 
                                   L^{\langle -\infty \rangle}_*(\smint{G} \cala).
  \end{equation*}   
  The $L$-theoretic Farrell-Jones Conjecture (with coefficients) asserts that 
  this map is an isomorphism if we use for $\calf$ the family $\VCyc$ 
  of virtually cyclic subgroups~\cite{Bartels-Lueck-Coeff-L}. 
  Again, Farrell and Jones' original formulation~\cite{Farrell-Jones-Isom-Conj} for the group
  ring $\IZ[G]$ is a special case of this more general formulation.

  We will say that a group $G$ \emph{satisfies the Farrell-Jones Conjecture relative to $\calf$}
  if for all $G$-additive categories $\cala$ (with strict involution in the $L$-theory case)
  the assembly maps $\alpha^K_{G,\calf}$ and $\alpha^L_{G,\calf}$ are isomorphisms.
  If $\calf = \VCyc$ then we will say that $G$ \emph{satisfies the Farrell-Jones Conjecture}.
  We will need the following results.

  \begin{theorem}[Transitivity principle]
    \label{thm:transitivity-principle} 
       Let $\calf$ be a family of subgroups of $G$.
       Assume that $G$ satisfies the Farrell-Jones Conjecture relative to $\calf$ and 
       that any subgroup $F \in \calf$ satisfies the Farrell-Jones Conjecture.
       Then $G$ satisfies the Farrell-Jones Conjecture.
  \end{theorem}

  \begin{proof}
    See for example~\cite[Thm.~2.10]{Bartels-Farrell-Lueck-cocompact-lattices}.
  \end{proof}
 
  \begin{remark} 
  	\label{rem:FJ-and-extensions}
  	An application of the Transitivity principle~\ref{thm:transitivity-principle} is the following inheritance property for the Farrell-Jones conjecture for extensions, see for example~\cite[Thm.~2.7]{Bartels-Farrell-Lueck-cocompact-lattices}.
    Let $N \to \hat G \to G$ be an extension.
    Suppose that $G$ satisfies the Farrell-Jones Conjecture and that the preimage in $\hat G$ of any virtually cyclic subgroup of $G$ also satisfies the Farrell-Jones Conjecture.
    Then $\hat G$ satisfies the Farrell-Jones Conjecture.
  \end{remark} 
 
  In the next result $\ER$ stands for Euclidean retract. Recall that a
  compact space $X$ is a {\it Euclidean retract} (or ER) if it can be
  embedded in some $\R^n$ as a retract. A compact metrizable space $X$
  is an ER if and only if it is a finite-dimensional contractible ANR.

  \begin{theorem}
    \label{thm:FJC-finitely-amenable-action}
    Let $\calf$ be a family of subgroups of $G$ that is closed under taking finite index 
    overgroups.
    Suppose that $G$ admits a finitely $\calf$-amenable action on a compact $\ER$.
    Then $G$ satisfies the Farrell-Jones Conjecture relative to $\calf$.
  \end{theorem}
 
  \begin{proof}
    This follows from the main axiomatic results 
    of~\cite{Bartels-Lueck-Borel, Bartels-Lueck-Reich-K-FJ-hyp}.
    The assumptions on $G$ are formulated somewhat differently in these references, 
    but it is not hard
    to check that the present assumptions imply the assumptions in these references,
    see also~\cite[Thm.~4.3]{Bartels-Coarse-flow}.
  \end{proof}
   
  If $\EF$ is a class of groups that is closed under taking subgroups and isomorphism, then
  for a group $G$ we denote  by $\EF(G)$ the family of subgroups of $G$ that belong 
  to $\EF$.
  For such a class of groups $\EF$ we define the  class $\Eac(\EF)$ of groups to consist of
  all groups $G$ that admit a finitely $\EF(G)$-amenable action on a compact $\ER$.
  Using the action on a point we see $\EF \subseteq \Eac(\EF)$.

  \begin{lemma}
    \label{lem:ac(F)}
    Let $\EF$ be a class of groups that is closed under isomorphisms, taking subgroups, 
    taking finite index overgroups, finite products and 
    central extensions with finitely generated kernel. Then
    \begin{enumerate}
    \item \label{lem:ac(F):FJC} 
       if all groups in $\EF$ satisfy the Farrell-Jones Conjecture, then
       all groups in $\Eac(\EF)$ satisfy the Farrell-Jones Conjecture;
    \item \label{lem:ac(F):over}
       the class $\Eac(\EF)$ is again closed under isomorphisms, taking subgroups, 
       taking finite index overgroups, finite products and central extensions
       with finitely generated kernel.
    \end{enumerate}
  \end{lemma}

  \begin{proof}
    \ref{lem:ac(F):FJC} By Theorem~\ref{thm:FJC-finitely-amenable-action} every group $G$ from
    $\Eac(\EF)$ satisfies the Farrell-Jones Conjecture relative to $\EF(G)$.
    By assumption every group from $\EF(G)$ satisfies the Farrell-Jones Conjecture.
    Therefore the transitivity 
    principle~\ref{thm:transitivity-principle}
    implies that $G$ satisfies the Farrell-Jones Conjecture.
    \\[1ex]
    \ref{lem:ac(F):over} That $\Eac(\EF)$ is closed under isomorphism and taking subgroups
    is clear from the definition. 
    Proposition~\ref{prop:finite-ind-calf-amenable-product} implies that
    it is also closed under finite index overgroups. 

    Let for $i=1,2$ the group $G_i$ act finitely $\calf_i$-amenable on $\Delta_i$.
    Then the product action of $G_1 \x G_2$ on $\Delta_1 \x \Delta_2$ is 
    finitely $\calf_1 \x \calf_2$-amenable.
    Since $\EF$ is assumed to be closed under finite products it follows that
    $\Eac(\EF)$ is also closed under finite products.

    Let $C \to \hat G \to G$ be a central extension with $C$ finitely generated.
    If $G$ acts finitely $\calf$-amenable on $\Delta$, then $\hat G$ acts via the projection
    $\hat G \to G$ finitely $\calf'$-amenable on $\Delta$, where $\calf'$ consists of
    central extensions with finitely generated kernel of groups in $\calf$. 
    Therefore, if $\calf \subseteq \EF$, then also $\calf' \subseteq \EF$.
    It follows that $\Eac(\EF)$ is closed under central extensions with finitely 
    generated kernel.       
  \end{proof}

  Starting with $\Eac^0(\EF) := \EF$ we can define inductively
  $\Eac^{n+1}(\EF) := \Eac (\Eac^{n}(\EF))$.
       We set $\EAC(\EF) := \bigcup \Eac^{n}(\EF)$. 

  \begin{corollary}
    \label{cor:AC}
    Let $\EF$ be a class of groups that is closed under isomorphisms, taking subgroups, 
    taking finite index overgroups and finite products.
    Assume that all groups from $\EF$ satisfy the Farrell-Jones Conjecture.
    Then all groups from $\EAC(\EF)$ satisfy the Farrell-Jones Conjecture.
  \end{corollary}

  \begin{proof}
    This follows by induction from Lemma~\ref{lem:ac(F)}.
  \end{proof}

  Let $\EVNil$ be the class of finitely generated virtually nilpotent groups.  
  Later, in Lemma~\ref{lem:MCG-in-AC}, we will use Theorem~\ref{thm:finitely-amenable-action-MCG} 
  to show that  mapping class groups of surfaces belong to $\EAC(\EVNil)$.
  Thus, to prove Theorem~\ref{thm:FJC-for-MCG} we will need the following 
  well-known result.

  \begin{proposition}
    \label{prop:FJC-for-VNil}
    All groups in the class $\EVNil$ 
    satisfy the Farrell-Jones Conjecture.
  \end{proposition}

  \begin{proof}
    Finitely generated virtually abelian groups
    satisfy the Farrell-Jones Conjecture, see for 
    example~\cite[Thm.~3.1]{Bartels-Farrell-Lueck-cocompact-lattices}.

    Let $N \to \hat G \to G$ be an extension.
    If $N$ is finitely generated and central, then all preimages of
    virtually cyclic groups are virtually finitely generated abelian.
    The inheritance property from Remark~\ref{rem:FJ-and-extensions} now implies that the Farrell-Jones Conjecture is stable under central extensions with finitely generated kernel.

    The case of finitely generated virtually nilpotent groups follows
    now by induction on the length of the lower central series.
    This induction is carried out in detail in an only marginally different
    situation in~\cite[Lem.~2.13]{Bartels-Lueck-Reich-FJ+appl}. 
  \end{proof}

  \begin{remark}
    \label{rem:not-nilpotent}
    All virtually nilpotent subgroups of the mapping class group are
    known to be virtually abelian, see \cite{BLM} and \cite[Theorem 8.9]{ivanov-tits}.
    Thus it may seem weird that nilpotent groups come up in our argument.
    Indeed, after a  reorganization of the induction process we could avoid mentioning
    nilpotent groups,
    but we would still need the fact that central extensions 
    (with finitely generated free abelian kernel) of groups satisfying the Farrell-Jones Conjecture
    satisfy the Farrell-Jones Conjecture.     
    As explained above, 
    the Farrell-Jones Conjecture for virtually nilpotent groups is an easy consequence of this fact and
    of the Farrell-Jones Conjecture for virtually abelian groups.
  \end{remark}

  \begin{remark}
    \label{rem:AC(Solv)}
    Lemma~\ref{lem:ac(F)} (and its proof) remains true  if we replace 
    \emph{central extension with finitely generated kernel}
    with \emph{extension with abelian kernel}.
    Since Wegner~\cite{Wegner-FJ-solv} proved the Farrell-Jones 
    Conjecture for the class $\EVSol$ of 
    virtually solvable groups it follows that all groups in $\EAC(\EVSol)$ 
    satisfy the Farrell-Jones Conjecture as well.
    Wegner's proof is considerably more involved than the proof of
    the Farrell-Jones Conjecture for virtually finitely generated nilpotent groups
    given above.
    
    It is not clear that $\EAC(\EVSol)$ or $\EAC(\EVNil)$ have all the inheritance properties
    known for the Farrell-Jones Conjecture.
    For example the Farrell-Jones Conjecture is also known to be stable under directed colimits and
    finite free products.
    If a group $G$ acts finitely $\calf$-amenably on a compact $\ER$, 
    then $G$ is (strongly) transfer reducible relative to $\calf$ 
    in the sense  of~\cite{Bartels-Lueck-Borel, Wegner-2012CAT0}.
    Groups that are transfer reducible relative to groups that satisfy the Farrell-Jones Conjecture
    satisfy the Farrell Jones Conjecture themselves.
    This is a consequence of the transitivity principle~\ref{thm:transitivity-principle}
    and the main axiomatic results from~\cite{Bartels-Lueck-Borel, Wegner-2012CAT0}. 
  \end{remark}

\section{Preliminaries on mapping class groups} \label{s:mcg}

The general references for this section are
\cite{FM,FLP,hubbard-book}. Let $\Sigma$ be a closed oriented surface
of genus $g$, with $p\geq 0$ punctures (i.e. distinguished points). We
denote by $P\subset\Sigma$ the set of punctures. The {\it mapping
  class group} $$\MCG=\pi_0(Homeo_+(\Sigma,P))$$ is the group of
components of the group of orientation-preserving homeomorphisms of
$\Sigma$ that leave $P$ invariant.

We will always assume $6g+2p-6>0$ regarding the other cases as
sporadic.  This condition is equivalent to the existence of more than
one complete hyperbolic structure of finite area on $\Sigma\smallsetminus P$.

The sporadic cases are $g=0$, $p\leq 3$ when $\MCG$ is finite,
and $g=1$, $p=0$ when $\MCG=SL_2(\Z)$ is virtually free.

A simple closed curve in $\Sigma\smallsetminus P$ is {\it essential}
if it does not bound a disk or a once punctured disk. We denote by
$\mathcal S$ the set of isotopy classes of essential simple closed
curves in $\Sigma\smallsetminus P$ and refer to its elements as {\it
  curves}. If $\Sigma\smallsetminus P$ is given a complete hyperbolic
structure of finite area, every $s\in \cals$ has a unique geodesic
representative. If $s,s'\in\mathcal S$ the {\it intersection number}
$i(s,s')$ is the smallest cardinality of $a\cap a'$ as $a,a'$ range
over simple closed curves in the isotopy classes $s,s'$
respectively. Thus $i(s,s)=0$ and for $s\neq s'$ $i(s,s')$ is the
cardinality of the intersection between the geodesic representatives
of $s,s'$. See \cite[Section 1.2]{FM} or \cite[Lemma
  2.6]{casson-bleiler}.

To $\Sigma$ one associates several spaces on which the
mapping class group $\MCG$ acts.

\subsection{Teichm\"uller space}

The {\it Teichm\" uller space} $\calt=\calt(\Sigma)$ is the space of
marked complex structures on $\Sigma$ with $P$ the set of
distinguished points. Equivalently, by the Uniformization Theorem,
$\calt$ is the space of marked complete hyperbolic structures of
finite area on $\Sigma\smallsetminus P$. 
Then $\calt$ is naturally a smooth (or even
complex analytic) manifold diffeomorphic to $\R^{6g+2p-6}$. The
mapping class group $\MCG$ acts by changing the marking. This
action is discrete but not cocompact; however, there are natural
cocompact subspaces. Fix an
$\epsilon>0$ and consider the {\it thick part} $\calt_{\geq\epsilon}\subset \calt$
 consisting of $X\in \calt$ such that every closed hyperbolic geodesic has length
$\geq\epsilon$.

\begin{theorem}[Mumford \cite{mumford}]\label{mumford}
The thick part $\calt_{\geq\epsilon}\subset \calt$ is cocompact.
\end{theorem}

Thus
for $\epsilon_n\searrow 0$, the sequence $\calt_{\geq\epsilon_n}$
forms an exhaustion of $\calt$ by  
cocompact subsets. 

For $X \in \calt$ we use the marking to identify the 
  set of curves on $X$ with $\cals$.

\subsection{Measured foliations}

This is an important tool introduced by Thurston, see \cite{FLP}. It
provides a ``completion'' of the set $\mathcal S$, much like the circle is a
completion of $\mathbb Q\cup\{\infty\}$.

A {\it measured foliation} on $\Sigma$ is a foliation with finitely
many singularities equipped with a transverse measure of full support.
The singularities are standard $k$-prong singularities with $k\geq 3$
except that $k=1$ is allowed at the punctures. Each leaf is either an
arc joining two singularities or punctures (that may coincide), or an
essential circle, or an injectively immersed line or ray that starts
at a puncture or a singular point.

The {\it Whitehead equivalence} on the set of measured foliations is generated by collapsing leaves that
are arcs joining distinct singularities and isotopies. Every measured
foliation $\mu$ determines a length function $\ell_\mu:\mathcal S\to
[0,\infty)$ that sends $s\in\mathcal S$ to the infimum of measures 
over simple closed curves in the isotopy class $s$.

The set of all measured foliations up to Whitehead equivalence has a
natural topology, homeomorphic to $\R^{6g+2p-6}-\{0\}$. It is defined
by embedding the set of measured foliations in the space of length
functions $\ell:\mathcal S\to [0,\infty)$. Adding the ``empty
  foliation'' 0, one obtains the space $\mathcal {MF}$ homeomorphic to
  $\R^{6g+2p-6}$,  
  and projectivizing with respect to the action of
  $\R_+$ that scales the measure, the space $\PMF$
  homeomorphic to $S^{6g+2p-7}$.
  
Every curve $s\in\mathcal S$ also determines a length function
$\ell_s:\mathcal S\to [0,\infty)$ via $\ell_s(s')=i(s,s')$. There is a
unique equivalence class $j(s)$ of measured foliations such that
$\ell_{j(s)}=\ell_s$. 
The function $j:\mathcal S\hookrightarrow \mathcal {MF}$ is the {\it canonical inclusion}. 
The foliation $j(s)$ has all but finitely many leaves in the isotopy class $s$. 
We will usually suppress $j$ and write $\mathcal S\subset \mathcal{MF}$.   
  
The subset $\mathcal S\subset
  \mathcal {MF}$ is closed and discrete, but after projectivizing, the
  image of $\mathcal S$ in $\PMF$ is dense.   The intersection
  pairing on $\mathcal S$ extends uniquely to
  $i:\mathcal{MF}\times\mathcal{MF}\to [0,\infty)$ in such a way that
    it is continuous and $\R_+$-equivariant in each
    variable. Moreover, $i$ is symmetric and $i(\mu,s)=\ell_\mu(s)$
    when $s\in \mathcal S$, and $i(\mu,\mu)=0$ for every
    $\mu\in\mathcal{MF}$. 
The intersection pairing does
    not descend to $\PMF$; however the statement
    $i(\xi,\eta)=0$ (or $\neq 0$) makes sense for projectivized
    measured foliations $\xi,\eta$.

A measured foliation $\mu$ is {\it filling} if $i(\mu,s)>0$ for every
$s\in\mathcal S$. This is equivalent to the condition that no curve
$s\in\mathcal S$ can be homotoped into the union
of finitely many leaves. If $\mu$ is filling, the set
$\Delta(\mu)=\{[\nu]\in\PMF\mid i(\mu,\nu)=0\}$ has the
structure of a simplex \cite[Theorem 14.7.6]{katok} and consists of classes of
measures with the same underlying foliation as $\mu$ (for the latter
see \cite[Theorem 1.12]{rees}). 
The vertices
correspond to ergodic measures, and general points to convex
combinations of ergodic measures. One characterization of an ergodic
measure $\mu$ is that if it is written as the sum $\mu=\mu_1+\mu_2$ of
transverse measures then necessarily both $\mu_i$ are multiples of
$\mu$. When the simplex degenerates to a point, $\mu$ is called {\it
  uniquely ergodic}.

Thurston \cite{FLP} constructed an equivariant compactification
$\overline\calt$ of $\calt$ such that $\overline\calt -
\calt=\PMF$ and the pair $(\overline\calt,\calt)$ is
homeomorphic to the pair $(B,int B)$ where $B$ is the closed ball of
dimension $6g+2p-6$. This compactification can be described as
follows. A point $X\in\calt$, thought of as a hyperbolic surface,
determines a length function $\ell_X$ by sending $s\in\mathcal S$ to its
hyperbolic length. Then a sequence of hyperbolic surfaces converges to
the projective class of a measured foliation if the corresponding
length functions converge projectively to the length function of the
foliation. 

\subsection{Measured geodesic laminations}\label{laminations}

A {\it geodesic lamination} in a complete hyperbolic surface $\Sigma$
of finite area is a nonempty compact subset of $\Sigma \smallsetminus
P$ which is a disjoint union of geodesics (as a set).
A {\it measured geodesic lamination} $\xi$ is
a geodesic lamination equipped with a transverse measure. To an arc
$A$ transverse to the lamination and with endpoints in the complement,
$\xi$ assigns a real number $\int_A\xi\geq 0$, additive under
concatenations and invariant under isotopy of such arcs. When
$s\in\mathcal S$ we have the intersection number $i(s,\xi)$ defined as
the infimum of $\int_A\xi$ as $A$ ranges over (transverse)
parametrizations of simple closed curves in the isotopy class
$s$.
There is a natural bijection between the set $\mathcal {MF}$ of
measured foliations up to isotopy and Whitehead moves, and the set
$\mathcal {ML}$ of measured geodesic laminations.  See \cite{levitt}
and \cite[Chapter 11]{misha-book}. If a measured foliation $\mu$
corresponds to the measured lamination $\xi$ then $i(s,\mu)=i(s,\xi)$
for every $s\in\mathcal S$. A rough
description of the correspondence is as follows. Let $\ell$ be a
generic leaf of $\mu$. Its lift to the universal cover of
$\Sigma\smallsetminus P$ (which can be identified with hyperbolic
plane via a complete finite area hyperbolic metric on
$\Sigma\smallsetminus P$) is a quasi-geodesic which is bounded
distance away from a unique infinite geodesic $l$. The image of $l$ is
a generic leaf of $\xi$.

When $\xi$ is a measured geodesic lamination, denote by $|\xi|$ its
support, i.e. the union of those leaves of $\xi$ such that the measure
of any arc crossing it transversally is nonzero.

\subsection{The supporting multisurface}\label{supporting}

Consider a measured geodesic lamination $\xi$.
The support $|\xi|$ is a geodesic lamination with
finitely many components and each is minimal (i.e. every leaf is
dense), including the possibility of a simple closed
geodesic. Since we require that $|\xi|$ be compact, there are no leaves
going to punctures. Even more
generally, a geodesic lamination (possibly not the support of a
measure) consists of finitely many minimal components and finitely
many isolated leaves, each of which is either closed or in each
direction spirals towards a closed leaf or a minimal component. See
e.g. \cite[Proposition 3]{bonahon2}. The spiraling leaves cannot be in
the support of a measure since they would give rise to transverse arcs
with infinite measure.

We say that
$\xi$ or $|\xi|$ is {\it filling} if every complementary component of
$|\xi|$ is homeomorphic to an open disk or to an open once punctured
disk. Equivalently, every simple closed geodesic $\alpha$ in $\Sigma$
intersects $|\xi|$, or equivalently again, $i(\alpha,\xi)>0$, i.e. the
corresponding measured foliation is filling. If $\xi$
is filling then $|\xi|$ is connected.

Unless otherwise stated, when we talk about subsurfaces
$Y\subset\Sigma$ we mean
\begin{itemize}
\item connected and closed, as subsets of $\Sigma$,
\item no punctures on the boundary,
\item $Y\neq\Sigma$,
\item $Y$ is not a disk or a once punctured disk or a pair of pants,
  by which we mean a sphere with the total of exactly three punctures
  and boundary components,  
\item no complementary component is a disk or a punctured disk,
\item up to isotopy rel $P$.
\end{itemize}
In particular, subsurfaces are $\pi_1$-injective.

A {\it multisurface} is a nonempty disjoint union of subsurfaces that
does not contain distinct annuli which are isotopic rel $P$.

When $|\xi|$ is connected but not filling there is a unique
subsurface $Y\subset\Sigma$ that contains $|\xi|$ and
$\xi$ is filling in $Y$.
We call $Y$ the {\it supporting subsurface} of $\xi$ and denote it
$Supp(\xi)$ or $Supp(|\xi|)$. That $\xi$ fills
$Supp(|\xi|)$ means that $i(s,\xi)>0$ for every essential curve $s$ in
$Supp(|\xi|)$ not homotopic into $\partial(Supp(|\xi|))$.

We note that the supporting subsurface of a simple closed curve is an
annulus, and otherwise the supporting subsurface has negative Euler
characteristic and cannot be a pair of pants (from the point of view of
foliations this was proved in \cite[Expos\'e 6]{FLP}).

In general, when $|\xi|$ is
disconnected, the supporting subsurfaces of the components can be
isotoped so that they are pairwise disjoint. The union of these
supporting subsurfaces of the components is by definition the
supporting multisurface $Supp(\xi)$ or $Supp(|\xi|)$.
(The annuli components correspond to closed geodesics in $|\xi|$ and are pairwise not isotopic rel $P$.)

Now that we made the careful distinction, we will revert to the
standard terminology and call $Supp(\xi)$ the supporting subsurface
even when it is not connected.

The set of geodesic laminations in $\Sigma$ is a compact space with
respect to Hausdorff topology on the space of compact subsets of
$\Sigma$. The following is standard. 

\begin{prop}\label{standard} 
Suppose $\xi_n\to\xi$ is a convergent sequence of measured geodesic
laminations, and suppose that $|\xi_n|\to \lambda$ in the Hausdorff
topology. Then $|\xi|\subseteq\lambda$.
\end{prop}

\begin{proof}
Let $s$ be a curve in the complement of $Supp(\lambda)$. Then
$i(s,\xi_n)=0$ for large $n$ since $s$ is disjoint from
$Supp(\xi_n)$. It follows that $i(s,\xi)=0$, so $s$ is disjoint from
the support of $\xi$. 
\end{proof}

\subsection{The Teichm\"uller metric}

The Teichm\"uller space $\calt$ is equipped with a proper geodesic
metric which is $\MCG$-invariant. The distance between two
complex surfaces is defined 
to be $$d_\calt (X,Y)= \inf\log(K_f)$$ where $f$ ranges over all
orientation preserving homeomorphisms $X\to Y$ which are smooth except
at finitely many points, and $$K_f=\sup K_f(p)$$ is the supremum of
dilatations $K_f(p)$ over the points $p\in X$ where $f$ is
smooth (it is customary to scale this expression by $\frac 12$ but we
will ignore this). 
Recall that $K_f(p)\geq 1$ is the ratio of major to minor axes
of the ellipse obtained by taking a round circle and applying the
derivative $df_p$. Teichm\"uller proved that the infimum of $K_f$ is
realized by a unique homeomorphism, called the {\it Teichm\"uller
  map}.
The Teichm\"uller metric is proper and
  $\MCG$-invariant.

\subsection{Holomorphic quadratic differentials} \label{subsec:quadratic-diff}

The cotangent space of $\calt$ at $X\in\calt$ is the space of
(holomorphic) quadratic differentials on $X$, each of which is defined
in charts as $q(z)=f(z)dz^2$ with $f$ holomorphic, possibly with simple
poles at the punctures. See e.g. \cite{it} for basic facts about
quadratic differentials.

A nonzero quadratic differential $q$ on $X\in\calt$ determines two
measured foliations, horizontal $q^H$ and vertical $q^V$. Away from
the singularities, i.e. points where $q$ has a zero or a pole,
there are charts where $q=dz^2$, and then the
vertical foliation is defined by the vertical lines and with
transverse measure $|dx|$, and similarly for the horizontal
foliation. When $s$ is a curve, we will say its {\it horizontal
  length} is $i(s,q^V)$, the measure assigned to $s$ by the {\it
  vertical} foliation, and we similarly define the vertical length of
$s$. 
The quadratic differential $q$ also determines a Euclidean
metric on $\Sigma$ with cone type singularities: on a chart where
$q=dz^2$ the metric is Euclidean. We will denote this metric by
$l_q$.

The
norm of $q$ is $||q||=\int_X |q|$, i.e. it equals the area of
$X$ with respect to the Euclidean metric. We denote by $QD(X)$ the
vector space of all quadratic differentials on $X$ and by $QD^1(X)$
the subset of unit norm quadratic differentials. The following fact
can be proved using the compactness of the unit area quadratic
differentials on $X$. 

\begin{lemma}\label{bdd below}
  Let $X$ be a hyperbolic surface. There is $\epsilon_X>0$ such that
  for every $q\in QD^1(X)$ and every curve $s$ we have
  $l_q(s)\geq\epsilon_X$. In particular, either the horizontal or the
  vertical length of $s$ is $\geq \epsilon_X/2$.
\end{lemma}

Geodesics in $\calt$ (i.e. {\it Teich\-m\"ul\-ler geodesics})
have a simple description in terms of quadratic differentials.  If
$X\in\calt$ and $q$ is a unit norm quadratic differential on $X$, for
$t\in\R$ define $X_t\in\calt$ by the rule that on a chart of $X$ where
$q=dz^2$, the chart for $X_t$ is $x+iy\mapsto
e^{t/2}x+ie^{-t/2}y$. Then $t\mapsto X_t$ is a geodesic line
determined by $X$ and $q$ and it is parametrized with unit
speed. The identity map
$X\to X_t$ is the Teichm\"uller map for these two points in $\calt$. 
There is a natural quadratic differential $q_t$ on $X_t$ given by
$q_t=dz^2$ in the new charts, and the Teichm\"uller geodesic defined
by $(X_{t_0},q_{t_0})$ is the same as the one defined by $(X,q)$ except for
the reparametrization $t\mapsto t+t_0$. By definition we have
$$q_t^H=e^{t/2}q^H\qquad\mbox{and}\qquad q_t^V=e^{-t/2}q^V.$$

We have a map from the cone 
$$\widetilde {QD}^1(X)=QD^1(X)\times [0,\infty)/QD^1(X)\times \{0\}$$
to $\calt$ given by $(q,t)\mapsto X_t$ with $X_t$ described above.

\begin{theorem}[Teichm\"uller's contractibility theorem]\label{teich cont}
This map is a homeomorphism.
\end{theorem}

For a proof see e.g. \cite[Theorem 7.2.1]{hubbard-book}. A consequence
of the theorem is that any two points in $\calt$ are joined by a
unique Teichm\"uller geodesic.

\subsection{Modulus and the Collar Lemma}\label{collar}
The interior of any closed complex annulus $A$ is conformally
equivalent (or biholomorphic) to the unique flat annulus $S^1\times
(0,2m\pi)$ where $S^1$ is the standard circle with length $2\pi$.  The
number $m>0$ is the {\it modulus} of $A$. See Ahlfors' book
\cite{ahlfors} for the classical theory. In particular, there is the
following equivalent definition of the modulus that depends only on
the conformal structure, see \cite[Chapter I.D, Example 2]{ahlfors}:

$$Mod(A)=\sup_{\rho}\inf_\lambda \frac{\ell_\rho(\lambda)^2}{Area_\rho(A)}$$
where $\rho$ runs over all conformally equivalent metrics, and $\lambda$
over all arcs connecting the two boundary components. In particular,
if $A$ is an annulus in a complete hyperbolic surface $X$ of finite area
and with the fixed underlying surface $\Sigma$,
and if the distance between the boundary components is large, then the
modulus of $A$ is large. This is because we can take the hyperbolic
metric for $\rho$ and then the area of $A$ is bounded by the area of
$X$, which in turn is bounded by the topology of $\Sigma$.

The following lemma is fundamental for the geometry of hyperbolic
surfaces. See \cite[Lemma 13.6]{FM}, which shows that the distance
between the boundary components is large.

\begin{lemma}[The Collar Lemma or the Margulis Lemma] 
 \label{lem:collar}
There are functions $F,G:(0,\infty)\to (0,\infty)$ such that 
$$\lim_{t\to 0}F(t)=\lim_{t\to 0}G(t)=\infty$$
and such that the following holds.
If a hyperbolic surface has a simple closed geodesic of length $<t$,
then its $F(t)$-neighborhood is an embedded annulus whose modulus is
$\geq G(t)$.
\end{lemma}

\section{Projections}\label{s:proj}

\subsection{Curve complex; arc and curve complex}

The {\it curve complex} $\mathcal C(\Sigma)$ is the simplicial complex
whose vertex set is the set of curves in $\mathcal S$, and simplices
are collections of curves that have pairwise disjoint
representatives. When $6g-6+2p>2$ this complex is connected.

The following is a celebrated theorem of Masur and Minsky.
It was shown recently that $\delta$ can be taken to be uniform
(e.g. $\delta=17$) for all surfaces \cite{Aou,Bow,CRS,HPW}.
For the purpose of this paper, this is however, not important.

\begin{thm}[\cite{MM}]
The
1-skeleton of $\mathcal C(\Sigma)$ is a $\delta$-hyperbolic graph
whenever it is connected.
\end{thm}

It will be more convenient to work with
the {\it arc and curve complex} $\mathcal{AC}(\Sigma)$. Its vertices
are represented by (essential) arcs and curves. By an {\it arc} we
mean a path in $\Sigma$ whose interior points are in
$\Sigma\smallsetminus P$ and
whose boundary is in $P$, and it is embedded except possibly at the
endpoints. Two arcs are equivalent if they are homotopic
through arcs. An arc is essential if it is not homotopic through arcs to a
small neighborhood of a puncture. A simplex in $\mathcal{AC}(\Sigma)$
is a collection of arcs and curves that have disjoint representatives,
except possibly for the endpoints of arcs.

The complex $\mathcal{AC}(\Sigma)$ is connected and
$\delta$-hyperbolic as soon as $6g-6+2p>0$. When $6g-6+2p>2$ the
natural inclusion
$$\mathcal C(\Sigma)\hookrightarrow \mathcal {AC}(\Sigma)$$
is a quasi-isometry. The inverse is constructed by sending an arc
$\alpha$ to
an essential component of the boundary of the regular neighborhood of
$\alpha$, see \cite[Theorem 1.3]{arc-curve}. 

When $6g-6+2p=2$ (i.e. when $(\Sigma,P)$ is the once-punctured torus
or the four times punctured sphere) the complex $\mathcal
{AC}(\Sigma)$ is quasiisometric to the Farey graph (hence also
hyperbolic), while $\mathcal C(\Sigma)$ is an infinite discrete space.

If $\alpha,\beta$ are two arcs or curves, their intersection number
$i(\alpha,\beta)$ is the smallest cardinality of the intersection of
their representatives, not counting the punctures. 

We have the following useful estimate on the distance in $\mathcal
C(\Sigma)$ and $\mathcal
{AC}(\Sigma)$.

\begin{prop}\label{i}
\begin{itemize}
\item
 If $\alpha,\beta$ are curves and $6g-6+2p>2$ then $d_{\mathcal
  C(\Sigma)}(\alpha,\beta)\leq i(\alpha,\beta)+1$.
\item If $\alpha,\beta$ are arcs or curves and $P\neq\emptyset$, then
  $d_{\mathcal {AC}(\Sigma)}(\alpha,\beta)\leq i(\alpha,\beta)+2$.
\end{itemize}
\end{prop}

\begin{proof}
The first claim is well known; it can be easily proved by induction on
the intersection number using surgery. See \cite[Lemma 1.1]{Bow2}.
There are also logarithmic bounds,
see~\cite{hempel}. 

The second claim can be proved similarly. E.g. see~\cite[Definition
  3.1 and Remark 3.2]{HPW} for the case when $\alpha,\beta$ are arcs, when $d_{\mathcal
  {AC}(\Sigma)}(\alpha,\beta)\leq i(\alpha,\beta)+1$. 
If $\alpha$ is a curve and $\beta$ an arc, we can
construct an arc $\alpha'$ disjoint from $\alpha$ and with
$i(\alpha',\beta)\leq i(\alpha,\beta)$. Then we have
$d_{\mathcal {AC}(\Sigma)}(\alpha,\beta)\leq d_{\mathcal
  {AC}(\Sigma)}(\alpha',\beta)+1\leq i(\alpha,\beta)+2$.
\end{proof}

When $\Sigma$ is a 3 times punctured sphere, the complex $\mathcal
{AC}(\Sigma)$ is finite, and is not useful when considering subsurface
projections.

\subsection{Curve complex of the annulus}

When $A$ is an annulus, we define $\mathcal C(A)=\mathcal {AC}(A)$ to
be the graph whose vertices are embedded arcs with endpoints on
distinct boundary components of $A$, modulo isotopy rel boundary, and
edges correspond to disjointness. Thus $\mathcal C(A)$ is
quasi-isometric to $\Z$.

\subsection{The Gromov boundary}\label{boundary}

Klarreich~\cite{klarreich} gave a description of the Gromov boundary of the curve
complex $\mathcal C(\Sigma)$ (or equivalently of $\mathcal
{AC}(\Sigma)$). A point in $\partial \mathcal C(\Sigma)$ is
represented by a filling measured geodesic lamination $\xi$ and two
such laminations $\xi,\xi'$ represent the same point if $|\xi|=|\xi'|$
(see Section~\ref{laminations}). In other words, a point in $\partial
\mathcal C(\Sigma)$ is a filling geodesic lamination that admits a
transverse measure of full support.

We now state Klarreich's work in more detail.
First recall  that if $x_n$ is
a sequence in a $\delta$-hyperbolic space $X$ then after passing to a subsequence
one of the following occurs:
\begin{itemize}
\item $x_n\to z\in\partial X$, or
\item there is some $x\in X$ so that $x_n$ {\it coarsely rotates}
  around $x$. This means that for any $n$ there is $m_0$ so that for
  $m>m_0$ any geodesic $[x_n,x_m]$ 
  passes within a uniform distance (e.g. $10\delta$) from $x$. 
\end{itemize}

This statement is really an exercise in Gromov products (e.g. the
reader should contemplate the case of a locally infinite tree). A more
sophisticated approach is via the horofunction boundary, see
e.g. \cite[Section 3]{maher-tiozzo}.

The theorem of Klarreich can now be summarized as follows.

\begin{thm}[Klarreich]\label{klarreich}
There is a coarse map $\pi:\overline\calt\to\mathcal C(\Sigma)\cup\partial\mathcal C(\Sigma)$
with the following properties.
\begin{enumerate}[(1)]
\item Suppose $x_n\in \overline\calt$, $x_n\to x\in
  \overline\calt$. If $\pi(x)\in\partial\mathcal C(\Sigma)$ then
  $\pi(x_n)\to\pi(x)$. If $\pi(x)\in\mathcal C(\Sigma)$ then $\pi(x_n)$
  coarsely rotates around $\pi(x)$.
\item If $X\in\calt$ then $\pi(X)$ is the collection of shortest
  curves on $X$ (or equivalently, collection of curves of length less
  than a suitable constant). If $\mu\in\PMF$ is not filling,
  $\pi(\mu)$ consists of boundary components of the supporting
  multisurface. If $\mu$ is filling then $\pi(\mu)\in\partial \mathcal
  C(\Sigma)$. Moreover, for every $b\in\partial\mathcal C(\Sigma)$ the
  preimage $\pi^{-1}(b)$ is nonempty and consists of the simplex of
  projectivized transverse measures on the same underlying
  foliation. In particular, if $\mu$ is uniquely ergodic, the preimage
  of $\pi(\mu)$ is a single point.
\end{enumerate}
\end{thm}

If $A$ is an annulus, its curve complex is quasi-isometric to $\Z$ and
the Gromov boundary has two points. We can think of them as the
two ways in which a geodesic can spiral in and out of the annulus
represented by a regular neighborhood of a simple closed geodesic.

\subsection{Subsurface projections}\label{ss:subsurface-projections}
This key concept was introduced by Masur and Minsky \cite{mm2}. Recall
our convention about subsurfaces from Section \ref{supporting}. In
particular, they are proper and connected. When a subsurface $Y$ is
not an annulus, we define its arc and curve complex $\mathcal
{AC}(Y)$ as $\mathcal
{AC}(\hat Y)$, where $\hat Y$ is obtained from $Y$ by collapsing each
boundary component to a puncture. Thus any essential arc with boundary
in $\partial Y$ represents a point in $\mathcal
{AC}(Y)$.
The complex $\mathcal {AC}(Y)$ is
always $\delta$-hyperbolic and of infinite diameter since $Y$ is not
allowed to be a pair of pants.

Let
$Y\subset\Sigma$ be a connected subsurface different from a pair of
pants. Fix a complete hyperbolic metric of finite area on $\Sigma
\smallsetminus P$ and
realize all nonperipheral boundary components of $Y$ by geodesics. If no two
boundary components of $Y$ are parallel, then $Y$ is realized as a
totally geodesic subsurface.

Let $\alpha$ be an arc or a curve in $\Sigma$, not isotopic into the
complement of $Y$, realized as a geodesic. The intersection $Y\cap
\alpha$ is a curve or a collection of arcs. We define
$\pi_Y(\alpha)\subset \mathcal{AC}(Y)$ to be this
intersection. This is a collection of points in $\mathcal{AC}(Y)$ at
pairwise distance $\leq 1$, so coarsely the projection is
well-defined.

If $Y$ has parallel boundary components but is not an annulus
(i.e. when a complementary component is an annulus) consider the
covering space $\Sigma_Y\to\Sigma$ corresponding to $Y\subset
\Sigma$. The subsurface $Y$ lifts to $\Sigma_Y$ and there is a unique
representative, up to isotopy, 
which is totally geodesic, and we will
identify it with $Y$. The entire covering space $\Sigma_Y$ is obtained
from $Y$ by attaching half-open annuli to the boundary
components. Each annulus is of the form $H/\Z$, where $H$ is the
hyperbolic half-plane and $\Z$ acts by translation along the
boundary. The Gromov compactification of $H/\Z$ is a (compact)
annulus, and attaching these annuli to $Y$ produces a surface
$\overline\Sigma_Y$ 
homeomorphic to $Y$, with homeomorphism being canonical up to
isotopy. Now define $\pi_Y(\alpha)\subset \mathcal{AC}(Y)$ as the
intersection of the preimage of $\alpha$ in $\Sigma_Y$ with
$Y$. Equivalently, identifying $Y$ with $\overline\Sigma_Y$, take the
closure of the preimage of $\alpha$, and discard the inessential
components. The resulting finite collection of arcs (or a curve) is
the projection. 

When $Y$ is the annulus, it is the latter description of the
projection that generalizes. Namely, $\overline\Sigma_Y$ is an
annulus. Again take the
closure of the preimage of $\alpha$, and discard the inessential
components to get the projection.

We make the same definition when $\alpha$ is a collection of pairwise
disjoint arcs or curves and at least one is not isotopic into the
complement of $Y$.

Subsurface projections can also be defined 
for other subsurfaces and for geodesic laminations.

If $Y'\subset \Sigma$ is another subsurface, define
$$\pi_Y(Y')=\pi_Y(\partial Y')$$
assuming the latter is defined; otherwise $\pi_Y(Y')$ is undefined.

\subsection{Projecting geodesic laminations}\label{ss:projecting-laminations} 
We now define $\pi_Y(\xi)=\pi_Y(|\xi|)$ when $\xi$ is a measured
geodesic lamination with support $|\xi|$. As suggested by the
notation, it will depend only on the support, and it will be defined
whenever $Y\cap Supp(\xi)\neq\emptyset$ (even after isotopy). 
If $Y$ is a component of $Supp(\xi)$ and it is not an annulus we
define $\pi_Y(\xi)$ to be the point at infinity of $\mathcal
{AC}(Y)$ represented by $|\xi|$.

If $Y$ is an annulus component of $Supp(\xi)$ we define $\pi_Y(\xi)$
to be the two points at infinity in the curve complex. 
     
Now suppose that $Y$ is not a component of $Supp(\xi)$. First assume
that $Y$ is realized as a totally geodesic subsurface of
$\Sigma$. Then $|\xi|\cap Y$ is the union of a collection of arcs (typically
uncountably many, but there are only finitely many isotopy classes)
and the underlying set of a measured geodesic lamination $\mu$. Define
$\pi_Y(\xi)$ as the set of these arcs and boundary components of
$Supp(\mu)$ that are not boundary components of $Y$.

More generally, if $Y$ is not an annulus, we lift to the cover
$\Sigma_Y$ and intersect with the totally geodesic copy of $Y$.

Finally, if $Y$ is an annulus, it is crossed by some leaves of
$\xi$. Lift those leaves to $\Sigma_Y$ and take their closure in
$\overline \Sigma_Y$ to get $\pi_Y(\xi)$.

Note that the set $\PMF(Y)\subset \PMF$ consisting of measured
geodesic laminations $\xi$ such that $\pi_Y(\xi)$ is defined is
open. This follows from Proposition \ref{standard}.

\subsection{Projection distance}\label{ss:proj-dist}
Let $\alpha,\beta$ be two curves, or arcs, or subsurfaces, or Riemann
surfaces, or measured
foliations, so that the projections $\pi_Y(\alpha)$ and $\pi_Y(\beta)$
to a subsurface $Y$ are defined. Then define the {\it projection
  distance}
$$d^\pi_Y(\alpha,\beta)=\diam (\pi_Y(\alpha)\cup\pi_Y(\beta))$$
When we use this notation at most one of $\pi_Y(\alpha),\pi_Y(\beta)$ will
represent a point (or points) at infinity, and in this case the
projection distance is infinite. In all other cases,
since the diameter of $\pi_Y(\alpha)$ is uniformly bounded, the
projection distance is a
well-defined finite number.

The following triangle inequality is obvious.

\begin{prop}\label{triangle}
If $\pi_Y(\alpha),\pi_Y(\beta),\pi_Y(\gamma)$ are all defined, then
$$d_Y^\pi(\alpha,\beta)+d_Y^\pi(\beta,\gamma)\geq
d_Y^\pi(\alpha,\gamma)$$
\end{prop}

The following key inequality was proved by Behrstock. We say that two
subsurfaces $Y,Y'$ {\it overlap} if $\partial Y\cap \partial
Y'\neq\emptyset$. 

\begin{prop}[{\cite[Theorem 4.3]{jason}}]\label{behrstock} 
There is a constant $C$ such that the following holds.  Let $Y,Z$ be
two overlapping subsurfaces of $\Sigma$
(i.e. $\dd Y \cap \dd Z \neq \emptyset$, even after
isotopy)
and let $\alpha$ be a collection of pairwise disjoint arcs and curves. Assume
$\pi_Y(\alpha)$ and $\pi_Z(\alpha)$ are defined. Then
$$d^\pi_Y(\alpha,\partial Z)\geq C\implies d^\pi_Z(\alpha,\partial
Y)\leq C$$

The same statement holds when $\alpha$ is replaced by a foliation.
\end{prop}

Leininger supplied explicit constant $C=10$ and a simple argument, see
\cite[Lemma 2.13]{johanna2}. When $\alpha$ is replaced by a foliation
$\xi$ the proof is an easy consequence, as follows:
\begin{itemize}
\item Leininger's proof works with no change if some leaf of $\xi$
  intersects $\partial Y$ or $\partial Z$ (in particular, this occurs
  if $\xi$ is filling).
\item If $\xi$ is not filling let $\alpha=\partial Supp(\xi)$. If
  $\pi_Y(\alpha)$ and $\pi_Z(\alpha)$ are defined, the claim about
  $\xi$ follows (after increasing $C$ by 1) after observing that
  $d^\pi_Y(\alpha,\xi)\leq 1$ and $d^\pi_Z(\alpha,\xi)\leq 1$.
\item If $Y$ is a component of $Supp(\xi)$, then $d^\pi_Z(\xi,\partial
  Y)\leq 1$, and similarly for $Z$.
\end{itemize}

The following was first proved in \cite{mm2}. A streamlined proof with
the explicit bound is in \cite[Lemma 5.3]{bbf}.

\begin{prop}\label{finiteness}
For any two subsurfaces $Y,Z$ there are only finitely many subsurfaces
$W$ such that $d_W^\pi(Y,Z)>3$.
\end{prop}

\subsection{The Bounded Geodesic Image Theorem}

This theorem was proved by Masur and Minsky \cite{mm2}. A more
combinatorial proof with a uniform bound on $M$ was given by Webb
\cite{webb}. 

\begin{thm}\label{bgit}
There exists $M=M(\Sigma)$ such that the following holds. Let
$Y\subset \Sigma$ be a subsurface and $g=x_0,x_1,\cdots,x_n$ a
geodesic in $\mathcal C(\Sigma)$ such that $\pi_Y(x_i)$ is defined for
all $i$. Then $d^\pi_Y(x_0,x_n)\leq M$.
\end{thm}

\begin{prop}\label{continuity}
There is a constant $N=N(\Sigma)$ so that the following holds.
Suppose $\pi_Y(\nu)$ is defined and let $\alpha\in \mathcal{AC}(Y)$
and $\Theta\in [0,\infty)$ such that
  $d^\pi_Y(\alpha,\nu)\geq\Theta$. Then there is a neighborhood $U$ of
  $\nu$ in $\PMF$ such that $\pi_Y(\mu)$ is defined and in addition
$$d^\pi_Y(\alpha,\mu)> \Theta-N$$
for all $\mu\in U$.
\end{prop}

\begin{proof} 
We already noted that $\pi_Y$ is defined on an open subset of
$\PMF$. Let $\mu_i\to\nu$ in $\PMF$.

We first consider the case when $Y$ is a component of $Supp(\nu)$. If
$Y$ is also a component of $Supp(\mu_i)$ there is nothing to
prove. Otherwise, the projection $\pi_Y(\mu_i)$ consists of a
collection of curves (coming from boundary components of the support
of $\mu_i$) and of a collection of arcs.  Denote by $\lambda_i$ either
one of the curves in the collection, or one of the essential
nonperipheral boundary components of $I\cup\partial Y$ where $I$ is
one of the arcs in the collection.  We view $\lambda_i$ as a measured
lamination. and we may assume $\lambda_i\to\lambda\in\PMF$. Then
$\lambda$ is supported in $Y$ and satisfies $i(\lambda,\nu)=0$. Since
$\pi_Y(\nu)$ fills $Y$ it follows that $|\lambda|=|\pi_Y(\nu)|$. By
Theorem \ref{klarreich} it follows that
$d^\pi_Y(\alpha,\lambda_i)\to\infty$, so we are done in this case
since $d^\pi_Y(\lambda_i,\mu_i)\leq 5$, see Proposition \ref{i}.  When
$Y$ is an annulus we cannot use Theorem \ref{klarreich}, but argue
directly as follows. Let $\alpha$ be the closed geodesic homotopic
into $Y$. If $\mu_i$ has $\alpha$ as a closed leaf then
$d^\pi_Y(\alpha,\mu_i)=\infty$ so we are done. Otherwise, $\mu_i$ has
leaves that intersect $\alpha$ at a small angle by Proposition
\ref{standard} and then again the projections go to infinity.  

Second, consider the case when a leaf $\ell$ of $|\nu|$ intersects $\partial
Y$. By Proposition \ref{standard} there is a leaf $\ell_i$ of $\mu_i$ for large
$i$ that also intersects $\partial Y$ and an arc component of
$\ell_i\cap Y$ is isotopic to an arc component of $\ell\cap Y$. Thus
in this case $d^\pi_Y(\mu_i,\nu)\leq 2$ (if $Y$ is an annulus we only
get $d^\pi_Y(\mu_i,\nu)\leq 3$), so here $N>3$ works.

Finally, suppose that a component $Z$ of $Supp(\nu)$ is isotopic into
$Y$, but is not $Y$. By the same argument as in the first case, we see
that $\pi_Z(\mu_i)$ go to infinity in $\mathcal {AC}(Z)$ (again argue
directly if $Z$ is an annulus). After
passing to a subsequence, we may assume that for $i<j$ we have
that $d^\pi_Z(\mu_i,\mu_j)$ is large. By the Bounded Geodesic Image
Theorem \ref{bgit}, we see that any geodesic joining $\pi_Y(\mu_i)$
and $\pi_Y(\mu_j)$ must contain a curve disjoint from $Z$, 
      and so is
within distance 1 from $\pi_Y(\nu)$. Then by
$\delta$-hyperbolic geometry we deduce that
$d^\pi_Y(\alpha,\mu_i)<\Theta-1-\delta$ 
for at most one $i$. So
$N>1+\delta$ works in this case.
\end{proof}

\subsection{Partitioning the subsurfaces and the color preserving subgroup}
\label{color}

We will need the following fact.

\begin{prop}[{\cite[Proposition 5.8]{bbf}}]
The set of subsurfaces of $\Sigma$ which are not pairs of pants can be
written as a finite disjoint union
$${\bfY^1}\sqcup{\bfY^2}\sqcup\cdots\sqcup {\bfY^k}$$
so that any two subsurfaces in any $\bfY^i$ overlap, and there is a
subgroup $G<\MCG$ of finite index that preserves each ${\bfY^i}$.
\end{prop}

The subgroup $G$ is the {\it color preserving subgroup}. We can
further arrange that each $\bfY^i$ is a $G$-orbit. We remark that
unlike in~\cite{bbf}, $\{ \Sigma \}$ is not one of the $\bfY^i$ since
we are considering only proper subsurfaces.

\subsection{Large intersection number implies large projection}

Notice that one can have two curves with large intersection number
that are at distance 2 in the curve complex. Thus the literal converse to
Proposition \ref{i} does not hold.  
The following is the correct converse to Proposition \ref{i}.

\begin{lemma}\label{subsurface}\cite[Theorem 1.5]{watanabe}

  For every $M$ there is
  $I$ such that the following holds. If $\alpha,\beta\in\mathcal S$
  and $i(\alpha,\beta)\geq I$ then there is a subsurface $Y\subset
  \Sigma$ such that $d^\pi_Y(\alpha,\beta)\geq M$
  (where we also allow $Y=\Sigma$).

The same statement holds for pairs of arcs-or-curves in $\mathcal
{AC}(\Sigma)$. 
\end{lemma}

\section{Verification of the Flow Axioms}\label{s:2}

Let $c$ be a geodesic (segment, ray or line) in a metric space $X$. 
We write $\rho_c$ for the {\it nearest point projection} from $X$ to $\Image(c)$.

\begin{definition}
The geodesic $c$ is {\it $D$-contracting} for $D\geq 0$ if for every
metric ball $B\subset X$ with $B\cap \Image(c)=\emptyset$ the set 
$$\rho_c(B):=\cup_{b\in B} \rho_c(b)\subset \Image(c)$$ 
has diameter $\leq D$.
\end{definition}

For example, if $X$ is $\delta$-hyperbolic then every geodesic is
$10\delta$-con\-tracting. It is an interesting phenomenon that many
spaces of interest, even though they are not hyperbolic, contain many
contracting geodesics. These are thought of as ``hyperbolic
directions''. Note that a line in $\R^2$ is not contracting.

\begin{definition} A geodesic (segment, ray or line) $c$ in a metric
  space is {\it Morse} if for every $A,B$ there is $N=N(A,B)$ so that
  any $(A,B)$-quasi-geodesic $c'$ with endpoints in $c$ is contained
  in the $N$-neighborhood $B_N(c)$ of $c$.
\end{definition}

The following lemma is well known. The first part states that contracting
geo\-desics are Morse, and the second is a variant for geodesic lines.

\begin{lemma} \label{morse}
\begin{enumerate}[(i)]
\item For every $D,A,B$ there is $N=N(D,A,B)$ such that every
  $(A,B)$-quasi-geodesic with endpoints on a $D$-con\-tracting geodesic 
  $c$ is contained in the $N$-neighborhood of $c$.
\item 
Suppose $c$ is a contracting geodesic line and $c'$ is a
geodesic line contained in some neighborhood $B_M(c)$ of $c$. Then
$c'$ is contained in the $R$-neighborhood of $c$, where $R$ is a
function of the contracting constant $D$.
\end{enumerate}
\end{lemma}

For the first statement, see e.g. \cite{arzhantseva}. The idea is that
if $c'$ contains a long segment outside a big neighborhood
of $c$, then covering this segment with slightly overlapping big balls
that miss $c$ and projecting we see that the projection of the segment
to $c$ has much smaller length, and this contradicts the assumption
that $c'$ is a quasi-geodesic. The second statement can be proved
similarly.

We will need the following facts about Teichm\"uller geodesics. Recall
that $\pi:\calt\to\mathcal C(\Sigma)$ is the coarse projection to the curve
complex. 

\begin{prop}\label{basic facts}
\begin{enumerate}[(1)]
\item (Arzela-Ascoli) Any sequence of Teichm\"uller geodesics that
  intersect a fixed compact set has a subsequence that (after
  re\-para\-metrization via translation) converges to a Teich\-m\"uller
  geodesic.  
\item If $c$ is a Teichm\"uller geodesic (segment, ray or line) in the
  thick part $\calt_{\geq\epsilon}$ then
\begin{itemize}
\item $\pi c$ is a $(K,L)$-quasi-geodesic in
  $\mathcal C(\Sigma)$, where $K,L$ depend only on $\epsilon$. For a much more
  general statement see \cite{Rafi-Combinatorical-model}.
\item (Minsky~\cite{minsky-contracting}) $c$
  is $D$-contracting for
  $D=D(\epsilon)$. 
\item (the Masur Criterion \cite{masur_criterion}) $c(t)$ converges as $t\to\infty$ to
  $c(\infty)\in\PMF$ which is filling and uniquely ergodic and equals
  the vertical foliation of $c$.
\end{itemize}
\end{enumerate}
\end{prop}

In particular, note that as a consequence of Theorem \ref{klarreich}
and the Masur Criterion, if $c,d$ are two Teichm\"uller rays in a
thick part such that their images $\pi c$ and $\pi d$ fellow travel in
$\mathcal C(\Sigma)$
then the vertical foliations of $c$ and $d$ determine the same point
in $\PMF$, and this point is $c(\infty)=d(\infty)$.

\begin{lemma}\label{above}
Suppose $[Z_n,Y_n]$ are Teichm\"uller geodesic segments that converge
to a Teichm\"uller ray $c$, so that $Z_n\to c(0)$. If $c$ and all
$[Z_n,Y_n]$ are in a
thick part $\calt_{\geq\epsilon}$ then $Y_n\to c(\infty)$.
\end{lemma}

\begin{proof}
   If this fails, then, after a subsequence, $Y_n\to y\neq c(\infty)$. 
   Theorem~\ref{klarreich} and the Masur criterion imply that $\pi(y) \neq \pi(c(\infty))$ as well. 
   There are two cases.

   Suppose first $\pi(y)\in\partial\mathcal C(\Sigma)$. 
   By Proposition~\ref{basic facts} $\pi c$ is a quasi-geodesic.
   Choose a quasi-geodesic from $\pi(c(0))$ to $\pi(y)$.
   These two quasi-geodesics start at the same point and go to distinct boundary points, so they coarsely form a tripod, with the center point $w\in\mathcal C(\Sigma)$, say.
   By Proposition~\ref{basic facts}, we can choose $T>>0$ such that any Teich\-m\"uller geodesic segment of length $T$ in $\calt_{\geq\epsilon}$ projects in $\mathcal C(\Sigma)$ with endpoints at distance $>>d(\pi c(0),w)$. 
   For large $n$ choose $W_n \in [Z_n,Y_n]$ at distance $T$ from $Z_n$. 
   Then $\pi(Y_n) \to \pi(y) \in \dd \calc(\Sigma)$ implies that $\pi(W_n)$ is in a bounded neighborhood of a quasi-geodesic from $w$ to $\pi(y)$. 
   But $W_n\to c(T)$, so for large $n$, $\pi(W_n)$ is in a bounded neighborhood of a quasi-geodesic from $w$ to $\pi(c(\infty))$. 
   This is impossible since $\pi(W_n)$ is also far from $w$.
   
   The other case is that $\pi(y)\in\mathcal C(\Sigma)$. 
   Then Theorem~\ref{klarreich} implies that $\pi(Y_n)$ coarsely rotate around $\pi(y)$. 
   Again choose $W_n \in [Z_n,Y_n]$ at a fixed distance $T$ from $Z_n$ so that $\pi(W_n)$ is much further from $\pi(Z_n)$ than $\pi(y)$. 
   Thus $\pi(W_n)$ also coarsely rotates around $\pi(y)$.
   But on the other hand $W_n \to c(T)$, so $\pi(W_n)$ will stay in an $R$-neighborhood of $\pi(c(T))$.
   Here $R$ is independent of $T$, so for large $T$, the $\pi(W_n)$ cannot coarsely rotate.
   Contradiction.  
\end{proof}

We can now prove that the collection of Teichm\"uller rays that stay in a fixed thick part $\calt_{\geq \e}$ satisfies our flow axioms (F1) to (F3).

\begin{proof}[Proof of (F1)]

Let $c_n$ be Teichm\"uller rays in $\calt_{\geq \e}$.
Let $Y_n := c_n(0)$ and by assumption there are $Z_n \in c_n$ that stay in a bounded subset. 
Let $X_n \in \calt$ such that $d(X_n,Y_n)$ is bounded.
We need to show $Y_n \to \xi \in \PMF$ if and only if $X_n \to \xi \in \PMF$.
In both cases $d(Z_n,Y_n), d(Z_n,X_n) \to \infty$.
We proceed by contradiction.
As $\PMF$ is compact, we can, after a subsequence, assume $Y_n \to \xi$ and $X_n \to \xi'$ with $\xi \neq \xi'$.
Passing to a further subsequence, we may assume that $[Z_n,Y_n]$ converges to a geodesic ray $c$ in $\calt_{\geq\epsilon}$. 
Likewise we may assume that $[Z_n,X_n]$ converges to a geodesic ray $c'$ with $c'(0)=c(0)$. 
By Proposition~\ref{basic facts} the $[Z_n,Y_n]$ are $D$-contracting.
Lemma~\ref{morse}(i) implies that $[Z_n,X_n]$ is contained in a fixed neighborhood of $[Z_n,Y_n]$.
Thus $c'$ is in a neighborhood of $c$ and the projections of $c$ and $c'$ fellow travel and therefore converge to the same point in $\partial \mathcal C(\Sigma)$.
By unique ergodicity (Theorem~\ref{klarreich} and Proposition~\ref{basic facts}) it follows that $c(\infty)=c'(\infty)$.
Lemma~\ref{above} now implies that both $Y_n$
and $X_n$ converge to this point. 
\end{proof}

\begin{proof}[Proof of (F2)]
Let $X \in \calt$, $\xi_+ \in \PMF$. 
Let $c_0$ be the Teichm\"uller ray starting in $X$ with vertical foliation $\xi_+$.
By Theorem~\ref{teich cont} the map that associates to $\xi \in \PMF$ the evaluation at $t$ of the geodesic ray starting in $X$ with vertical foliation $\xi$ is continuous. 
Thus for a given $t$ there is a neighborhood $U_+$ of $\xi_+$ such that $d(c'(t),c_0(t))<1$
whenever $c'(0) = X$ and the vertical foliation of $c'$ is in  $U_+$. 
Let $c$ be a Teichm\"uller ray in $\calt_{\geq \e}$ with $d(c(0),X) < \rho$ and $c(\infty) \in U_+$.
By the Masur criterion (Proposition~\ref{basic facts}), $c(n) \to c(\infty)$.
We will show that $d(c(t),c_0(t))$ is bounded, where the bound depends only on $\e$ and $\rho$.
By Proposition~\ref{basic facts} the ray $c$ is contracting.
By Lemma~\ref{morse}(i), the geodesic segments $[X,c(n)]$ stay in a bounded neighborhood of $c$.
Let $c'$ be the limiting ray of (a subsequence of) the $[X,c(n)]$.
The $[X,c(n)]$ and $c'$ stay in the thick part (as they are in a bounded neighborhood of $\calt_{\geq \e}$).  
Thus Lemma~\ref{above} applies and $c(n) \to c'(\infty)$.
Therefore $c'(\infty) = c(\infty)$.
By the Masur criterion (Proposition~\ref{basic facts}) the vertical foliation of $c'$ is $c'(\infty) = c(\infty)$ and therefore in $U_+$.
Thus $d(c'(t),c_0(t)) < 1$.
As $c'$ stays in a bounded neighborhood of $c$, $d(c(t), c'(t))$ is bounded as well.
Thus $d(c(t),c_0(t))$ is bounded.
\end{proof}

\begin{remark}
The question when Teichm\"uller rays are parallel, or asymptotic, is
well understood. See
\cite{Eskin-Mirzakhani-Rafi-Counting-in-Orbits,ivanov,lenzhen-masur,masur-thesis,rafi-hyperbolicity}.
\end{remark}

\begin{proof}[Proof of (F3)]
Let $X\in T_{K,\rho}(\xi_-,\xi_+)$. Choose a sequence $c_n$ as in the
definition. After a subsequence and reparametrization, $c_n\to
c$. Thus $c$ is a Teichm\"uller geodesic in $\calt_{\geq\epsilon}$
that passes within $\rho$ of $X$. It remains to prove that if $c,c'$
are two geodesics in the thick part and $c(\pm\infty)=c'(\pm\infty)$,
then $c,c'$ are in uniform neighborhoods of each other. By Lemma
\ref{morse}(ii), uniformity is automatic if we can show that $c,c'$
are parallel, i.e. contained in each other's metric neighborhoods. 
Consider the geodesics $[c(0),c'(t)]$, $t \geq 0$. 
Let $\tilde c$ be the limiting ray of a subsequence.
By Lemma~\ref{morse}(i) the $[c(0),c'(t)]$ and $\tilde c$ stay in a bounded neighborhood of $c'$ (where the bound depends on $d(c(0),c'(0))$) and in particular in some thick part.
Thus Lemma~\ref{above} yields $\tilde c(+\infty) = c'(+\infty) = c(+\infty)$.
By the Masur criterion (Proposition~\ref{basic facts}) the vertical foliations of $c$ and $\tilde c$ are $\tilde c(+\infty)$ and $c(\infty)$ and therefore coincide. 
Of course $\tilde c(0) = c(0)$.
Thus the restriction of $c$ to $[0,+\infty)$ coincides with $\tilde c$ and stays in a bounded neighborhood of $c'$.
Similarly, the restriction of $c$ to $(-\infty,0]$ stays in a bounded neighborhood of $c'$.
\end{proof}

\begin{remark} The Gardiner-Masur theorem \cite{gardiner-masur} states
  that if $\xi_\pm$ are two measured foliations such that
  $i(\xi_+,\mu)+i(\xi_-,\mu)>0$ for every $\mu\in\MF$, then there is a
  unique Teichm\"uller geodesic with vertical and horizontal
  foliations $\xi_\pm$. Conversely, any pair $\xi_\pm$ of a horizontal
  and vertical foliation satisfies this condition.
\end{remark}

\begin{remark}
We stated the axioms with an eye towards applying them to $Out(F_n)$. At
present we do not know how to make them all work, but we point out the
following. For $R>>0$ consider biinfinite folding paths that are
contained in the $\frac 1R$-thick part of Culler-Vogtmann's Outer
space and whose projections to the complex of free factors are
$(R,R)$-quasi-geodesics. The main ingredients needed for the flow
axioms are all known in the $Out(F_n)$ context: these lines are
contracting \cite{bf}, the Masur Criterion holds \cite{npr}, and the
analog of Klarreich's theorem is in \cite{br,h}. The projection axioms
are more subtle, but see \cite{bf2}.
\end{remark}

\section{Teichm\"uller geodesics intersecting the thin part}
    \label{sec:teich-geod-in-thin}

    \subsection{Thick or thin}
Suppose $X\in\calt_{\geq\epsilon}$ and $\xi\in\PMF$. Section
\ref{s:2} was concerned with Teichm\"uller rays completely contained
in the thick part. Here we want to establish that if there are short
curves along the ray, then there is a (proper) subsurface where the
projection distance between the vertical foliation and a suitable
point thought of as the projection of the initial point of the ray, is
large. The next theorem is a variant of a theorem of Rafi
\cite{rafi-short}.

\begin{thm} \label{large-proj}
For each collection $\bfY^i$ (see Section \ref{color}) fix a basepoint
$X^i\in\bfY^i$, and also fix $X_0\in\calt$. Then for every $\Theta>0$
there is a compact set $K\subset \calt$ such that for any $(g,\xi)\in
G\times\PMF$ one of the following holds:
\begin{enumerate}[(a)]
\item The Teichm\"uller ray $c$ starting in $gX_0$ with vertical foliation $\xi$ is contained in
  $G\cdot K$, or
\item there is some $\bfY^i$ and $Z\in\bfY^i$ so that
  $d^\pi_Z(gX^i,\xi)>\Theta$. 
\end{enumerate}
\end{thm}

In order to prove this theorem we will need the following statement.

\begin{prop}\label{short implies big intersection}
  Fix a Riemann surface $X$ and a finite filling collection of curves
  $\beta_i$ in $X$. For every $k$ there exists $\epsilon>0$ so that
  the following holds.  Suppose $g$ is a Teichm\"uller geodesic with
  $g_0=X$ and with vertical measured foliation $\xi^+$ at $X$. Assume
  that $\xi^+$ is filling. Suppose there is a curve $\alpha$ of
  hyperbolic length $\epsilon$ at $g_T$ for some $T>0$, and let
  $\beta=\beta_i$ be one of the curves that intersects $\alpha$.  Then
  there is a component $Y$ of $X$ cut open along $\alpha$, so that after
  putting $Y$, $\beta$ and the leaves of $\xi^+$ in minimal position,
  there are arcs
  $\beta',\ell$ of intersections of $\beta$ and $\xi^+$ with $Y$ so
  that the intersection number $i_Y(\beta',\ell)>k$.
\end{prop}

\begin{proof}[Proof of Theorem \ref{large-proj} assuming Proposition
    \ref{short implies big intersection}]
If $\xi$ is not filling we can take $Y$ to be a component of
$Supp(\xi)$ and then (b) holds since $d^\pi_Y(gX^i,\xi)=\infty$. So
assume $\xi$ is filling. 
By equivariance, we can also assume
$g=1$. Fix a finite collection of curves $\beta_1,\cdots,\beta_m$ that
fill $\Sigma$. Note that there is a bound on the intersection numbers
between any $\beta_j$ and any $\partial X^i$. This means that when
$d^\pi_Y(\beta_j,X^i)$ is defined, it is uniformly bounded. 

Now suppose that there is a curve $s$ that has length $<\epsilon$
somewhere along $c$. 
Choose $\beta=\beta_j$ that intersects
$s$. By Proposition \ref{short implies big intersection} there is a
subsurface $Y$ of $\Sigma$ and arcs $\beta',\ell$ of intersection
between $\beta,\xi$ with $Y$ whose intersection number is as large as
we want, if $\epsilon$ is small enough.  

By
Lemma~\ref{subsurface} there is a further subsurface $Z$ of $Y$ such
that $d^\pi_Z(\beta,\xi)$ is large. Thus $Z$ satisfies the conclusion,
since replacing $\beta$ by $X^i$ (where $Z\in\bfY^i$) changes
projection distance by a bounded amount.

If no curve gets $\epsilon$-short along $c$, then (a) holds with a
suitable $K$ (see Theorem \ref{mumford}).
\end{proof}

\subsection{Outline for Proposition \ref{short implies big
    intersection}}
  We will follow Rafi's strategy in \cite{rafi-short} where he proved
similar statements. Here is the basic idea. Since $\alpha$ is short in
$g_T$ it will have, by the Collar Lemma \ref{lem:collar}, an
annulus neighborhood where the distance between the boundary
components is much larger than the length of $\alpha$. Thus the arcs
of intersection between $\beta$ and $X\smallsetminus\alpha$ must be
long, as they have to go across the annulus. On the other hand, if the
intersection numbers are bounded, we will argue that the horizontal
lengths of these arcs with respect to the quadratic differential $q_T$
are too small. Since the vertical lengths are small as well, given
that they are bounded at time 0, this will give a contradiction.

Technically, we will work with the metric $l_{q_T}$ induced by $q_T$
and need to review the work of Minsky \cite{minsky-harmonic} on the
Collar Lemma in this setting.

\subsection{Primitive annuli}

Let $X$ be a Riemann surface and $q$ a quadratic differential of area
1 on $X$. Thus $X$ is equipped with the associated flat metric $l_q$ (with
cone type singularities). The following concept was introduced by
Minsky \cite[Section 4.3]{minsky-harmonic}. Consider a closed annulus
$B\subset X$ with piecewise smooth boundary curves $\partial_0$ and
$\partial_1$.  We say that $B$ is {\it regular} if
\begin{itemize}
\item $\partial_0$ and $\partial_1$ are equidistant from each other,
  \item either $\partial_0$ is a geodesic, or else the curvature
    vector of $\partial_0$ points away from $B$ (and the internal
    angles are $\geq\pi$), and likewise either $\partial_1$ is a
    geodesic or the curvature vector of $\partial_1$
    points into $B$ (and the internal angles are $\leq\pi$).
\end{itemize}

Denote by $\kappa(\partial_i)$ the total signed curvature 
of $\partial_i$. Thus
$\kappa(\partial_0)\leq 0$ and $\kappa(\partial_1)\geq 0$. These
numbers can also be thought of as the total turning number 
of $\partial_i$ with respect to the horizontal (or
vertical) foliation.  We call the annulus $B$ {\it primitive} if it is
regular and contains no singular points 
in its interior. In that case
$\kappa(\partial_1)=-\kappa(\partial_0)$ is a nonnegative integral
multiple of $\pi$, and we denote this quantity by $\kappa(B)$. When
$\kappa(B)=0$ then $B$ is a flat cylinder.

Other than the flat cylinder (when $\kappa=0$) an example of a
primitive annulus (with $\kappa=2\pi$) is 
the region
in the plane between two concentric circles (here $\partial_0$ is the
smaller circle). The
following is a Theorem of Minsky \cite[Theorem 4.6]{minsky-harmonic}.

\begin{prop}
  For any topological surface $\Sigma$ there are constants   
  $m_0,c_3,c_4>0$ so that for any Riemann surface $X$ homeomorphic to
  $\Sigma$ and any homotopically nontrivial annulus $A\subset X$ with
  $Mod(A)\geq m_0$ there exists a primitive annulus $B\subset A$
  homotopic to $A$ and with $Mod(B)\geq c_3 Mod(A)-c_4$.
\end{prop}

\begin{remark}
  In the statement of \cite[Theorem 4.6]{minsky-harmonic} there is a
  typo; the inequality in $Mod(A)\geq m_0$ is reversed.

  Further, \cite[Theorem 4.6]{minsky-harmonic} finds only a regular
  annulus $B$. One can then pass to a primitive subannulus by taking a
  parametrization by curves equidistant from $\partial_0$
  and passing to a subannulus bounded by two such curves that avoids the
  singular points. Since the number of singular points is bounded by
  the topology of $X$, we can arrange that the modulus still satisfies
  the above estimate. 
  See \cite[Theorem 4.5]{minsky-harmonic} and the
  paragraph before it.
\end{remark}

We will need the following elementary facts about a primitive
annulus $B$. Denote by $d$ the distance between the two boundary
components. 

\begin{prop} \label{primitive}
  Let $B\subset X$ be an essential primitive annulus.
  \begin{enumerate}[(i)]
  \item If $\kappa(B)=0$ then $Mod(B)=\frac d{l_q(\partial_0)}$,  
    \item If $\kappa(B)=n\pi>0$ then $$\kappa(B)Mod(B)\leq
      c_5\log\bigg(\frac d{l_q(\partial_0)}\bigg)+c_6$$
      for constants $c_5,c_6>0$ that depend only on the topology of
      the surface
      $\Sigma$ underlying $X$.
    \item $Area(B)=d l_q(\partial_0)+\frac {\kappa(B)d^2}2$.
  \end{enumerate}
\end{prop}

\begin{proof}
  (i) is the definition of the modulus of a flat cyclinder. (ii) is
  \cite[Lemma 3.6]{rafi-short}. For (iii), we compute the area by
  integrating lengths of equidistant curves. The curve at distance $r$
  from $\partial_0$
  has length $l_q(\partial_0)+\kappa(B)r$
  so
  $$Area(B)=\int_0^d(l_q(\partial_0)+\kappa(B)r)dr=d~
  l_q(\partial_0)+\frac {\kappa(B)d^2}2$$
  \end{proof}

\begin{proof}[Proof of Proposition \ref{short implies big
      intersection}]
  For this proof the reader is invited to review Section \ref{s:mcg},
  and particularly subsections \ref{subsec:quadratic-diff} and \ref{collar}
  We will be working with singular Euclidean metrics $l_{q_t}$
  associated to quadratic differentials $q_t$ along $g_t$.  Let
  $h_t(\alpha)$ be the horizontal length of $\alpha$ at time $t$. 
  By the Collar Lemma, $X$ contains an annulus with core curve $\alpha$ 
  and with arbitrarily large modulus, if $\epsilon$ is sufficiently
  small. By the above Proposition, $(g_T,l_{q_T})$ contains a primitive
  annulus $B$ with arbitrarily large modulus and core curve homotopic
  to $\alpha$. Instead of cutting along $\alpha$ we will cut along
  $\partial_0$ (which is homotopic to $\alpha$). 
  If $\partial_0$ is separating, we let $Y$ be the component of
  $X\smallsetminus \partial_0$ that contains $B$. If $\partial_0$ is
  nonseparating, we let $Y=X\smallsetminus\partial_0$ and note that the
  core curve of $B$ is homotopic to one of the two copies of
  $\partial_0$ in $Y$.

  Let $d$ be the distance between $\partial_0$ and $\partial_1$.  Then
  $\frac d{l_{q_T}(\partial_0)}$ is arbitrarily large, by Proposition
  \ref{primitive} (i) and (ii). Since the area of $B$ is $\leq 1$
  (because the area of $X$ is 1) it follows from Proposition
  \ref{primitive} (iii) that $l_{q_T}(\partial_0)$ is arbitrarily
  small (can be made arbitrarily close to 0 if $\epsilon$ is
  sufficiently small). Thus $h_T(\alpha)\leq l_{q_T}(\alpha)\leq
  l_{q_T}(\partial_0)$ is also arbitrarily small. Since the horizontal
  length is growing, we have that $h_t(\alpha)$ is arbitrarily small
  for $t\in [0,T]$. Put $\alpha,\beta,\partial Y$ and leaves of
  $\xi^+$ in minimal position
  with respect to each other. 
  The vertical foliation restricted to $Y$ 
  induces a
  decomposition of $Y$ into strips, each foliated by pairwise isotopic
  arcs of intersection between $Y$ and the leaves of $\xi^+$. See
  Figure \ref{strips}.

\begin{figure}[h]
  \begin{tikzpicture}[scale=0.6]%,y=-1cm]
\sf
\clip (0,0) ellipse (8.03cm and 5.02cm);
\fill [gray,opacity=.02](0,0) ellipse (8cm and 5cm);
\fill [gray,opacity=.3] (0,0) -- (2.5,2) -- (2.5,5) -- (-2.5,5) -- (-2.5,2) -- cycle; 
\draw[thick,red] (0,0) -- (2.5,2) -- (5,0) (0,0) -- (-2.5,2) -- (-5,0) (2.5,2) -- (2.5,4.76) (-2.5,2) -- (-2.5,4.76) (0,0) -- (2.5,-2) -- (5,0) (0,0) -- (-2.5,-2) -- (-5,0) (2.5,-2) -- (2.5,-4.76) (-2.5,-2) -- (-2.5,-4.76);  
\draw[blue] (-8,0) -- (8,0) (0,5) -- (0,-5);
\draw[blue] plot[smooth] 
        coordinates{(.3,0) (2.5,1) (4.7,0)};
\draw[blue] plot[smooth] 
        coordinates{(-.3,0) (-2.5,1) (-4.7,0)};
\draw[blue] plot[smooth] 
        coordinates{(.3,0) (2.5,-1) (4.7,0)};
\draw[blue] plot[smooth] 
        coordinates{(-.3,0) (-2.5,-1) (-4.7,0)};  
\draw[blue] plot[smooth] 
        coordinates{(0,.3) (1.5,2.2) (1.6,5)};  
\draw[blue] plot[smooth] 
        coordinates{(0,-.3) (1.5,-2.2) (1.6,-5)};  
\draw[blue] plot[smooth] 
        coordinates{(0,.3) (-1.5,2.2) (-1.6,5)};  
\draw[blue] plot[smooth] 
        coordinates{(0,-.3) (-1.5,-2.2) (-1.6,-5)}; 
\draw[blue] plot[smooth] 
        coordinates{(5,.5) (3.8,2.2) (3.5,5)};           
\draw[blue] plot[smooth] 
        coordinates{(-5,.5) (-3.8,2.2) (-3.5,5)};           
\draw[blue] plot[smooth] 
        coordinates{(5,-.5) (3.8,-2.2) (3.5,-5)};           
\draw[blue] plot[smooth] 
        coordinates{(-5,-.5) (-3.8,-2.2) (-3.5,-5)};  
\draw[blue] plot[smooth] 
        coordinates{(5,.2) (6,2) (8,5)};                  
\draw[blue] plot[smooth] 
        coordinates{(-5,.2) (-6,2) (-8,5)};                  
\draw[blue] plot[smooth] 
        coordinates{(5,-.2) (6,-2) (8,-5)};                  
\draw[blue] plot[smooth] 
        coordinates{(-5,-.2) (-6,-2) (-8,-5)};                                                                              
\fill[white] (0,0) circle (1cm) (5,0) circle (1cm) (-5,0) circle (1cm);
\draw[black] (0,0) circle (1cm) (5,0) circle (1cm) (-5,0) circle (1cm);
\draw[black] (0,0) ellipse (8cm and 5cm);
\end{tikzpicture}%
  \caption{Singular leaves, drawn in red, decompose the subsurface $Y$
    (disk with three holes)
  into (six) strips, each foliated by arcs of $\xi^+$, some of which are
  drawn in blue; one of the strips is shaded in gray.}\label{strips}
\end{figure}
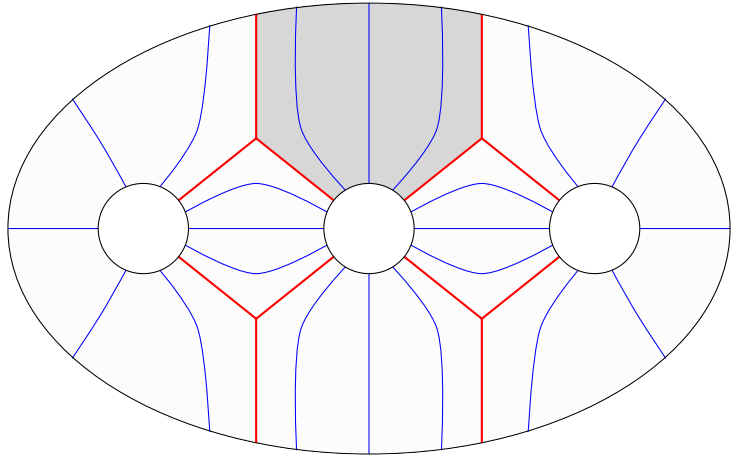

Each such strip has a width, i.e. the transverse measure
  assigned to a cross section by the measure on $\xi^+$. The sum of
  the widths over all the strips is $h_t(\alpha)$ (when $\alpha$ is
  nonseparating) or $\frac 12 h_t(\alpha)$ (when $\alpha$ is
  separating). Assuming the conclusion is false, each arc of $\beta$
  intersects each strip at most $k$ times, so its horizontal length is
  at most $k~ h_t(\alpha)$. The vertical length of $\beta$ at $g_0=X$
  is bounded above by some fixed number, say $L$.
Likewise, the vertical length of $\alpha$ at $g_0=X$
  is bounded {\it below} by some fixed constant $K$ depending only on
  $X$ (see Lemma \ref{bdd below}, this is because
  the horizontal length is very small and the $q_0$-length of $\alpha$
  is bounded below). Since the ratio between the vertical lengths of
  $\beta$ and $\alpha$ does not change over time, it follows that the
  ratio $l_{q_T}(\beta')/l_{q_T}(\alpha)$ is bounded by a function of
  $k$ for any arc $\beta'$. The ratio
  $l_{q_T}(\beta')/l_{q_T}(\partial_0)$ is even smaller, and is thus also
  bounded. It follows that no such arc is long enough to traverse the
  annulus $B$, as soon as $d/l_{q_T}(\partial_0)$ is large enough. Thus $\beta$ 
cannot intersect $\alpha$. Contradiction.
\end{proof}

\section{The Farrell-Jones Conjecture for mapping class groups}
     \label{sec:FJC-for-MCGs}

In this section we prove that mapping class groups satisfy the Farrell-Jones Conjecture.
We fix a closed oriented surface $\Sigma$ of genus $g$ with $p$ punctures and write $\MCG$ for its mapping class group.

We start with a proof of Theorem B, which we restate for convenience.
\vskip 0.3cm

\noindent
{\bf Theorem B.}
{\it 
    Let $\Sigma$ be a closed oriented surface of genus $g$ with $p$ punctures.
    Assume that $6g + 2p -6 > 0$.
    Then the action of its mapping class group $\MCG$ on the
    space $\PMF$ of projective measured foliations on $\Sigma$
    is finitely $\calf$-amenable where $\calf$ is the family of subgroups that are 
    either virtually cyclic or virtually fix a subsurface of $\Sigma$.}

\vskip 0.3cm

Earlier we described $\calf$ as the family of subgroups that virtually fix a curve or are virtually cyclic.
This is the same family, because a subgroup that fixes a subsurface fixes all boundary curves of the subsurface virtually, and if a subgroup fixes a curve, then it also fixes the corresponding annulus.

\begin{proof}[Proof of Theorem~\ref{thm:finitely-amenable-action-MCG}]
By Proposition \ref{prop:finite-index+fin-F-amenabilty} it suffices
to show that the action of the color preserving subgroup $G<\MCG$ (see
Section \ref{color}) is
finitely $\mathcal F$-amenable. 
Let $T=\calt$ be the Teichm\"uller space of $\Sigma$ and let
$\Delta=\PMF$. Thus $\overline T=T\cup\Delta$ is homeomorphic to a
ball and $T$ is its interior. To finish defining the flow data, for a
compact set $K\subset T$ let $\calg_K$ be the set of Teichm\"uller
rays which are contained in $G\cdot K$. By the Masur Criterion (see
Proposition \ref{basic facts})
every such ray converges to a point of
$\Delta$.

We now define the projection data.  For a subsurface $Y\subset\Sigma$
let $\Delta(Y)\subset\Delta$ be the open set of foliations whose
support is not disjoint from $Y$. The collection of subsurfaces which
are not pairs of pants is partitioned into subcollections
$\bfY^1,\cdots,\bfY^k$ each of which is a $G$-orbit and consists of
overlapping subsurfaces (i.e. $Y,Y'\in \bfY^i$, $Y\neq Y'$ implies
$\partial Y\cap\partial Y'\neq\emptyset$). We set $\mathcal
Y=\{\bfY^1,\cdots,\bfY^k\}$. Projections and projection distance
$d^\pi_Y(X,Z)$ was defined in Sections \ref{ss:subsurface-projections}
and \ref{ss:proj-dist}
and $d^\pi_Y(X,\xi)$ for $\xi\in \Delta(Y)$ in Sections
\ref{ss:projecting-laminations} and \ref{ss:proj-dist}. Axiom (P1) is
obvious, (P2) was recorded as Proposition \ref{triangle}, (P3) is
Proposition \ref{behrstock}, (P4) is Proposition \ref{finiteness} and
(P5) is Proposition \ref{continuity}. This finishes the projection
axioms.

Axiom \ref{axiom:large-projection} holds by Theorem \ref{large-proj} and 
the flow axioms were proved in Section \ref{s:2}.
\end{proof}

  \begin{corollary}
    \label{cor:FJC-MCG-ref-calf}
    Assume $6g + 2p -6 > 0$.
    Then the action of the  Mapping class group $\MCG$
    on the Thurston compactification $\overline{\calt} = \calt \cup \PMF$ 
    of Teichm\"uller space
    is finitely $\calf$-amenable where $\calf$ is the family of subgroups that are either
    virtually cyclic or stabilize a subsurface.  
  \end{corollary}

  \begin{proof}
    By Theorem~\ref{thm:finitely-amenable-action-MCG} the action of $\MCG$ on
    $\PMF$ is $N$-$\calf$-amenable for some $N$.
    Let $N' := \dim \calt$.
    We will show that the action of $\MCG$ on $\overline{\calt}$
    is $N+N'+1$-$\calf$-amenable.
    
    Let $S \subseteq \MCG$ be finite.
    Then there exists an  open  $\calf$-cover of the product 
    $\MCG \x \PMF$ of order at most $N$ that is $S$-long in the
    group coordinate.
    It is not difficult to extend this cover to a $\MCG$-invariant collection  $\calu$ 
    of open $\calf$-subsets of $\MCG \x \overline{\calt}$ of order
    at most $N$ satisfying
    \begin{equation*}
      \forall (g,\xi) \in \MCG \x \PMF \; \exists \; U \in \calu \;
        \text{with} \; gS \x \{ \xi \} \subseteq U,
    \end{equation*}
    see Lemma~\ref{lem:extension-of-covers}.  
    Teichm\"uller space $\calt$ carries the structure of a proper $\MCG$-$CW$-complex.
    (Even simpler, we could restrict here to a finite index subgroup of $\MCG$ that acts freely.)   
    We therefore find an open $\Fin$-cover $\calv$ of $\calt$ of dimension $N$, where $\Fin$ denotes the family of finite subgroups.
    Now the cover of $\MCG \x \PMF$ consisting of all $U \in \calu$
    and all sets of the form $\MCG \x V$ with $V \in \calv$ is an 
    open $\calf$ cover, is  of order at most $N + N' + 1$, and 
    is $S$-long in the group coordinate. 
  \end{proof}

  We now write $c(\Sigma) := 6g + 2p -6$.
  For a subsurface $Y$ of $\Sigma$ we write $Y'$ for the closed surface obtained by collapsing
  the boundary curves of $Y$ to punctures.
  Note that $c(Y') < c(\Sigma)$. 

  \begin{lemma}
    \label{lem:stab-subsurface}
    Let $G$ be a subgroup of $\MCG$ that fixes a subsurface $Y$ of $\Sigma$. 
    Then there exists a finite index subgroup $G_0$ of $G$ 
    and a central extension
    \begin{equation*}
      \IZ^k \xrightarrow{i} G_0 \xrightarrow{p} Q 
    \end{equation*}
    where $Q$ is a subgroup of a product of mapping class groups of
    closed surfaces $\hat Y_i$ with punctures for which $c(\hat Y_i) <
    c(\Sigma)$.
  \end{lemma}

  \begin{proof}
    Let $Y_0 := Y$ and let $Y_1,\dots,Y_n$ be the components of the
    complement. For each $Y_i$ that's not an annulus let $\hat Y_i$ be
    the surface obtained from $Y_i$ by collapsing the boundary
    components to punctures.  Then $c(\hat Y_i) < c(\Sigma)$.  Since
    $G$ permutes the $Y_i$'s and their boundary components, there is a
    finite index subgroup $G_0$ of $G$ that preserves each $Y_i$ and
    each boundary component of $Y_i$.  Restriction to each non-annular
    $Y_i$ followed by capping off the boundary components induces a
    group homomorphism~\cite[Proposition 3.19]{FM}
    \begin{equation*}
      G_0 \xrightarrow{p} \MCG_0 \x \dots \x \MCG_n.
    \end{equation*}
    whose kernel is the free abelian group 
    generated by Dehn twists around boundary curves of $Y$ and is
    central in $G_0$.
  \end{proof}

  We will now use some of the classes of groups introduced in
  Section~\ref{subsec:FJC}.

  \begin{lemma}
    \label{lem:MCG-in-AC}
    The mapping class group $\MCG$ of $\Sigma$ belongs to the class of groups $\EAC(\EVNil)$.
  \end{lemma}

  \begin{proof}
    We will proceed by induction on $c(\Sigma)$. 
    In the sporadic case $c(\Sigma) = 0$ we have either $g=0$, $p \leq 3$ 
    or $g = 1$, $p=0$.
    If $g=0$, $p \leq 3$, then $\MCG$ is finite and belongs to $\EAC(\EVNil)$.
    If $g=1$, $p=0$, then $\MCG = \SL_2(\IZ)$ acts cocompactly on a
    locally finite tree $T$.
    The geodesic ray compactification $\overline{T}$ of the tree 
    is a  compact $\ER$ and the action of $\SL_2(\IZ)$
    on $\overline{T}$ is finitely $\VCyc$-amenable by~\cite{Bartels-Lueck-Reich-cover}.
    Alternatively, we could use the action of $\SL_2(\IZ)$ on the hyperbolic plane $\IH^2$ 
    as discussed in Example~\ref{ex:SL_2}.   
    Thus $\MCG \in \EAC(\EVNil)$.
    
    Suppose now $c(\Sigma) > 0$.
    The Thurston compactification $\overline{\calt(\Sigma)}$ of Teichm\"uller space
    is homeomorphic to a closed Euclidean ball  
    and in particular a compact $\ER$.
    Virtually cyclic groups belong to $\EVNil$ and therefore also to $\EAC(\EVNil)$.
    Lemma~\ref{lem:ac(F)}~\ref{lem:ac(F):over} implies that $\EAC(\EVNil)$
    is closed under finite products, central extensions with finitely generated kernel,
    taking subgroups and with 
    taking overgroups of finite index.    
    Therefore the induction hypothesis and Lemma~\ref{lem:stab-subsurface} imply that all 
    stabilizers of subsurfaces of $\Sigma$  belong to $\EAC(\EVNil)$.
    Now Corollary~\ref{cor:FJC-MCG-ref-calf} allows us to use the defining property
    of the operation $\Eac$.
    Therefore $\MCG$ belongs to $\EAC(\EVNil)$. 
  \end{proof}

  \begin{proof}[Proof of Theorem~\ref{thm:FJC-for-MCG}]
    Corollary~\ref{cor:AC} and Proposition~\ref{prop:FJC-for-VNil}
    imply that all groups in $\EAC(\EVNil)$ satisfy the Farrell-Jones Conjecture.
    Since mapping class groups of surfaces belong to this class by
    Lemma~\ref{lem:MCG-in-AC} they satisfy the Farrell-Jones Conjecture.
  \end{proof}
 
  \begin{remark}
    The \emph{Farrell-Jones Conjecture with wreath products}
      ~\cite[Sec.~6]{Bartels-Lueck-Reich-Rueping-GLnZ} is a strengthening of the
    Farrell-Jones Conjecture that has the advantage that it is closed under 
    taking finite index overgroups.
    Since the class $\EAC(\EVNil)$ is closed under taking finite overgroups
    and under finite products all groups in $\EAC(\EVNil)$ satisfy the 
    Farrell-Jones Conjecture with wreath products.
    In particular, mapping class groups of surfaces satisfy the
    Farrell-Jones Conjecture with wreath products.
    This is also true for all groups in the class $\EAC(\EVSol)$. 
   \end{remark}

  \bibliographystyle{abbrv}
  \bibliography{mcg-fjc2.bib}

\begin{thebibliography}{10}

\bibitem{ahlfors}
L.~V. Ahlfors.
\newblock {\em Lectures on quasiconformal mappings}, volume~38 of {\em
  University Lecture Series}.
\newblock American Mathematical Society, Providence, RI, second edition, 2006.
\newblock With supplemental chapters by C. J. Earle, I. Kra, M. Shishikura and
  J. H. Hubbard.

\bibitem{Aou}
T.~Aougab.
\newblock Uniform hyperbolicity of the graphs of curves.
\newblock {\em Geom. Topol.}, 17(5):2855--2875, 2013.

\bibitem{arzhantseva}
G.~N. Arzhantseva, C.~H. Cashen, D.~Gruber, and D.~Hume.
\newblock Characterizations of {M}orse quasi-geodesics via superlinear
  divergence and sublinear contraction.
\newblock {\em Doc. Math.}, 22:1193--1224, 2017.

\bibitem{Bartels-Coarse-flow}
A.~Bartels.
\newblock Coarse flow spaces for relatively hyperbolic groups.
\newblock {\em Compos. Math.}, 153(4):745--779, 2017.

\bibitem{Bartels-Farrell-Lueck-cocompact-lattices}
A.~Bartels, F.~T. Farrell, and W.~L{\"u}ck.
\newblock The {F}arrell-{J}ones conjecture for cocompact lattices in virtually
  connected {L}ie groups.
\newblock {\em J. Amer. Math. Soc.}, 27(2):339--388, 2014.

\bibitem{Bartels-Lueck-Coeff-L}
A.~Bartels and W.~L\"uck.
\newblock On twisted group rings with twisted involutions, their module
  categories and ${L}$-theory.
\newblock In {\em Cohomology of groups and algebraic $K$-theory}, volume~12 of
  {\em Advanced Lectures in Mathematics}, pages 1--55, Somerville, U.S.A.,
  2009. International Press.

\bibitem{Bartels-Lueck-Borel}
A.~Bartels and W.~L{\"u}ck.
\newblock The {B}orel conjecture for hyperbolic and {CAT(0)}-groups.
\newblock {\em Ann. of Math. (2)}, 175:631--689, 2012.

\bibitem{Bartels-Lueck-Reich-cover}
A.~Bartels, W.~L{\"u}ck, and H.~Reich.
\newblock Equivariant covers for hyperbolic groups.
\newblock {\em Geom. Topol.}, 12(3):1799--1882, 2008.

\bibitem{Bartels-Lueck-Reich-K-FJ-hyp}
A.~Bartels, W.~L{\"u}ck, and H.~Reich.
\newblock The {$K$}-theoretic {F}arrell-{J}ones conjecture for hyperbolic
  groups.
\newblock {\em Invent. Math.}, 172(1):29--70, 2008.

\bibitem{Bartels-Lueck-Reich-FJ+appl}
A.~Bartels, W.~L{\"u}ck, and H.~Reich.
\newblock On the {F}arrell-{J}ones conjecture and its applications.
\newblock {\em J. Topol.}, 1(1):57--86, 2008.

\bibitem{Bartels-Lueck-Reich-Rueping-GLnZ}
A.~Bartels, W.~L{\"u}ck, H.~Reich, and H.~R{\"u}ping.
\newblock K- and {L}-theory of group rings over {$GL_n({\bf Z})$}.
\newblock {\em Publ. Math. Inst. Hautes \'Etudes Sci.}, 119:97--125, 2014.

\bibitem{Bartels-Reich-coeff-for-FJ}
A.~Bartels and H.~Reich.
\newblock Coefficients for the {F}arrell-{J}ones {C}onjecture.
\newblock {\em Adv. Math.}, 209(1):337--362, 2007.

\bibitem{jason}
J.~A. Behrstock.
\newblock Asymptotic geometry of the mapping class group and {T}eichm\"uller
  space.
\newblock {\em Geom. Topol.}, 10:1523--1578, 2006.

\bibitem{bbf}
M.~Bestvina, K.~Bromberg, and K.~Fujiwara.
\newblock Constructing group actions on quasi-trees and applications to mapping
  class groups.
\newblock {\em Publ. Math. Inst. Hautes \'Etudes Sci.}, 122:1--64, 2015.

\bibitem{bbfs}
M.~Bestvina, K.~Bromberg, K.~Fujiwara, and A.~Sisto.
\newblock Acylindrical actions on projection complexes.
\newblock arXiv:1711.08722.

\bibitem{bf}
M.~Bestvina and M.~Feighn.
\newblock Hyperbolicity of the complex of free factors.
\newblock {\em Adv. Math.}, 256:104--155, 2014.

\bibitem{bf2}
M.~Bestvina and M.~Feighn.
\newblock Subfactor projections.
\newblock {\em J. Topol.}, 7(3):771--804, 2014.

\bibitem{br}
M.~Bestvina and P.~Reynolds.
\newblock The boundary of the complex of free factors.
\newblock {\em Duke Math. J.}, 164(11):2213--2251, 2015.

\bibitem{BLM}
J.~S. Birman, A.~Lubotzky, and J.~McCarthy.
\newblock Abelian and solvable subgroups of the mapping class groups.
\newblock {\em Duke Math. J.}, 50(4):1107--1120, 1983.

\bibitem{bonahon2}
F.~Bonahon.
\newblock Geodesic laminations on surfaces.
\newblock In {\em Laminations and foliations in dynamics, geometry and topology
  ({S}tony {B}rook, {NY}, 1998)}, volume 269 of {\em Contemp. Math.}, pages
  1--37. Amer. Math. Soc., Providence, RI, 2001.

\bibitem{Bow2}
B.~H. Bowditch.
\newblock Intersection numbers and the hyperbolicity of the curve complex.
\newblock {\em J. Reine Angew. Math.}, 598:105--129, 2006.

\bibitem{bowditch}
B.~H. Bowditch.
\newblock Relatively hyperbolic groups.
\newblock {\em Internat. J. Algebra Comput.}, 22(3):1250016, 66, 2012.

\bibitem{Bow}
B.~H. Bowditch.
\newblock Uniform hyperbolicity of the curve graphs.
\newblock {\em Pacific J. Math.}, 269(2):269--280, 2014.

\bibitem{casson-bleiler}
A.~J. Casson and S.~A. Bleiler.
\newblock {\em Automorphisms of surfaces after {N}ielsen and {T}hurston},
  volume~9 of {\em London Mathematical Society Student Texts}.
\newblock Cambridge University Press, Cambridge, 1988.

\bibitem{CRS}
M.~Clay, K.~Rafi, and S.~Schleimer.
\newblock Uniform hyperbolicity of the curve graph via surgery sequences.
\newblock {\em Algebr. Geom. Topol.}, 14(6):3325--3344, 2014.

\bibitem{Eskin-Mirzakhani-Rafi-Counting-in-Orbits}
A.~Eskin, M.~Mirzakhani, and K.~Rafi.
\newblock Counting closed geodesics in strata.
\newblock arXiv:1206.5574.

\bibitem{farb}
B.~Farb.
\newblock Relatively hyperbolic groups.
\newblock {\em Geom. Funct. Anal.}, 8(5):810--840, 1998.

\bibitem{FM}
B.~Farb and D.~Margalit.
\newblock {\em A primer on mapping class groups}, volume~49 of {\em Princeton
  Mathematical Series}.
\newblock Princeton University Press, Princeton, NJ, 2012.

\bibitem{Farrell-Jones-K-th-dynamics-I}
F.~T. Farrell and L.~E. Jones.
\newblock {$K$}-theory and dynamics. {I}.
\newblock {\em Ann. of Math. (2)}, 124(3):531--569, 1986.

\bibitem{Farrell-Jones-Isom-Conj}
F.~T. Farrell and L.~E. Jones.
\newblock Isomorphism conjectures in algebraic ${K}$-theory.
\newblock {\em J. Amer. Math. Soc.}, 6(2):249--297, 1993.

\bibitem{FLP}
A.~Fathi, F.~Laudenbach, and V.~Po\'enaru.
\newblock {\em Travaux de {T}hurston sur les surfaces}.
\newblock Soci\'et\'e Math\'ematique de France, Paris, 1991.
\newblock S{\'e}minaire Orsay, Reprint of {{\i}t Travaux de Thurston sur les
  surfaces}, Soc. Math. France, Paris, 1979 [ MR0568308 (82m:57003)],
  Ast{\'e}risque No. 66-67 (1991).

\bibitem{gardiner-masur}
F.~P. Gardiner and H.~Masur.
\newblock Extremal length geometry of {T}eichm\"uller space.
\newblock {\em Complex Variables Theory Appl.}, 16(2-3):209--237, 1991.

\bibitem{gromov}
M.~Gromov.
\newblock Hyperbolic groups.
\newblock In {\em Essays in group theory}, volume~8 of {\em Math. Sci. Res.
  Inst. Publ.}, pages 75--263. Springer, New York, 1987.

\bibitem{groves-manning}
D.~Groves and J.~F. Manning.
\newblock Dehn filling in relatively hyperbolic groups.
\newblock {\em Israel J. Math.}, 168:317--429, 2008.

\bibitem{Guentner-Willett-Yu-dyn-asy-dim}
E.~Guentner, R.~Willett, and G.~Yu.
\newblock Dynamic asymptotic dimension: relation to dynamics, topology, coarse
  geometry, and {$C^*$}-algebras.
\newblock {\em Math. Ann.}, 367(1-2):785--829, 2017.

\bibitem{h}
U.~Hamenst\"adt.
\newblock The boundary of the free splitting graph and the free factor graph.
\newblock arXiv:1211.1630.

\bibitem{Hamenstaedt-boundary-amenable}
U.~Hamenst{\"a}dt.
\newblock Geometry of the mapping class groups. {I}. {B}oundary amenability.
\newblock {\em Invent. Math.}, 175(3):545--609, 2009.

\bibitem{hempel}
J.~Hempel.
\newblock 3-manifolds as viewed from the curve complex.
\newblock {\em Topology}, 40(3):631--657, 2001.

\bibitem{HPW}
S.~Hensel, P.~Przytycki, and R.~C.~H. Webb.
\newblock 1-slim triangles and uniform hyperbolicity for arc graphs and curve
  graphs.
\newblock {\em J. Eur. Math. Soc. (JEMS)}, 17(4):755--762, 2015.

\bibitem{Higson-bivariant+Novikov}
N.~Higson.
\newblock Bivariant {$K$}-theory and the {N}ovikov conjecture.
\newblock {\em Geom. Funct. Anal.}, 10(3):563--581, 2000.

\bibitem{hubbard-book}
J.~H. Hubbard.
\newblock {\em Teichm\"uller theory and applications to geometry, topology, and
  dynamics. {V}ol. 1}.
\newblock Matrix Editions, Ithaca, NY, 2006.
\newblock Teichm{\"u}ller theory, With contributions by Adrien Douady, William
  Dunbar, Roland Roeder, Sylvain Bonnot, David Brown, Allen Hatcher, Chris
  Hruska and Sudeb Mitra, With forewords by William Thurston and Clifford
  Earle.

\bibitem{it}
Y.~Imayoshi and M.~Taniguchi.
\newblock {\em An introduction to {T}eichm\"uller spaces}.
\newblock Springer-Verlag, Tokyo, 1992.
\newblock Translated and revised from the Japanese by the authors.

\bibitem{ivanov-tits}
N.~V. Ivanov.
\newblock {\em Subgroups of {T}eichm\"uller modular groups}, volume 115 of {\em
  Translations of Mathematical Monographs}.
\newblock American Mathematical Society, Providence, RI, 1992.
\newblock Translated from the Russian by E. J. F. Primrose and revised by the
  author.

\bibitem{ivanov}
N.~V. Ivanov.
\newblock Isometries of {T}eichm\"uller spaces from the point of view of
  {M}ostow rigidity.
\newblock In {\em Topology, ergodic theory, real algebraic geometry}, volume
  202 of {\em Amer. Math. Soc. Transl. Ser. 2}, pages 131--149. Amer. Math.
  Soc., Providence, RI, 2001.

\bibitem{Kammeyer-Lueck-Rueping-lattices}
H.~Kammeyer, W.~L\"{u}ck, and H.~R\"{u}ping.
\newblock The {F}arrell-{J}ones conjecture for arbitrary lattices in virtually
  connected {L}ie groups.
\newblock {\em Geom. Topol.}, 20(3):1275--1287, 2016.

\bibitem{misha-book}
M.~Kapovich.
\newblock {\em Hyperbolic manifolds and discrete groups}.
\newblock Modern Birkh\"auser Classics. Birkh\"auser Boston, Inc., Boston, MA,
  2009.
\newblock Reprint of the 2001 edition.

\bibitem{Kasprowski-On-FDC}
D.~Kasprowski.
\newblock On the {$K$}-theory of groups with finite decomposition complexity.
\newblock {\em Proc. Lond. Math. Soc. (3)}, 110(3):565--592, 2015.

\bibitem{katok}
A.~Katok and B.~Hasselblatt.
\newblock {\em Introduction to the modern theory of dynamical systems},
  volume~54 of {\em Encyclopedia of Mathematics and its Applications}.
\newblock Cambridge University Press, Cambridge, 1995.
\newblock With a supplementary chapter by Katok and Leonardo Mendoza.

\bibitem{Kida-MCG-measure-equiv-Memoirs}
Y.~Kida.
\newblock The mapping class group from the viewpoint of measure equivalence
  theory.
\newblock {\em Mem. Amer. Math. Soc.}, 196(916):viii+190, 2008.

\bibitem{klarreich}
E.~Klarreich.
\newblock The boundary at infinity of the complex of curves and the relative
  {T}eichm\"uller space.
\newblock arXiv:1803.10339, 1999.

\bibitem{arc-curve}
M.~Korkmaz and A.~Papadopoulos.
\newblock On the arc and curve complex of a surface.
\newblock {\em Math. Proc. Cambridge Philos. Soc.}, 148(3):473--483, 2010.

\bibitem{lenzhen-masur}
A.~Lenzhen and H.~Masur.
\newblock Criteria for the divergence of pairs of {T}eichm\"uller geodesics.
\newblock {\em Geom. Dedicata}, 144:191--210, 2010.

\bibitem{levitt}
G.~Levitt.
\newblock Foliations and laminations on hyperbolic surfaces.
\newblock {\em Topology}, 22(2):119--135, 1983.

\bibitem{Lueck-ICM2010}
W.~L{\"u}ck.
\newblock {$K$}- and {$L$}-theory of group rings.
\newblock In {\em Proceedings of the {I}nternational {C}ongress of
  {M}athematicians. {V}olume {II}}, pages 1071--1098. Hindustan Book Agency,
  New Delhi, 2010.

\bibitem{Lueck-Reich-BC-FJ-survey}
W.~L{\"u}ck and H.~Reich.
\newblock The {B}aum-{C}onnes and the {F}arrell-{J}ones conjectures in {$K$}-
  and {$L$}-theory.
\newblock In {\em Handbook of {$K$}-theory. {V}ol. 1, 2}, pages 703--842.
  Springer, Berlin, 2005.

\bibitem{maher-tiozzo}
J.~Maher and G.~Tiozzo.
\newblock Random walks on weakly hyperbolic groups.
\newblock {\em J. Reine Angew. Math.}, 742:187--239, 2018.

\bibitem{johanna2}
J.~Mangahas.
\newblock A recipe for short-word pseudo-{A}nosovs.
\newblock {\em Amer. J. Math.}, 135(4):1087--1116, 2013.

\bibitem{masur-thesis}
H.~Masur.
\newblock On a class of geodesics in {T}eichm\"uller space.
\newblock {\em Ann. of Math. (2)}, 102(2):205--221, 1975.

\bibitem{masur_criterion}
H.~Masur.
\newblock Hausdorff dimension of the set of nonergodic foliations of a
  quadratic differential.
\newblock {\em Duke Math. J.}, 66(3):387--442, 1992.

\bibitem{MM}
H.~A. Masur and Y.~N. Minsky.
\newblock Geometry of the complex of curves. {I}. {H}yperbolicity.
\newblock {\em Invent. Math.}, 138(1):103--149, 1999.

\bibitem{mm2}
H.~A. Masur and Y.~N. Minsky.
\newblock Geometry of the complex of curves. {II}. {H}ierarchical structure.
\newblock {\em Geom. Funct. Anal.}, 10(4):902--974, 2000.

\bibitem{minsky-harmonic}
Y.~N. Minsky.
\newblock Harmonic maps, length, and energy in {T}eichm\"uller space.
\newblock {\em J. Differential Geom.}, 35(1):151--217, 1992.

\bibitem{minsky-contracting}
Y.~N. Minsky.
\newblock Quasi-projections in {T}eichm\"uller space.
\newblock {\em J. Reine Angew. Math.}, 473:121--136, 1996.

\bibitem{mumford}
D.~Mumford.
\newblock A remark on {M}ahler's compactness theorem.
\newblock {\em Proc. Amer. Math. Soc.}, 28:289--294, 1971.

\bibitem{npr}
H.~Namazi, A.~Pettet, and P.~Reynolds.
\newblock Ergodic decompositions for folding and unfolding paths in {O}uter
  space.
\newblock arXiv:1410.8870.

\bibitem{rafi-short}
K.~Rafi.
\newblock A characterization of short curves of a {T}eichm\"uller geodesic.
\newblock {\em Geom. Topol.}, 9:179--202, 2005.

\bibitem{Rafi-Combinatorical-model}
K.~Rafi.
\newblock A combinatorial model for the {T}eichm\"uller metric.
\newblock {\em Geom. Funct. Anal.}, 17(3):936--959, 2007.

\bibitem{rafi-hyperbolicity}
K.~Rafi.
\newblock Hyperbolicity in {T}eichm\"uller space.
\newblock {\em Geom. Topol.}, 18(5):3025--3053, 2014.

\bibitem{Ranicki-blue-book}
A.~A. Ranicki.
\newblock {\em Algebraic {$L$}-theory and topological manifolds}, volume 102 of
  {\em Cambridge Tracts in Mathematics}.
\newblock Cambridge University Press, Cambridge, 1992.

\bibitem{rees}
M.~Rees.
\newblock An alternative approach to the ergodic theory of measured foliations
  on surfaces.
\newblock {\em Ergodic Theory Dynamical Systems}, 1(4):461--488 (1982), 1981.

\bibitem{watanabe}
Y.~Watanabe.
\newblock Intersection numbers in the curve complex via subsurface projections.
\newblock {\em J. Topol. Anal.}, 9(3):419--439, 2017.

\bibitem{webb}
R.~C.~H. Webb.
\newblock Uniform bounds for bounded geodesic image theorems.
\newblock {\em J. Reine Angew. Math.}, 709:219--228, 2015.

\bibitem{Wegner-2012CAT0}
C.~Wegner.
\newblock The {$K$}-theoretic {F}arrell-{J}ones conjecture for {CAT}(0)-groups.
\newblock {\em Proc. Amer. Math. Soc.}, 140(3):779--793, 2012.

\bibitem{Wegner-FJ-solv}
C.~Wegner.
\newblock The {F}arrell-{J}ones conjecture for virtually solvable groups.
\newblock {\em J. Topol.}, 8(4):975--1016, 2015.

\end{thebibliography}

\end{document}